\documentclass[11pt]{amsart}

\usepackage{amssymb,amsfonts}
\usepackage{hyperref}
\usepackage{mdwlist}

\usepackage{amssymb,textcomp}
\usepackage{enumerate}

\usepackage[utf8]{inputenc}
\usepackage[T1]{fontenc}

\usepackage[mathscr]{eucal}

\usepackage{hyperref}
 \hypersetup{
     colorlinks=true,
     linktocpage=true,
     linkcolor=red,
     filecolor=blue,
     citecolor = blue,
     urlcolor=cyan,
     }

\usepackage{xcolor}

\usepackage[a4paper, twoside=false, vmargin={2cm,3cm}, includehead]{geometry}

\usepackage{comment}

\newtheorem{lemma}{Lemma}[section]
\newtheorem{theorem}[lemma]{Theorem}
\newtheorem*{theorem*}{Theorem}

\newtheorem{corollary}[lemma]{Corollary}

\newtheorem{proposition}[lemma]{Proposition}
\newtheorem*{proposition*}{Proposition}

\newtheorem{conjecture}[lemma]{Conjecture}
\newtheorem{problem}{Problem}

\newtheorem*{nh}{Nilpotent Heuristic}
\newtheorem*{mproblem}{Main Problem}

\newtheorem*{problem*}{Problem}

\theoremstyle{definition}
\newtheorem*{claim*}{Claim}

\newtheorem{definition}[lemma]{Definition}

\newtheorem{example}[lemma]{Example}

\DeclareMathOperator*{\oE}{\overline{\mathbb{E}}}
\DeclareMathOperator*{\E}{\mathbb{E}}

\newcommand{\C}{{\mathbb C}}

\newcommand{\N}{{\mathbb N}}
\renewcommand{\P}{{\mathbb P}}
\newcommand{\Q}{{\mathbb Q}}
\newcommand{\R}{{\mathbb R}}
\renewcommand{\S}{\mathbb{S}}
\newcommand{\T}{{\mathbb T}}
\newcommand{\Z}{{\mathbb Z}}

\newcommand{\CA}{{\mathcal A}}
\newcommand{\CB}{{\mathcal B}}
\newcommand{\CC}{{\mathcal C}}
\newcommand{\CD}{{\mathcal D}}

\newcommand{\CI}{{\mathcal I}}
\newcommand{\CK}{{\mathcal K}}

\newcommand{\CP}{{\mathcal P}}
\newcommand{\CQ}{{\mathcal Q}}

\newcommand{\CX}{{\mathcal X}}
\newcommand{\CY}{{\mathcal Y}}

\newcommand{\CZ}{{\mathcal Z}}

\newcommand{\FB}{{\mathfrak B}}

\newcommand{\FI}{{\mathfrak I}}

\newcommand{\bn}{{\mathbf{n}}}

\newcommand{\Aut}{{\text{Aut}}}

\newcommand{\veps}{\varepsilon}

\newcommand{\e}{\varepsilon}

\newcommand{\eps}{\epsilon}

\newcommand{\supp}{\textrm{supp}}

\newcommand{\norm}[1]{\left\Vert #1\right\Vert}
\newcommand{\nnorm}[1]{\lvert\!|\!| #1|\!|\!\rvert}

\newcommand{\cube}[1]{{[\![#1]\!]}}

\newcommand{\inv}{^{-1}}

\newcommand{\ab}{{\mathrm{ab}}}
\newcommand{\nab}{{\mathrm{nonab}}}
\newcommand{\nonab}{{\mathrm{nonab}}}

\newcommand{\bi}{{\textbf{i}}}

\DeclareMathOperator{\id}{id}
\DeclareMathOperator{\Id}{id}
\DeclareMathOperator{\codim}{codim}

\newcommand{\Krat}{{\CK_{\text{\rm rat}}}}

\newcommand{\abs}[1]{\mathopen{}\left| #1\mathclose{}\right|}
\newcommand{\bigabs}[1]{\bigl| #1 \bigr|}

\newcommand{\brac}[1]{\mathopen{}\left( #1 \mathclose{}\right)}

\newcommand{\Bigbrac}[1]{\Bigl( #1 \Bigr)}

\newcommand{\sfloor}[1]{{\lfloor #1 \rfloor}}

\newcommand{\rem}[1]{\left \{ #1 \right \}}

\newcommand{\srem}[1]{ \{ #1 \}}

\title[]{Structure of 2-step nilpotent ergodic averages\\
for distinct-degree polynomials}
\author{Andreas Koutsogiannis, Borys Kuca, Wenbo Sun}
\date{}

\address[Andreas Koutsogiannis]{
Department of Mathematics, Aristotle University of Thessaloniki, Thessaloniki 54124, Greece}
\email{akoutsogiannis@math.auth.gr}

\address[Borys Kuca]{Faculty of Mathematics and Computer Science, Jagiellonian University, 30-348 Krak\'ow, Poland}
\email{borys.kuca@uj.edu.pl}

\address[Wenbo Sun]{Department of Mathematics, Virginia Tech, 225 Stanger Street, Blacksburg, VA, 24061, USA}
\email{swenbo@vt.edu}

\thanks{For the first author, the research project is implemented in the framework of H.F.R.I call ``3rd Call for H.F.R.I.’s
Research Projects to Support Faculty Members \& Researchers'' (H.F.R.I. Project Number: 24979).  The second author was supported by the NCN Polonez Bis 3 grant No. 2022/47/P/ST1/00854 (H2020 MSCA GA No. 945339) and the NCN Sonata grant No. 2024/55/D/ST1/00468, and would also like to acknowledge the Foundation of Polish Science (FNP) Start stipend. The third author was supported by the NSF Grant DMS-2247331. For the purpose of Open Access, the authors have applied a CC-BY public copyright license to any Author Accepted Manuscript (AAM) version arising from this submission.}



\begin{document}
\begin{abstract}
    We investigate manifestations of the Nilpotent Heuristic, which posits that recurrence and convergence phenomena known for measure-preserving $\mathbb{Z}^D$-systems extend to nilpotent group actions. Our main results establish seminorm estimates and limiting formulas for multiple ergodic averages arising from actions of 2-step nilpotent groups. In particular, if $T_1,\ldots,T_\ell$ are totally ergodic and generate a 2-step nilpotent group, then
\[
\lim_{N\to\infty}\frac{1}{N}\sum_{n=1}^N T_1^n f_1 \cdots T_\ell^{n^\ell}f_\ell
= \prod_{j=1}^{\ell}\int f_j\,d\mu
\]
in the $L^{2}$ norm for all bounded functions $f_{1},\dots,f_{\ell}$; the same holds for any distinct-degree polynomial iterates. We also obtain popular-common-difference versions of the polynomial Szemer\'edi theorem in the same setting.
In a different direction, our approach allows us to completely resolve the joint ergodicity conjecture for multidimensional polynomials and $\mathbb Z^D$-systems; we also present an example showing that, surprisingly enough, the 2-step nilpotent analog fails. 
We conclude with many open problems concerning joint ergodicity, seminorm estimates, and the structure theory of nilpotent systems.
\end{abstract}

\maketitle

\tableofcontents 

\section{Introduction}

\subsection{Motivation: Nilpotent Heuristic}

A substantial part of modern ergodic theory deals with themes of multiple recurrence, the limiting behavior of multiple ergodic averages, and the structure of Host-Kra factors and seminorms. The motivation comes from combinatorics and number theory, where ergodic methods have delivered far-reaching extensions of Szemer\'edi's theorem on arithmetic progressions (e.g. \cite{BL96, Fu77, FuKa78, KMRR26}), partition regularity results (e.g. \cite{FH17, FKM25, Mo17}), and estimates of correlations of multiplicative functions in the spirit of Chowla conjecture \cite{FH18}. These topics are also deeply connected with homogeneous dynamics of nilpotent Lie groups and estimates for multilinear Hilbert transforms in harmonic analysis. 

Starting with the seminal work of Furstenberg in the 1970s \cite{Fu77}, the study of multiple recurrence and multiple ergodic averages has been subject to phenomenal breakthroughs due to Bergelson-Leibman \cite{BL96}, Host-Kra \cite{HK05a}, and others. In turn, these developments inspired radical advances in combinatorics and number theory, exemplified by the celebrated theorems of Green-Tao \cite{GT08b} and Tao-Ziegler \cite{TZ08} on arithmetic configurations in the primes.

While original works focused on recurrence, averages, and factors for measure-preserving $\Z$-actions, later research extended some of these early results all the way to the setting of nilpotent systems. 
\begin{definition}[Nilpotent systems]\label{D: nilpotent systems}
    Let $H$ be a finitely generated $k$-step nilpotent group. We call $(X, \CX, \mu, U)$ a \emph{$k$-step nilpotent system} if $U = (U_h)_{h\in H}$ is a measure-preserving action of $H$ on a standard probability space $(X, \CX, \mu)$. If $U$ is generated by invertible measure-preserving transformations $T_1, \ldots, T_\ell$, we write the system as $(X, \CX, \mu, T_1, \ldots, T_\ell)$. We call a system \emph{nilpotent} if it is $k$-step nilpotent for some $k\in\N$.
\end{definition}

The nilpotent universe is the broadest possible setting in which positive results on multiple recurrence and norm convergence on multiple ergodic averages are known to hold. This is evidenced by the following two theorems due to Leibman and Walsh that provide important context for this work. 
\begin{theorem}[Nilpotent polynomial Szemer\'edi theorem {\cite{L98}}]\label{T: Leibman nilpotent}
    Let $p_1, \ldots, p_\ell\in\Z[n]$ have zero constant terms. For every nilpotent system $(X, \CX, \mu,$\! $T_1, \ldots, T_\ell)$ and every $E\in\CX$ with $\mu(E)>0$, 
    \begin{align*}
        \liminf_{N\to\infty}\E_{n\in[N]}
        \mu(E\cap T_1^{-p_1(n)}E\cap \cdots \cap T_\ell^{-p_\ell(n)}E)>0.
    \end{align*}
\end{theorem}
Throughout, $[N]:=\{1, \ldots, N\}$ and $\E_{n\in I} := \frac{1}{|I|}\sum_{n\in I}$ for any finite set $I$. Other notation used in the introduction can be found in Section \ref{SS: Notation}.
\begin{theorem}[Nilpotent norm convergence theorem {\cite{W12}}]\label{T: Walsh}
    Let $p_1, \ldots, p_\ell\in\Z[n]$. For every nilpotent system $(X, \CX, \mu,$\! $T_1, \ldots, T_\ell)$ and all $f_1, \ldots, f_\ell\in L^\infty(\mu)$, the average
    \begin{align}\label{E: Walsh's average}
        \E_{n\in[N]}T_1^{p_1(n)}f_1\cdots T_\ell^{p_\ell(n)}f_\ell
    \end{align}
    converges in $L^2(\mu)$.
\end{theorem}
Other results for nilpotent systems of similar flavor have been proved by Bergelson and Leibman \cite{BL02, BL03}, Zorin-Kranich \cite{ZK14, ZK16}, Ionescu, Magyar, Mirek, and Szarek \cite{IMMS23}, as well as Huang, Shao, and Ye \cite{HSY19}.

In a similar spirit, Candela and Szegedy extended the structure theory of Host-Kra factors and seminorms to ergodic actions by nilpotent groups.
\begin{theorem}[Structure theorem for nilpotent Host-Kra factors {\cite{CS23}}]\label{T: Candela-Szegedy}
    Let $H$ be a finitely generated nilpotent group and $(X, \CX, \mu, U)$ be an ergodic $H$-system. Then for every $s\in\N$, the $H$-system $(X, \CZ_s(H), \mu, U)$ is isomorphic to an $s$-step $H$-pronilsystem.\footnote{See Appendix \ref{A: nilsystems} for all relevant definitions and properties of (pro)nilsystems. {From Definition \ref{D: nilsystems}, an $s$-step $H$-pronilsystem is an $s$-step pronilsystem $Y=G/\Gamma$ on which $H$ acts by nilrotations. We caution the reader that if the group $H$ is $k$-step nilpotent, then the $s$-step $H$-pronilsystem  will be a $k$-step nilpotent system in the sense of Definition \ref{D: nilpotent systems}. Since $H$ can be identified with a subgroup of the group $G$, necessarily $k\leq s$.}}
\end{theorem}

Theorems \ref{T: Leibman nilpotent}, \ref{T: Walsh}, and \ref{T: Candela-Szegedy} reinforce the following heuristic, previously stated in \cite{Kuc26}. 
\begin{nh}\label{H: nilpotent heuristic}
    All reasonable properties of multiple recurrence, multiple ergodic averages, and Host-Kra factors that hold for abelian systems should extend to nilpotent systems.
\end{nh}

On the other hand, Bergelson and Leibman showed that both Theorems \ref{T: Leibman nilpotent} and \ref{T: Walsh} fail for measure-preserving actions of solvable groups of exponential growth \cite{BL02, BL04}. A similar counterexample can likely be constructed for Theorem~\ref{T: Candela-Szegedy}; see Problem~\ref{Pr: solvable}.

A good heuristic requires a convincing justification: what, then, is so special about nilpotent systems? The answer lies in the very definition of nilpotent groups: their level of noncommutativity, measured by the length of (the nontrivial part of) their lower central series, is bounded. As a result, various inductive arguments performed on nilpotent groups conclude in finite time. One vital consequence is that polynomial sequences in nilpotent groups form a group. This powerful fact was established by Leibman \cite{L02} and is heavily exploited in Section \ref{S: PET}.

Verification of Nilpotent Heuristic remains a formidable challenge in contemporary ergodic theory: the problems it concerns are notoriously difficult to deal with due to the confluence of dynamical, analytic, and group-theoretic issues. One direction in which the heuristic remains almost completely unexplored is the structure of limits of nilpotent ergodic averages. Walsh's convergence theorem (Theorem~\ref{T: Walsh}) is extremely general, but the price to pay is that it says nothing about what the limit is. This brings us to the central problem investigated in this paper.
\begin{mproblem}
    What can we say about the $L^2(\mu)$ limit of \eqref{E: Walsh's average}?
\end{mproblem}
Describing the limit is not only a natural question in itself; it is also a way to obtain further multiple recurrence results, and hence new extensions of Szemer\'edi's theorem. Furthermore, this problem is intimately connected to the study of multilinear operators in harmonic analysis. This includes both the discrete operators naturally arising in the Furstenberg-Bergelson-Leibman conjecture in pointwise ergodic theory (see e.g. \cite{KMPWW24, KMT20, Mirek26}) and variants of the multilinear Hilbert transform  from continuous harmonic analysis (e.g. \cite{CDR21, HL23, LT97, Tao16}).

Up to several years ago, we had limited tools to analyze the limits of \eqref{E: Walsh's average} even for commuting actions. Recently, a number of breakthroughs \cite{DFKS22, DKKST25, DKKST24, DKS22,DKS23,Fr21, FrKu22b, FrKu22a} allowed us to overcome this barrier and develop a potent theory of averages in the commuting setting. This paper takes the next logical step towards Main Problem~and develops machinery for the study of 2-step nilpotent ergodic averages.

\subsection{Limits of 2-step nilpotent ergodic averages}

Our main results concern the limiting behavior of \eqref{E: Walsh's average} whenever the group is 2-step nilpotent and the iterates are nonconstant polynomials of distinct degrees. {We recall that a measure-preserving transformation $T$ is \textit{totally ergodic} if $T, T^2, T^3, \ldots, $ are all ergodic.}
\begin{theorem}\label{T: nilpotent joint ergodicity}
    Let $\ell\in\N$. For all 2-step nilpotent systems $(X, \CX, \mu, T_1, \ldots, T_\ell)$ with totally ergodic $T_1, \ldots, T_\ell$, nonconstant polynomials $p_1, \ldots, p_\ell\in\Z[n]$ of distinct degrees, and functions $f_1, \ldots, f_\ell\in L^\infty(\mu)$, we have
    \begin{align}\label{E: main equality}
    \lim_{N\to\infty}\norm{\E_{n\in[N]}T_1^{p_1(n)}f_1\cdots T_\ell^{p_\ell(n)}f_\ell - \int f_1\; d\mu\cdots \int f_\ell\; d\mu}_{L^2(\mu)} = 0.
\end{align}
\end{theorem}
Theorem~\ref{T: nilpotent joint ergodicity} exemplifies the phenomenon of \textit{joint ergodicity}, which has received significant attention in recent years (see \cite{Kuc26} for a survey on these developments). In particular, it resolves Problem~49 and makes substantial progress towards addressing Problems 45 and 46 from the aforementioned survey. 
When $T_1 = \ldots = T_\ell$, Theorem~\ref{T: nilpotent joint ergodicity} was proved by Frantzikinakis and Kra \cite{FrKr05}. In the commuting case, the result follows from a work of Chu, Frantzikinakis, and Host \cite{CFH11} for monomials, and it was extended by Frantzikinakis and the second author \cite{FrKu22a} to all independent\footnote{We call polynomials $p_1, \ldots, p_\ell\in\Z[n]$ \textit{independent} if any nontrivial  linear combination is nonconstant.} polynomials. In the 2-step nilpotent case, the result is completely new even for the relatively simple limit
\begin{align*}
    \lim_{N\to\infty}\E_{n\in[N]}T_1^n f_1\cdot T_2^{n^2}f_2.
\end{align*}

The assumption of total ergodicity prevents local obstructions; without this assumption, \eqref{E: main equality} fails in general even for $\ell=1$ and $p_1(n)=n^2$. In the absence of total ergodicity, we obtain Host-Kra seminorm\footnote{Defined in Definition \ref{D: box seminorms}.} estimates for the same average.

\begin{theorem}\label{T: nilpotent seminorm estimates}
    Let $d,\ell\in\N$. There exists an integer $s=O_{d,\ell}(1)$ such that for all 2-step nilpotent systems $(X, \CX, \mu, T_1, \ldots, T_\ell)$, polynomials $p_1, \ldots, p_\ell\in\Z[n]$ satisfying
    \begin{align*}
        0<\deg p_1 < \cdots < \deg p_\ell \leq d,
    \end{align*}
    and functions $f_1, \ldots, f_\ell\in L^\infty(\mu)$, we have
    \begin{align}
    \lim_{N\to\infty}\norm{\E_{n\in[N]}T_1^{p_1(n)}f_1\cdots T_\ell^{p_\ell(n)}f_\ell}_{L^2(\mu)} = 0
\end{align}
whenever $\nnorm{f_j}_{s,T_j} = 0$ for some $1\leq j\leq \ell$.
\end{theorem}
The single-transformation case of Theorem~\ref{T: nilpotent seminorm estimates} is known due to Host and Kra \cite{HK05b} and Leibman \cite{L05a}, whereas the commuting case was handled by Chu, Frantzikinakis, and Host \cite{CFH11} (with later extension by Frantzikinakis and the second author \cite{FrKu22a} to all pairwise independent polynomials).

While we believe the following conjecture to be true, our methods fail beyond the setting of 2-step nilpotent systems and distinct-degree polynomials. The necessity of these assumptions for our techniques is explained in Section \ref{SS: outline}.
\begin{conjecture}\label{C: main conjecture}
     Let $\ell\in\N$. For all nilpotent systems $(X, \CX, \mu, T_1, \ldots, T_\ell)$ with totally ergodic $T_1, \ldots, T_\ell$, independent polynomials $p_1, \ldots, p_\ell\in\Z[n]$, and functions $f_1, \ldots, f_\ell\in L^\infty(\mu)$, we have
    \begin{align}
    \lim_{N\to\infty}\norm{\E_{n\in[N]}T_1^{p_1(n)}f_1\cdots T_\ell^{p_\ell(n)}f_\ell - \int f_1\; d\mu\cdots \int f_\ell\; d\mu}_{L^2(\mu)} = 0.
\end{align}
\end{conjecture}
The only previous results on the structure of properly nilpotent ergodic averages have been established by Bergelson and Leibman \cite{BL02} in a work predating Walsh's. They studied the average
\begin{align}\label{E: BL average}
    \lim_{N\to\infty}\E_{n\in[N]}T^n f_1 \cdot S^n f_2
\end{align}
for a nilpotent group of transformations $\langle T, S\rangle$. This setting is somewhat orthogonal to ours as having only linear iterates comes with a different set of advantages and disadvantages. Among other things, Bergelson and Leibman established norm convergence and identified characteristic factors for such averages. While their remarkable work covers nilpotent groups of arbitrary step, it does not extend to more than two terms.

Arguing as in \cite{FHK10}, we can restrict the averaging in Theorems \ref{T: nilpotent joint ergodicity} and \ref{T: nilpotent seminorm estimates} to the primes.\footnote{The same comment applies applies to other joint ergodicity results (Theorems \ref{T: joint ergodicity conjecture} and \ref{T: nilpotent joint ergodicity}) and seminorm estimates (Theorem~\ref{T: new estimates}). As these modifications are completely standard, we will not provide more details.} Indeed, to see this for Theorem~\ref{T: nilpotent joint ergodicity}, let $(X, \CX, \mu, T_1, \ldots, T_\ell)$ be a 2-step nilpotent system with totally ergodic $T_1, \ldots, T_\ell$, nonconstant polynomials $p_1, \ldots, p_\ell\in\Z[n]$ of distinct degrees, and functions $f_1, \ldots, f_\ell\in L^\infty(\mu)$. Theorem~\ref{T: nilpotent joint ergodicity} implies that
    \begin{align}
    \lim_{N\to\infty}\norm{\E_{n\in[N]}T_1^{p_1(Wn+b)}f_1\cdots T_\ell^{p_\ell(Wn+b)}f_\ell - \int f_1\; d\mu\cdots \int f_\ell\; d\mu}_{L^2(\mu)} = 0
\end{align}
for all $W,b\in\N.$ Following the strategy laid out in \cite[Subsection~1.3]{FHK10}, the transference result \cite[Theorem~1.3]{FHK10} holds more generally for nilpotent actions. Hence
\begin{align}
    \lim_{N\to\infty}\norm{\E_{n\in\P\cap [N]}T_1^{p_1(n)}f_1\cdots T_\ell^{p_\ell(n)}f_\ell - \int f_1\; d\mu\cdots \int f_\ell\; d\mu}_{L^2(\mu)} = 0,
\end{align}
where $\P$ is the set of primes. The result that undepins these manipulations is the Gowers uniformity of the shifted modified von Mangoldt function due to Green and Tao \cite{GT10b, GT12b}.

{Another standard corollary of Theorem \ref{T: nilpotent seminorm estimates} is that for 2-step nilpotent systems $(X, \CX, \mu$, $T_1, \ldots, T_\ell)$, distinct-degree polynomials $p_1, \ldots, p_\ell\in\Z[n]$, and functions $f_1, \ldots, f_\ell\in L^\infty(\mu)$, the multicorrelation sequence 
\begin{align*}
    \psi(n) = \int f_0 \cdot T_1^{p_1(n)}f_1 \cdots T_\ell^{p_\ell(n)}f_\ell\; d\mu
\end{align*}
admits a nil+null decomposition, i.e. it can be written as a sum of a nilsequence (whose degree depends only on the degrees of $p_1, \ldots, p_\ell$) and a null-sequence. The proof goes as in \cite[Section 2.6]{FrKu22a}, which also contains the definitions of nilsequences and nullsequences.} 

\subsection{Nilpotent polynomial Szemer\'edi theorem}

As a corollary of Theorem~\ref{T: nilpotent seminorm estimates}, we obtain Khintchine-type lower bounds in a particular case of Theorem~\ref{T: Leibman nilpotent}.
\begin{theorem}\label{T: Khintchine}
    Let $\ell\in\N$. For all 2-step nilpotent systems $(X, \CX, \mu, T_1, \ldots, T_\ell)$, sets $E\in\CX$ with $\mu(E)>0$, integers $0<d_1< \cdots<d_\ell$, and $\veps>0$, the set
    \begin{align}
    \rem{n\in\Z\colon \mu(E\cap T_1^{-n^{d_1}}E\cap \cdots \cap T_\ell^{-n^{d_\ell}}E)>\mu(E)^{\ell+1}-\veps}
\end{align}
is syndetic.
\end{theorem}
Theorem~\ref{T: Khintchine} extends results of Frantzikinakis-Kra (for $T_1 = \cdots = T_\ell$) \cite{FrKr06}, Chu-Frantzikinakis-Host \cite{CFH11} (for commuting transformations), as well as Frantzikinakis and the second author (for commuting transformations and independent polynomials \cite{FrKu22a}). 
The derivation of Theorem~\ref{T: Khintchine} from Theorem~\ref{T: nilpotent seminorm estimates} follows verbatim the arguments from \cite[Section 7]{CFH11}, where the corresponding lower-bound result \cite[Theorem~1.3]{CFH11} is deduced from seminorm estimates \cite[Theorem~1.2]{CFH11} (the proofs in \cite[Section 7]{CFH11} use nothing about the nature of the group $\langle T_1, \ldots, T_\ell\rangle$).

Combining Theorem~\ref{T: Khintchine} with the Furstenberg correspondence principle for amenable groups (see e.g. \cite[Theorem~2.3]{BFM21}), we obtain a variant of the polynomial Szemer\'edi theorem in nilpotent groups for configurations of the form
\begin{align*}
    x,\; h_1^{n^{d_1}} x,\; \ldots,\; h_\ell^{n^{d_\ell}}x
\end{align*}
for some fixed choice of group elements $h_1, \ldots, h_\ell$ and degrees $d_1, \ldots, d_\ell$.
In what follows, $d^*(A)$ denotes the upper Banach density of a subset $A\subseteq H$ of an amenable group $H$.\footnote{A \textit{left F{\o}lner sequence} on an amenable group $H$ is a sequence $\Phi = (\Phi_N)_{N\in\N}$ of finite subsets of $H$ such that $\frac{|\Phi_N\cap (h\cdot \Phi_N)|}{|\Phi_N|}\to 0$ for all $h\in H$. Then the \textit{upper density} of $A$ \textit{along} $H$ is given by \begin{align*}
        \overline{d}_\Phi(A):=\limsup_{N\to\infty}\frac{|A\cap \Phi_N|}{|\Phi_N|}>0,
    \end{align*}
    and the \textit{upper Banach density} of $A$ is $d^*(A):=\sup_{\Phi}\bar d_\Phi(A)$,  with the supremum taken over all left F{\o}lner sequences.}

\begin{corollary}\label{C: popular common differences}
    Let $H$ be a finitely generated 2-step nilpotent group and $A\subseteq H$ be a set with $d^*(A)>0$.
    Then for every $h_1, \ldots, h_\ell\in H$ and $\veps>0$, the set
        \begin{align}\label{E: popular common difference}
    \rem{n\in\Z\colon d^*(A\cap (h_1^{-n^{d_1}}\cdot A)\cap \cdots \cap (h_\ell^{-n^{d_\ell}}\cdot A))>d^*(A)^{\ell+1}-\veps}
\end{align}
is syndetic.\footnote{The same result works if we replace $d^*$ everywhere by $\bar d_\Phi$ for a fixed choice of a left F{\o}lner sequence $\Phi$.}
\end{corollary}

The members of the set \eqref{E: popular common difference} are often referred to as \textit{popular common differences}. To the best of our knowledge, Corollary \ref{C: popular common differences} is the first positive result on the existence of popular common differences in the nilpotent setting. On the other hand, Donoso and the last author previously showed that the configurations of the form $x,\; h_1^nx,\; h_2^n x$ generally do \textit{not} admit popular common differences in the 2-step nilpotent setting \cite{DS18}.

We believe that Theorem~\ref{T: Khintchine} and Corollary \ref{C: popular common differences} should hold in the following level of generality.
\begin{conjecture}\label{Conj: popular common differences}
    The conclusions of Theorem~\ref{T: Khintchine} and Corollary \ref{C: popular common differences} hold for a nilpotent system/group of arbitrary step and all independent polynomials $p_1, \ldots, p_\ell\in\Z[n]$ (rather than  just $n^{d_1}, \ldots, n^{d_\ell}$).
\end{conjecture}
However, the methods of \cite{CFH11} that allow us to derive Theorem~\ref{T: Khintchine} and Corollary \ref{C: popular common differences} from Theorem~\ref{T: nilpotent seminorm estimates}, and specifically the equidistribution result \cite[Lemma~7.6]{CFH11}, work solely for polynomials $p_1, \ldots, p_\ell$ satisfying $n^{\deg p_j +1}|p_{j+1}(n)$ for all $1\leq j<\ell$. 

\subsection{Joint ergodicity}

As already remarked in the previous section, Theorem~\ref{T: nilpotent joint ergodicity} can be seen through the lenses of joint ergodicity, a term introduced by Berend and Bergelson \cite{BB84}.
\begin{definition}[Joint ergodicity]
    Let $\ell, L\in\N$. Let $H$ be a countable amenable group, $(X, \CX, \mu, U)$ be an $H$-system, and $a_1, \ldots, a_\ell: \Z^L\to H$ be sequences. We say that $a_1, \ldots, a_\ell$ are \textit{jointly ergodic} for $(X, \CX, \mu, U)$ (or that $(U_{a_1(\bn)}, \ldots, U_{a_\ell(\bn)})_{\bn\in\Z^L}$ is \textit{jointly ergodic}) if for all $f_1, \ldots, f_\ell\in L^\infty(\mu)$,
    \begin{align}\label{E: joint ergodicity}
        \lim_{N\to\infty}\norm{\E_{\bn\in[N]^L}U_{a_1(\bn)}f_1\cdots U_{a_\ell(\bn)}f_\ell - \int f_1\; d\mu\cdots \int f_\ell\;d\mu}_{L^2(\mu)} = 0.
    \end{align}
\end{definition}
Thus, Theorem~\ref{T: nilpotent joint ergodicity} provides \textit{sufficient} conditions for joint ergodicity in 2-step nilpotent systems. The following conjecture, stated e.g. in {\cite[Conjecture 1.3]{DFKS22}, \cite[Conjecture 1.5]{DKS22}, \cite[Problem~18]{Kuc26}}, proposes \textit{sufficient and necessary} conditions for polynomials in abelian systems.
\begin{conjecture}[Joint ergodicity conjecture in the abelian setting]\label{C: joint ergodicity conjecture}
    Let $D,\ell, L\in\N$. Let $(X, \CX, \mu, U)$ be a $\Z^D$-system and $p_1, \ldots, p_\ell\in\Z^L[\bn]$ be polynomials. Then $p_1, \ldots, p_\ell$ are jointly ergodic for $(X, \CX, \mu, U)$ if and only if the following two conditions hold:
    \begin{enumerate}
        \item (Product ergodicity condition) the product action $(U_{p_1(\bn)}\times \cdots\times U_{p_\ell(\bn)})_{\bn\in\Z^L}$ is ergodic\footnote{{For an amenable group $H$, we say that a sequence $(S_h)_{h\in H}$ of measure-preserving transformations on $(X, \CX, \mu)$ is \textit{ergodic} if for all $f\in L^\infty(\mu)$, the uniform limit $\E_{h\in H}S^h f$ exists and equals $\int f\; d\mu$.}} for $(X^\ell, \CX^{\otimes \ell}, \mu^{\otimes \ell})$;
        \item (Difference ergodicity condition) the difference actions $(U_{p_i(\bn)}\circ U\inv_{p_j(\bn)})_{\bn\in\Z^L}$ are ergodic for $(X, \CX, \mu)$ for all $1\leq i < j\leq \ell$.
    \end{enumerate}
\end{conjecture}
Our next result resolves Conjecture \ref{C: joint ergodicity conjecture} in full.
\begin{theorem}\label{T: joint ergodicity conjecture}
    Conjecture \ref{C: joint ergodicity conjecture} holds.
\end{theorem}

Most of the work on Conjecture \ref{C: joint ergodicity conjecture} up to now focused on \textit{single-dimensional} (and single-variable) polynomials, i.e. those that take the form $vq(n)$ for some $v\in \Z^D$ and $q\in\Z[n]$. In this setting, we examine the joint ergodicity of $(T_1^{p_1(n)}, \ldots, T_\ell^{p_\ell(n)})_{n\in\Z}$ for some commuting transformations $T_1, \ldots, T_\ell$ and polynomials $p_1, \ldots, p_\ell\in\Z[n]$. Conjecture \ref{C: joint ergodicity conjecture} for single-dimensional polynomials was tackled by Donoso as well as the first and third authors in the same-iterate case (i.e. $p_1 = \cdots = p_\ell$) \cite{DKS22}; by Frantzikinakis and the second author in the complementary case of pairwise independent polynomials \cite{FrKu22a}; and eventually it was resolved by Frantzikinakis and the second author for all single-dimensional polynomials \cite{FrKu22b}, with an alternative proof supplied later by Donoso, Tsinas and the authors \cite{DKKST25}. In the multidimensional setting, some special cases were previously addressed by Donoso, Ferr{\'e}-Moragues as well as the first and third authors \cite{DFKS22}.

We do emphasize that the proof of Theorem~\ref{T: joint ergodicity conjecture} is quite different from earlier proofs, including those fully resolving the single-dimensional case \cite{DKKST25, FrKu22b}. At its heart lie much more robust seminorm estimates, presented in Theorem~\ref{T: new estimates} below, that also happen to be one of the key new ingredients in the proof of Theorem~\ref{T: nilpotent seminorm estimates}, our seminorm estimates for averages coming from 2-step nilpotent systems.

There is by now a vast body of works examining variants of Conjecture \ref{C: joint ergodicity conjecture} in other contexts \cite{BLS16, BS23, DKKST25, DKKST24, DKS23, DKS24,  KoSu23}. In particular, in \cite{DKKST25}, Donoso, Tsinas, and the authors found a counterexample \cite[Theorem~1.10]{DKKST25} to Conjecture \ref{C: joint ergodicity conjecture} for Hardy sequences. It turns out that the natural generalization of Conjecture \ref{C: joint ergodicity conjecture} also fails for polynomials in the 2-step nilpotent setting.

\begin{theorem}\label{T: counterexample to joint ergodicity conjecture}
    There exists a 2-step affine nilsystem\footnote{In this case (and in Theorem~\ref{T: counterexample to joint ergodicity criteria} below), both the nilsystem is 2-step and the group of transformations is 2-step nilpotent, so the phrase ``2-step affine nilsystem'' causes no ambiguity.}
    $(\T^2, \CB_{\T^2}, m_{\T^2}, T_1, T_2, T_3, T_4)$ and polynomials $p_1, p_2, p_3, p_4\in\Z[n]$ such that:
        \begin{enumerate}
        \item $(T_1^{p_1(n)}\times T_2^{p_2(n)}\times T_3^{p_3(n)} \times T_4^{p_4(n)})_{n\in\Z}$ is ergodic for $(\T^8, \CB_{\T^8}, m_{\T^8})$;
        \item $(T_i^{p_i(n)} T_j^{-p_j(n)})_{n\in\Z}$ is ergodic for $(\T^2, \CB_{\T^2}, m_{\T^2})$ for all $1\leq i < j\leq 4$;
        \item $(T_1^{p_1(n)}, T_2^{p_2(n)}, T_3^{p_3(n)}, T_4^{p_4(n)})_{n\in\Z}$ is \emph{not} jointly ergodic for $(\T^2, \CB_{\T^2}, m_{\T^2})$.
    \end{enumerate}
\end{theorem}

As a counterpoint to Theorem~\ref{T: counterexample to joint ergodicity conjecture}, Bergelson and Leibman \cite[Theorem~B]{BL02} previously showed that if $T_{1},T_{2}$ generate a nilpotent group of arbitrary step, then $(T_{1}^{n},T_{2}^{n})_{n\in\Z}$ is jointly ergodic if and only if $T_{1}\times T_{2}$ is ergodic for $(X^{2},\mathcal{X}^{2},\mu^{\otimes 2})$ and the group $\langle T_{1}^{-n}T_{2}^{n}\colon n\in\Z\rangle$ acts ergodically on $(X,\mathcal{X},\mu)$.
We add another positive result to this collection, showing that on the whole, the extent to which Conjecture \ref{C: joint ergodicity conjecture} holds in the nilpotent setting remains unclear. 
\begin{theorem}\label{T: joint ergodicity conjecture2}
    Let $\ell\in\N$ and $(X, \CX, \mu, T_{1},\dots,T_{\ell})$ be a measure preserving system with $T_{1},\dots,T_{\ell}$ generating a 2-step nilpotent group. Let $p_1, \ldots, p_\ell\in\Z[n]$ be nonconstant polynomials of distinct degrees. Then $(T^{p_{1}(n)}_{1},\dots,T^{p_{\ell}(n)}_{\ell})_{n\in\Z}$ is jointly ergodic  if and only if  $(T^{p_{1}(n)}_{1}\times\dots\times T^{p_{\ell}(n)}_{\ell})_{n\in\Z}$ is ergodic for $(X^\ell ,\CX^{\otimes \ell}, \mu^{\ell})$.
\end{theorem}

A true game changer in the recent study of joint ergodicity in the abelian setting were new joint ergodicity criteria developed in \cite{BFM22, Fr21, FrKu22a}. Using these criteria, the task of establishing joint ergodicity was reduced to two independent endeavors: obtaining Host-Kra seminorm control and verifying the identity \eqref{E: joint ergodicity} for eigenfunctions (which in practice amounted to applying known exponential sums estimates). Unfortunately, the same criteria fail to hold in the 2-step nilpotent setting.
\begin{theorem}\label{T: counterexample to joint ergodicity criteria}
    There exists a 2-step affine nilsystem  $(\T^2, \CB_{\T^2}, m_{\T^2}, T_1, T_2, T_3, T_4)$ and polynomials $p_1, p_2, p_3, p_4\in\Z[n]$ such that the polynomials are not jointly ergodic for the system, but they are jointly ergodic for the projection of the system onto the Kronecker factor (i.e. onto the first coordinate).\footnote{This last property is known in the joint ergodicity literature as being \textit{good for equidistribution}, as it is equivalent to the equidistribution of $(\alpha_1 p_1(n), \alpha_2 p_2(n), \alpha_3 p_3(n), \alpha_4 p_4(n))_{n\in\Z}$ on $\T^4$ for all nonzero eigenvalues $\alpha_j$ of $T_j$, or the validity of \eqref{E: joint ergodicity} for all eigenfunctions of respective transformations.}
\end{theorem}

What Theorems \ref{T: counterexample to joint ergodicity conjecture} and \ref{T: counterexample to joint ergodicity criteria} indicate is that outside the realm of independent sequences, joint ergodicity is governed by different principles in the nilpotent vs. abelian settings. Thus, we are getting a class of results contradicting Nilpotent Heuristic. In Section \ref{S: open problems}, we propose conjectures regarding nilpotent joint ergodicity criteria.

\subsection{New seminorm estimates for polynomial averages in $\Z^D$-systems}
A major ingredient in the proofs of Theorems \ref{T: nilpotent seminorm estimates} and \ref{T: joint ergodicity conjecture} are new seminorm estimates for multiple ergodic averages involving {multiparameter} polynomials over $\Z^D$. 
To put them in context, we compare them with existing results. Throughout this subsection, let $(X, \CX, \mu, U)$ be a $\Z^D$-system, $\bn\in\Z^L$ be an indeterminate for some fixed $L\in\N$, and $p_1, \ldots, p_\ell\in \Z^D[\bn]$ be distinct polynomials of degree at most $d$ (which we assume without loss of generality to have zero constant terms).

For $1\leq j\leq \ell$, denote $$p_{j}(\bn)=\sum\limits_{\substack{\bi\in\N_0^{L}\colon\\ 1\leq \vert\bi\vert\leq d}} b_{j,\bi}\bn^\bi,$$
where $\bn^{\bi}:=n_{1}^{i_{1}}\dots n_{L}^{i_{L}}$ and $\vert\bi\vert:=i_{1}+\dots+i_{L}$. Set also $p_0 := 0$.
For distinct $0\leq j,j'\leq \ell$, 
define two groups:
\begin{align*}
    H_{j,j'}&:=\langle b_{j,\bi}-b_{j',\bi}\colon\; 1\leq \vert\bi\vert\leq d \rangle,\\
    \tilde H_{j,j'}&:=\langle b_{j,\bi}-b_{j',\bi}\colon\; \vert\bi\vert= \deg (p_j-p_{j'}) \rangle.
\end{align*}
The group $H_{j,j'}$ is generated by all the coefficients of $p_j-p_{j'}$ while the group $\tilde H_{j,j'}$ is generated only by its leading coefficients.
Clearly, $\tilde H_{j,{j'}}\subseteq H_{j,{j'}}$. Moreover, both types of subgroups are symmetric in that $H_{j,{j'}} = H_{j',j}$ and $\tilde H_{j,{j'}} = \tilde H_{j',j}$ for any distinct $0\leq j,{j'}\leq \ell$. Existing seminorm estimates, proved in \cite[Theorem~2.5]{DFKS22} and quantified in \cite[Theorem~1.6]{KKL24a}  and \cite[Proposition~8.6]{DKKST24}, give us the following (for the definition of box seminorms, see Definition \ref{D: box seminorms}, and for the averaging notation, consult Section \ref{SS: Notation}).
\begin{theorem}\label{T: original estimates}
    There exists a natural number $s=O_{d,\ell,L}(1)$ such that for all  $f_1, \ldots, f_\ell\in L^\infty(\mu)$, we have
    \begin{align*}
        \norm{\E_{\bn\in\Z^{L}} U_{p_1(\bn)}f_1\cdots U_{p_\ell(\bn)}f_\ell}_{L^2(\mu)} = 0
    \end{align*}
    whenever $\nnorm{f_\ell}_{\tilde H_{\ell,0}^{\times s}, \tilde H_{\ell,1}^{\times s}, \ldots, \tilde H_{\ell, \ell-1}^{\times s}} = 0$.
\end{theorem}

In this paper, we derive alternative seminorm estimates of the following form.
\begin{theorem}\label{T: new estimates}
    There exists a natural number $s=O_{d,\ell,L}(1)$ such that for all  $f_1, \ldots, f_\ell\in L^\infty(\mu)$, we have
    \begin{align*}
        \norm{\E_{\bn\in\Z^{L}} U_{p_1(\bn)}f_1\cdots U_{p_\ell(\bn)}f_\ell}_{L^2(\mu)} = 0
    \end{align*}
    whenever $\nnorm{f_\ell}_{H_{j,{j'}}^{\times s}\colon\; 0\leq j<{j'}\leq \ell} = 0$.
\end{theorem}

In both Theorems \ref{T: original estimates} and \ref{T: new estimates}, the estimates are stated in terms of a seminorm of $f_\ell$, but by symmetry one can easily deduce estimates in terms of other functions. We note that the box seminorms appearing in Theorem~\ref{T: original estimates} depend on which of the functions $f_1, \ldots, f_\ell$ we look at, whereas the estimates in Theorem~\ref{T: new estimates} are the same for all functions.

\begin{example}
    To compare the conclusions of Theorems \ref{T: original estimates} and \ref{T: new estimates}, 
    consider the average
    \begin{align*}
        \E_{n\in\Z}U_{b_{1,1} n+ b_{1,2} n^2}f_1\cdot U_{b_{2,1}n + b_{2,2}n^2}f_2
    \end{align*}
    for distinct and nonzero $b_{j,i}$'s.
    Theorem~\ref{T: original estimates} allows us to control\footnote{Throughout the paper, we say that a seminorm $\nnorm{\cdot}$ on a vector space $V$ \textit{controls} a function $\Lambda\colon V^\ell\to W$ at some fixed index $1\leq j\leq \ell$ if $\nnorm{f_j} = 0$ implies that $\Lambda(f_1, \ldots, f_\ell) = 0$ for all $f_1, \ldots, f_\ell\in V$. Here, $W$ is also a vector space, and it is typically either $V$ or $\C$. More casually, we express the same notion by saying that $\nnorm{f_j}$ controls $\Lambda(f_1, \ldots, f_\ell)$.} this average by the seminorms 
    \begin{align*}
        \nnorm{f_1}_{b_{1,2}^{\times s}, (b_{1,2}-b_{2,2})^{\times s}}\quad \textrm{and}\quad \nnorm{f_2}_{b_{2,2}^{\times s}, (b_{2,2}-b_{1,2})^{\times s}}
    \end{align*}
    for some $s\in\N$, while Theorem~\ref{T: new estimates} delivers control in terms of
    \begin{align*}
        \nnorm{f_j}_{\langle b_{1,1}, b_{1,2}\rangle^{\times s}, \langle b_{2,1}, b_{2,2}\rangle^{\times s}, \langle b_{1,1}-b_{2,1}, b_{1,2}-b_{2,2}\rangle^{\times s}}
    \end{align*}
    for any $j=1,2$ and some $s\in\N$.
\end{example}

Based on this example, it is not objectively clear which of the two outcomes is better. The original estimates from Theorem~\ref{T: original estimates} are strictly better for multiple ergodic averages along pairwise independent single-dimensional polynomials (i.e. taking the form $p_j = v_j q_j$ for some $v_j\in\Z^D$ and $q_j\in\Z[n]$). Indeed, in this case $ H_{j,j'} =\tilde H_{j,j'}$ for all distinct $0\leq j,j'\leq \ell$, and so Theorem~\ref{T: original estimates} involves fewer different subgroups than Theorem~\ref{T: new estimates}. For instance, for
\begin{align*}
    \E_{n\in\Phi}U_{v_1q_1(n)}f_1\cdot U_{v_2q_2(n)}f_2,
\end{align*}
Theorem~\ref{T: original estimates} gives control in terms of a seminorm of $f_2$ involving only
\begin{align*}
    \tilde H_{2,0} = H_{2,0} = \langle v_2\rangle\quad \textrm{and}\quad \tilde H_{2,1} = H_{2,1} =\langle v_2-v_1\rangle
\end{align*}
while Theorem~\ref{T: new estimates} additionally involves the subgroup $H_{1,0} = \langle{v_1}\rangle$.
On the other hand, the estimates provided by Theorem~\ref{T: new estimates} are better suited for applications covered in this paper, in particular for the proof of Theorem~\ref{T: nilpotent seminorm estimates}. We do emphasize, though, that the proof of Theorem~\ref{T: new estimates} is built upon Theorem~\ref{T: original estimates}.

\subsection{Equidistribution on nilsystems}
As a step in the proof of Theorem~\ref{T: nilpotent joint ergodicity}, we derive an equidistribution result on nilsystems that is of independent interest. For definitions and notation regarding nilsystems, we direct the reader to Appendix \ref{A: nilsystems}.
\begin{theorem}\label{T: equidistribution}
    Let $(Y, \CB_Y, m_Y, S_1, \ldots, S_\ell)$ be a nilsystem with $S_1, \ldots, S_\ell$ totally ergodic,\footnote{If the nilmanifold structure of $Y$ comes from the quotient $Y=G/\Gamma$ for an $s$-step nilpotent group $G$, then $\langle S_1, \ldots, S_\ell\rangle$ is isomorphic to a subgroup of $G$, so it is necessarily nilpotent of step at most $s$.} and let $p_1, \ldots, p_\ell\in\Z[n]$ be independent polynomials.
    Then for $m_Y$-a.e. $y\in Y$, the orbit
    \begin{align*}
        (S_1^{p_1(n)}y, \ldots, S_\ell^{p_\ell(n)}y)_{n\in\Z}
    \end{align*}
    is equidistributed in $Y^{\ell}$.
\end{theorem}
For $S_1 = \cdots = S_\ell$, Theorem~\ref{T: equidistribution} has been established by Frantzikinakis and Kra \cite{FrKr05} while in the commuting case, it follows from the joint ergodicity result of Frantzikinakis and the second author \cite{FrKu22a} combined with the pointwise convergence results of Leibman \cite{L05c, L05b}.

\subsection{Notation, conventions, and basic lemmas}\label{SS: Notation}
The symbols $\C, \R, \R_+, \Q, \Z, \N_0, \N, \P$ respectively stand for the sets of complex numbers, reals, positive reals, rationals, integers, nonnegative integers, positive integers, and primes. With $\T$, we denote the one dimensional torus, and we often identify it with $\R/\Z$ or  with $[0,1)$. Likewise, $\S^1$ stands for the unit circle inside $\C$. 
For $N\in\N$, we denote $[N]:=\{1, \ldots, N\}$. 

{In a group $G$, we write $\langle E \rangle$ for the subgroup generated by $E\subseteq G$. We denote the identity element via $e_G$  unless it is clear from the context that the group is abelian - then we write 0 instead.}

We let $\CC z := \overline{z}$ be the complex conjugate of $z\in \C$ and $|\eps|:=\eps_1 + \cdots + \eps_s$ for $\eps=(\eps_1, \ldots, \eps_s)\in\{0,1\}^s.$

Let $(X, \CX, \mu)$ be a probability space. We call a function $f\in L^\infty(\mu)$ \emph{$1$-bounded} if $\norm{f}_{L^\infty(\mu)}\leq 1$.

Unless stated otherwise, all the limits are taken in $L^2(\mu)$.


For a $\sigma$-algebra $\CA\subseteq\CX$ and $f\in L^1(\mu)$, we let $\E(f|\CA)$ be the conditional expectation of $f$ with respect to $\CA$. 

Throughout, all the equalities between functions and sets are understood to hold up to sets of measure 0.

Given two functions $a,b: X\to\C$ on some set $X$, we write
\begin{enumerate}
    \item $a\ll b$ or $a = O(b)$ if there exists $C>0$ 
    such that $|a(x)| \leq C |b(x)|$ for all $x\in X$; 
    \item $a \asymp b$ if $a\ll b$ and $b\ll a$.
\end{enumerate}

For a set $E$, we define its indicator function by $1_E$. 

Given a measure preserving system $(X, \CX, \mu, T)$, we also write $T$ for the corresponding\textit{ Koopman operator}; that is, $(Tf)(x) := f(Tx)$. In expressions like $TSf$, we interpret $TS$ as the Koopman operator corresponding to the measure-preserving transformation $TS$ (rather than as a composition of two Koopman operators), and so
\begin{align}\label{E: composition convention}
    (TSf)(x) := f(TSx) = (S(Tf))(x).
\end{align}

\subsubsection*{Averaging notation}

For a finite set $E,$ we define the average of a sequence $a:E\to \C$ as 
\[\E_{n\in E}a(n) :=\frac{1}{|E|}\sum_{n\in E}a(n).\] 
We extend the averaging notation to limits in amenable groups as follows. If $H$ is a countable amenable group and $\Phi = (\Phi_N)_{N\in\N}$ is a F{\o}lner sequence on $H$, we write
\begin{align*}
    \E_{h\in \Phi}f(h):=\lim_{N\to\infty}\E_{h\in\Phi_N}f(h)
\end{align*}
for $f:H\to\C$ if the limit exists. If the limit exists and is independent on the choice of the F{\o}lner sequence, we write
\begin{align*}
    \E_{h\in H}f(h):=\E_{h\in\Phi}f(h)
\end{align*}
and refer to it as the \textit{uniform limit}. We also write $\E_{h,h'\in H}$ for $\E_{(h,h')\in H^2}$.

If $f$ is $\R$-valued, we also write
\begin{align*}
    \oE_{h\in \Phi}f(h):=\limsup_{N\to\infty}\E_{h\in\Phi_N}f(h)
\end{align*}
and
\begin{align*}
    \oE_{h\in H}f(h):=\sup_{\Phi}\oE_{h\in\Phi}f(h),
\end{align*}
where the supremum is taken over all the F{\o}lner sequence on $H$.

{\subsubsection*{Van der Corput inequalities}
Let $(f_n)_{n\in \Z}$ be a uniformly bounded sequence in some Hilbert space $(\mathfrak{H}, \norm{\cdot})$. Then the \textit{asymmetric} and \textit{symmetric van der Corput inequalities} are given by 
\begin{align*}
    \norm{\oE_{n\in \Z}f_n}^2\leq \oE_{h\in \Z}\abs{\oE_{n\in\Z} \langle f_n, f_{n+h}\rangle} \quad \textrm{and}\quad \norm{\oE_{n\in \Z}f_n}^2\leq \oE_{h,k\in \Z}\abs{\oE_{n\in\Z} \langle f_{n+h}, f_{n+k}\rangle}.
\end{align*}
If the limit $\E_{n\in\Z} \langle f_{n+h}, f_{n+k}\rangle$ exists for every $h,k\in\Z$, we can remove the absolute value in the symmetric van der Corput inequality, a fact that we crucially use in Section \ref{S: PET}.}

\subsection*{Acknowledgments}
We would like to thank Bryna Kra for help with the proof of Proposition~\ref{P: comparing HK factors in nilpotent systems}, Pablo Candela for directing us to Theorem~\ref{T: Candela-Szegedy} and extensive discussions regarding nilmanifolds and nilspace theory, James Leng for feedback on Section \ref{S: open problems}, and Nikos Frantzikinakis for suggesting Theorem~\ref{T: Khintchine}.

\section{Outline of the paper, main challenges, and new ideas}\label{SS: outline}
Having stated our main results, we proceed to outline new ideas and major challenges. Theorem~\ref{T: nilpotent joint ergodicity} is proved by assembling four ingredients whose proofs determine the structure of the paper:
\begin{enumerate}
    \item \textit{2-step nilpotent PET induction scheme}, laid out in Section \ref{S: PET}, with Section \ref{S: PET n, n^2} presenting it for the model average
    \begin{align}\label{E: n, n^2 0}
    \E_{n\in\Z}T^n f_1\cdot S^{n^2}f_2;
\end{align}
    \item new \textit{abelian seminorm estimates} (Theorem~\ref{T: new estimates}), derived in Section \ref{S: seminorm estimates}, with the key ingredient of seminorm control transfer isolated in Section \ref{S: control transfer}; 
    \item \textit{comparison results for Host-Kra factors}, whose proofs occupy Section \ref{S: comparison of factors};
    \item \textit{equidistribution of polynomial orbits on nilsystems} (Theorem~\ref{T: equidistribution}), proved in Section \ref{S: equidistribution}.
\end{enumerate}
When combined, the first two of these ingredients yield Theorem~\ref{T: nilpotent seminorm estimates} while the latter two result in Theorem~\ref{T: nilpotent joint ergodicity}.

In addition, Section \ref{S: seminorms and factors} summarizes the properties of box seminorms and factors used in the paper; Section \ref{SS: counterexamples} provides counterexamples announced in Theorems \ref{T: counterexample to joint ergodicity conjecture} and \ref{T: counterexample to joint ergodicity criteria}; Section \ref{S: proofs of main results} concludes the proofs of joint ergodicity results; and Section \ref{S: open problems} contains a broad array of open problems naturally arising from this work. Three appendices contain all the needed properties of invariant and Kronecker factors (Appendix \ref{A: Kronecker}), nilsystems (Appendix \ref{A: nilsystems}), and cubic structures (Appendix \ref{A: cubes}). 

\subsection{2-step nilpotent PET}
Here is the high-level overview of how the abovementioned four main ingredients come together to give Theorem~\ref{T: nilpotent joint ergodicity}. Like in the abelian setting, the 2-step nilpotent PET induction scheme is a complexity reduction argument that reduces the original average to a simpler one via a sequence of van der Corput operations. The general strategy behind our variant resembles that of Chu, Frantzikinakis and Host \cite{CFH11} in the commuting case. That is, we eliminate  $T_1, \ldots, T_{\ell-1}$ in finitely many van der Corput operations by exploiting the lower degree of their iterates relative to the iterate of $T_\ell$. However, our approach differs at several critical junctures due to challenges brought by noncommutativity.

As we run the nilpotent PET, we inevitably need to swap transformations with each other, which gives rise to commutators, an issue obviously not present in \cite{CFH11}. A major challenge is that while the group operation is linear in $T_j$'s, it is quadratic in commutators. For instance, if our group is 2-step nilpotent and generated by $T, S$, then 
$$(T^{a_1}S^{b_1}[T,S]^{c_1})(T^{a_2}S^{b_2}[T,S]^{c_2}) = T^{a_1 + a_2}S^{b_1 + b_2}[T,S]^{c_1 + c_2 - a_2 b_1},$$
and so the iterates of $T, S$ are linear while those of $[T, S]$ are quadratic in terms of the original ones.\footnote{Here is a technical explanation for why this matters, which the reader is advised to skip on the first read. The iterates of $T_1, \ldots, T_\ell$ change linearly at each step of PET, and since compositions of linear maps are linear, it is possible to track how they change throughout the entire PET. This observation underpins the \textit{coefficient tracking methods} used in dealing with ergodic averages in the abelian setting. By contrast, tracking the iterates of the commutators across successive iterations is rather difficult; the reason is essentially the fact that quadratic maps do not behave nicely under compositions.} 
As a consequence, we can quite well control the iterates of $T_1, \ldots, T_\ell$ in the successive stages of PET (just like in the abelian case), but our grip on the new powers of the commutators is less firm. This is a serious stumbling block, as all recent seminorm estimates obtained in the abelian setting (e.g. \cite{DFKS22, DKKST24, KKL24a, Kuc23}) rely heavily on careful tracking of the iterates from one stage of PET to another.  To overcome this difficulty, we identify  a property of polynomial families called \textit{distinguishability} (see Definition \ref{D: nice families}) that remains preserved throughout the iterations of PET, and which
gives us just enough control over the shape of the intermediate averages to continue with the argument.

In the abelian setting, the PET induction scheme concludes once we arrive at an average with all iterates linear in the variable $n$. This is emphatically \textit{not} what we do in the 2-step nilpotent setting. Here, we stop as soon as we annihilate all instances of $T_1, \ldots, T_{\ell-1}$ and are left only with $T_\ell$ and the commutators - even if the powers of the latter are not linear in $n$. The reason is that, with $T_1, \ldots, T_{\ell-1}$ out of the way, we get an average which involves only \textit{commuting} transformations. This is a key point where we use the 2-step nilpotence - without it, the commutators would not commute with $T_\ell$. Upon reaching this stage, we apply the new seminorm estimates from Theorem~\ref{T: new estimates}. Strikingly, we can ensure that for this new average, the groups of coefficients appearing in Theorem~\ref{T: new estimates} always contain a power of $T_\ell$. This crucial and highly nonobvious observation gives us Theorem~\ref{T: nilpotent seminorm estimates}, our nilpotent seminorm estimates, when combined with standard properties of box seminorms.

{We emphasize that in the proof of Theorem~\ref{T: nilpotent seminorm estimates}, we invoke Theorem~\ref{T: new estimates} for genuinely multidimensional and multiparameter polynomial families. Indeed, in applying Theorem~\ref{T: new estimates}, we treat the variables $h_1, \ldots, h_r$ arising from the van der Corput inequality on a par with the original variable $n$.}

For the illustration of the aforementioned steps on a specific example, we refer the reader to Section \ref{S: PET n, n^2}, which in its entirety is dedicated to the proof of Theorem~\ref{T: new estimates} for the model average \eqref{E: n, n^2 0}.

\subsection{Abelian seminorm estimates}
How, then, do we go about proving Theorem~\ref{T: new estimates}? We broadly follow the seminorm smoothing strategy first laid out in \cite{FrKu22a} and then refined in \cite{DKKST25, DKKST24, FrKu22b, Kuc23}. In all previous papers, this strategy deals with averages of the form
\begin{align}\label{E: average smoothing}
    \E_{n\in\Z}T_1^{p_1(n)}f_1\cdots T_\ell^{p_\ell(n)}f_\ell,
\end{align}
involving \textit{single-dimensional} polynomials $p_j\in \Z[n]$. The game there is about reducing the average \eqref{E: average smoothing} to one that is less complex in the sense of involving fewer transformations. The end goal is to arrive at an average which sees only one of $T_1, \ldots, T_\ell$, use the single-transformation estimates for that average, and lift them all the way back to \eqref{E: average smoothing}. 

In the framework of Theorem~\ref{T: new estimates}, the task is more delicate because we need to consider averages
\begin{align}\label{E: average smoothing 2}
    \E_{\bn\in\Z^{L}}U_{p_1(\bn)}f_1\cdots U_{p_\ell(\bn)}f_\ell
\end{align}
for \textit{multidimensional} and \textit{multiparameter} polynomials $p_1, \ldots, p_\ell\in\Z^D[\bn]$. The key challenges are to identify a correct way of measuring the complexity of averages \eqref{E: average smoothing 2}; a reduction that allows us to reduce a more complex average to one of lower complexity; and the base case in which the estimates of Theorem~\ref{T: new estimates} can be deduced from some known estimates. None of these three challenges suggest a natural answer, as the intuitions developed for single-dimensional polynomials fail quite dramatically in the multidimensional setup. Eventually, we manage to address all three of them, and the known estimates used in two base cases are provided by Theorem~\ref{T: original estimates} and the spectral theorem. A (relatively simple) example of such reduction is given in Section \ref{SS: seminorm example}. 

While the original motivation for Theorem~\ref{T: new estimates} comes from its role in the derivation of Theorem~\ref{T: nilpotent seminorm estimates}, the former is also a key input in the proof of Theorem~\ref{T: joint ergodicity conjecture} presented in Section \ref{S: proofs of main results}.
\subsection{Failure of degree lowering}
In the abelian setting, the modern way to derive a joint ergodicity result like Theorem~\ref{T: nilpotent joint ergodicity} from Host-Kra seminorm estimates like Theorem~\ref{T: nilpotent seminorm estimates} would be to apply joint ergodicity criteria from \cite{BFM22, Fr21, FrKu22a}. The latter are based on the degree lowering argument developed in additive combinatorics by Peluse \cite{Pel19} and Peluse-Prendiville \cite{PP19}. Perhaps the most exasperating aspect of the nilpotent universe is that neither degree lowering nor seminorm smoothing works any longer.\footnote{For clarity: we mentioned above a use of seminorm smoothing, but in the \textit{abelian} setting, as a part of the proof of Theorem~\ref{T: new estimates}.} The place where things break is the dual-difference interchange step; we briefly illustrate this failure now. Suppose that we study \eqref{E: n, n^2 0} for 2-step nilpotent $\langle T, S\rangle$, and for simplicity assume that $S$ is ergodic. Consider the corresponding \textit{dual function}\footnote{Recall \eqref{E: composition convention} for our convention regarding compositions: $(TSf)(x) := f(TSx) = (S(Tf))(x).$}
\begin{align*}
    F := \E_{n\in\Z} S^{-n^2}\bar f_0\cdot T^n S^{-n^2}\bar f_1,
\end{align*}
and suppose that we want to deduce $\nnorm{F}_{2,S}>0$ from $\nnorm{F}_{3, S}>0$ using the degree lowering strategy. On expanding and using the inverse theorem for degree-2 Host-Kra seminorms, we get
\begin{align*}
    \nnorm{F}_{3,S}^8 = \E_{h\in\Z}\nnorm{F\cdot S^h\bar F}_{2,S}^4\leq \oE_{h\in\Z}\int F\cdot S^h\bar F\cdot \chi_h\; d\mu
\end{align*}
for some $S$-eigenfunctions $\chi_h$. Expanding the definition of the dual function and using the assumption $\nnorm{F}_{3, S}>0$, we get
\begin{align*}
    \oE_{h\in\Z}\int F\cdot S^h\brac{\E_{n\in\Z} S^{-n^2}f_0\cdot T^n S^{-n^2}f_1}\cdot \chi_h\; d\mu > 0.
\end{align*}
Now, composing the integral with $S^{n^2}$ gives us
\begin{align*}
    \oE_{h\in\Z} \E_{n\in\Z}\int S^hf_0\cdot T^n S^h f_1 \cdot S^{n^2}(F\cdot \chi_h)\; d\mu > 0.
\end{align*}
By swapping the order of averages,\footnote{This maneuver can be made precise by swapping only finite portions of averages and keeping the limits in their current order. A careful check shows that this is correct.} applying the Cauchy-Schwarz inequality to double $h$, and swapping the order of averages again, we get
\begin{align*}
    \oE_{h,h'\in\Z} \E_{n\in\Z}\int (S^hf_0\cdot S^{h'}\bar f_0)\cdot (T^n S^h f_1\cdot T^n S^{h'} \bar f_1) \cdot S^{n^2}(\chi_h\cdot \bar\chi_{h'})\; d\mu > 0;
\end{align*}
{note that $F$ disappears as it does not depend on $h$.}
This can be rearranged as
\begin{align}\label{E: failure of degree lowering}
    \oE_{h,h'\in\Z} \E_{n\in\Z}\int \Delta'_{S; (h,h')}f_0 \cdot \Delta'_{S; (h,h')}(T^n f_1) \cdot S^{n^2}(\chi_h\cdot \bar\chi_{h'})\; d\mu > 0
\end{align}
on setting $\Delta'_{S; (h,h')}f:=S^h f\cdot S^{h'}\bar f$ (a note of caution: $S^h(T^nf) = T^nS^hf$, not the other way, since $(S^h(T^nf))(x) = (T^n f)(S^hx) = f(T^nS^h x)=(T^nS^hf)(x)$). Unfortunately, 
\begin{align*}
    \Delta'_{S; (h,h')}(T^n f_1) = T^n S^h f_1\cdot T^n S^{h'} \bar f_1
\end{align*}
is \textit{not} the same as
\begin{align*}
    T^n(\Delta'_{S; (h,h')} f_1) = S^h T^n f_1 \cdot S^{h'} T^n \bar f_1
\end{align*}
if $T, S$ fail to commute. Hence, the average \eqref{E: failure of degree lowering} is not of the form
\begin{align*}
    \oE_{h,h'\in\Z} \E_{n\in\Z}\int \Delta'_{S; (h,h')}f_0 \cdot T^n(\Delta'_{S; (h,h')} f_1) \cdot S^{n^2}(\chi_h\cdot \bar\chi_{h'})\; d\mu > 0
\end{align*}
that we would need it to be in order to run degree lowering.

\subsection{Host-Kra factors and nilsystems}
In the absence of degree lowering, we revert to the old-school strategy of reducing joint ergodicity to an equidistribution problem on nilsystems. In the abelian case, this would be pretty straightforward: we would use the ergodicity of $T, S$ to conclude that
\begin{align}\label{E: failing factor identity}
\CZ_s(T)=\CZ_s(S) = \CZ_s(\langle T, S\rangle)    
\end{align}
 for every $s\in\N_0$, apply the structure theorem for the ergodic Host-Kra factor $\CZ_s(\langle T, S\rangle)$, and then use the $L^2(\mu)$ approximation argument to pass to a nilsystem. To our surprise, an identity like \eqref{E: failing factor identity} seems shockingly difficult to establish in the nilpotent setting. 
 To overcome this difficulty,  we prove  the following weaker version of (\ref{E: failing factor identity}) instead:
 \begin{align}\label{E: working factor identity}
 \CZ_s(T), \CZ_s(S)\subseteq \CZ_{2s^2}(\langle T, S\rangle),    
 \end{align}
 which is a special case of a more general Proposition~\ref{P: comparing HK factors in nilpotent systems}. Curiously, we do not know how to extend this identity to higher-step nilpotent systems, nor whether the degree $2s^2$ can be lowered (see Problem~\ref{Pr: comparing factors}). However, the identity \eqref{E: working factor identity} is good enough for our purposes, as it allows us to pass into nilsystems when applied in conjunction with the Candela-Szegedy structure theorem (Theorem~\ref{T: Candela-Szegedy}). The proof of \eqref{E: working factor identity} is substantially more difficult than that of \eqref{E: failing factor identity} in the abelian world. Among other things, it uses an old result of Parry that a measurable isomorphism between two nilsystems is necessarily affine as well as some spectral-theoretic arguments collected in Appendix \ref{A: Kronecker}.

The last ingredient used in the proof of Theorem~\ref{T: nilpotent joint ergodicity} is Theorem~\ref{T: equidistribution}, an equidistribution result of nilsystems. Its proof emulates the strategy employed by Frantzikinakis and Kra \cite{FrKr05} in the single transformation case $T_1 = \cdots = T_\ell$ - but with some interesting twists explained in Section \ref{S: equidistribution}.

    \section{Seminorms and factors for nilpotent systems} \label{S: seminorms and factors}
Throughout this section, we fix a countable amenable group $H$ and an \textit{$H$-system} $(X, \CX, \mu, U)$. That is, $U=(U_h)_{h\in H}$ is a measure-preserving action of $H$ on $(X,\CX,\mu)$. Our goal is to introduce box seminorms and factors along subgroups $H_1, \ldots, H_s\subseteq H$ and summarize their properties in two cases of special interest:
\begin{enumerate}
    \item when $H$ is abelian;
    \item when $H$ is nilpotent and $H_1 = \cdots = H_s$.
\end{enumerate}
While these two settings are all that we require in this paper, the only property of nilpotent groups needed to define box seminorms is that they satisfy the mean ergodic theorem. Since the latter holds for any amenable group, amenability is the right level of generality to present the basic theory of cubic measures as well as box seminorms and factors.\footnote{In the absence of amenability, one can define averaging as in \cite[Section 1.4]{TZ16}. We do not pursue this level of abstraction here.} The assumption of countability is made purely for notational convenience.

The definitions and proofs in this section closely follow those already present in the vast literature on the subject. However, there are certain nuances coming from the nonabelian setting, and we shall take particular care to point those out. In particular, we caution the reader that since $(U_hf)(x):=f(U_h x)$, we have $U_{hh'}f = U_{h'}(U_h f)$ for any $f\in L^2(\mu)$ and $h,h'\in H$ as 
\begin{align}\label{E: composition}
(U_{hh'}f)(x)=f(U_{hh'}x) = f(U_h U_{h'}x) = (U_{h'}(U_{h}f))(x). 
\end{align}

\subsection{General theory}\label{SS: general theory}
\subsubsection{Cubic measures}

Following \cite{H09, HK05a, HK18}, we start by introducing a family of measures on cubic product spaces.
\begin{definition}[Cubic measures]\label{D: cubic measures}
Let $X^{\cube{s}}:=X^{2^s}$. The \textit{cubic measure on $X^{\cube{s}}$ with respect to $H_1, \ldots, H_s$} is the measure $\mu_{H_1, \ldots, H_s}$ constructed iteratively by setting $\mu_{\emptyset}:=\mu$ and defining 
\begin{align*}
    \mu_{H_1, \ldots, H_s} := \mu_{H_1, \ldots, H_{s-1}}\times_{\CI(H_{s})} \mu_{H_1, \ldots, H_{s-1}}
\end{align*}
for $s\geq 1$,
where the induced action of $H_s$ on $X^{\cube{s-1}}$ is the diagonal action $$(U^\cube{s-1})_h:=(U_h)^\cube{s-1}.$$  The invariant factor $\CI(H_s)$ is defined in Definition \ref{D: invariant}.

If $H=H_1 = \cdots = H_s$, we also denote the corresponding measure on $X^{\cube{s}}$ via $\mu_{s,H}$ (with $\mu_{0,H} = \mu$). If instead we want to emphasize the dependence of the cubic measure on the action $U$, we write $\mu_{s,U} = \mu_{s,H}$.
\end{definition}

By the mean ergodic theorem for amenable actions (see e.g. \cite[Theorem~3.33]{Glasner}), for every $f\in L^2(\mu)$ we have
\begin{align*}
    \E_{h\in H}U_h f = \E_{h,h'\in H}U_{h{h'}\inv} f = \E(f|\CI(H))
\end{align*}
{(we recall that all the limits are in $L^2(\mu)$ unless stated otherwise).}
Iterating this result, we get a dynamical description of the cubic measures.
\begin{lemma}\label{L: cubic measures using averages}
Let $f_\eps\in L^\infty(\mu)$ for every $\eps\in{\{0,1\}^s}$. Then 
    \begin{align*}
        \E_{h_s,h_s'\in H_s}\cdots \E_{h_1,h_1'\in H_1}\prod_{\eps\in\{0,1\}^s}U_{h_{s,\eps_s}\cdots h_{1,\eps_1}}f_\eps = \int \bigotimes_{\eps\in\{0,1\}^s}f_\eps\; d\mu_{H_1, \ldots, H_s},
    \end{align*}
    where $h_{i,0} := h_i$, $h_{i,1}:=h_i'$.
\end{lemma}
\begin{proof}
While the proof is standard, we sketch it nonetheless to emphasize that the order of $h_i$'s in the transformation $U$ is compatible with the order of averages.
    For $s=1$, the mean ergodic theorem implies
    \begin{align*}
        \E_{h_1,h_1'\in H_1} U_{h_1}f_0\cdot  U_{h_1'}f_1 = \int \E(f_0|\CI(H_1))\cdot \E(f_1|\CI(H_1))\; d\mu = \int f_0 \otimes f_1\; d\mu_{H_1}. 
    \end{align*}
    The case $s>1$ follows by induction; to make the notation palatable, we just show the proof for $s=2$. Recasting
    \begin{align*}
        U_{h_{2,\eps_2}h_{1,\eps_1}}f_\eps = U_{h_{1,\eps_1}}(U_{h_{2,\eps_2}}f),
    \end{align*}
    using \eqref{E: composition},
    we apply the inductive hypothesis to get
        \begin{multline*}
        \E_{h_2,h'_2\in H_2}\E_{h_1,h_1'\in H_1} U_{h_2 h_1}f_{00}\cdot U_{h_2 h_1'}f_{10}\cdot U_{h_2' h_1}f_{10}\cdot U_{h_2' h_1'}f_{11}\\ = \E_{h_2,h'_2\in H_2}\E_{h_1,h_1'\in H_1} U_{h_1}(U_{h_2}f_{00}\cdot U_{h_2'}f_{01})\cdot U_{h_1'}(U_{h_2}f_{10}\cdot U_{h_2'}f_{11}).
    \end{multline*}
    By the case $s=1$, this equals
    \begin{multline*}
        \E_{h_2,h'_2\in H_2} \int (U_{h_2}f_{00}\cdot U_{h_2'}f_{01})\otimes (U_{h_2}f_{10}\cdot U_{h_2'}f_{11})\; d\mu_{H_1}\\ = \E_{h_2,h'_2\in H_2}\int (U\times U)_{h_2}(f_{00}\otimes f_{10})\cdot (U\times U)_{h_2'}(f_{01}\otimes f_{11})\; d\mu_{H_1}.
    \end{multline*}
    The claim then follows from one more application of the mean ergodic theorem and the definition of $\mu_{H_1, H_2}$.
\end{proof}

\subsubsection{Box seminorms}
With cubic measures {already defined,} we proceed to define box seminorms. In doing so, we shall need the following standard notion.
\begin{definition}[Symmetric multiplicative derivatives]
    Given $h,h'\in H$ and $f\in L^\infty(\mu)$, the \textit{symmetric multiplicative derivative} of $f$ along $(h,h')$ is $\Delta'_{(h,h')}f:=U_h f\cdot U_{h'}\bar f$.
\end{definition}
We interpret {the} iterated multiplicative derivative as
\begin{align*}
    \Delta'_{(h_1,h_1')}\cdots \Delta'_{(h_s,h'_s)}f:=\Delta'_{(h_1,h_1')}(\Delta'_{(h_2,h'_2)}(\cdots (\Delta_{(h_s,h'_s)}' f)\cdots)),
\end{align*}
so for instance
\begin{align*}
    \Delta'_{(h_1,h_1')}\Delta'_{(h_2,h_2')}f = \Delta'_{(h_1,h_1')}(U_{h_2}f\cdot U_{h'_2}\overline{f}) = U_{h_2 h_1}f\cdot U_{h_2 h_1'}\overline{f}\cdot U_{h_2' h_1}\overline{f}\cdot U_{h_2' h_1'}f.
\end{align*}

Using cubic measures and iterated multiplicative derivatives, we can define the box seminorm on $L^\infty(\mu)$ as follows. The equivalence of both formulas comes from Lemma~\ref{L: cubic measures using averages}.

\begin{definition}[Box and Host-Kra seminorms]\label{D: box seminorms}
    Let $f\in L^\infty(\mu)$. The \textit{box seminorm} of $f$ along $H_1, \ldots, H_s$ is
    \begin{align*}
        \nnorm{f}_{H_1, \ldots, H_s} :=&\int \bigotimes_{\eps\in\{0,1\}^s}\CC^{|\eps|}f\; d\mu_{H_1, \ldots, H_s}\\
    =&\brac{\E_{h_s,h_s'\in H_s}\cdots \E_{h_1,h_1'\in H_1}\int \Delta_{(h_1,h_1')}'\cdots \Delta'_{(h_s,h_s')}f\; d\mu}^{1/2^s}.
    \end{align*}
{If $H:=H_1 = \cdots = H_s$, then we call $\nnorm{\cdot}_{s,H}:=\nnorm{\cdot}_{H_1, \ldots, H_s}$ the \textit{degree-$s$ Host-Kra seminorm} along $H$.} 
\end{definition}

As another indication that the relative order of taking derivatives and averages is correct, we recover the standard inductive formula for box seminorms:
\begin{align}\label{E: inductive formula}
    \nnorm{f}_{H_1, \ldots, H_s}^{2^s} = \E_{h_s,h'_s\in H_s}\cdots \E_{h_{j+1},h'_{j+1}\in H_{j+1}}\nnorm{\Delta'_{(h_{j+1},h_{j+1}')}\cdots \Delta'_{(h_s,h'_s)}f}_{H_1, \ldots, H_j}^{2^j}.
\end{align}

Via standard arguments, our seminorms and measures satisfy the Gowers-Cauchy-Schwarz inequality.
\begin{proposition}[Gowers-Cauchy-Schwarz inequality]
    Let $f_\eps\in L^\infty(\mu)$ for every $\eps\in\{0,1\}^s$. Then
    \begin{align*}
        \abs{\int \bigotimes_{\eps\in\{0,1\}^s}f_\eps\; d\mu_{H_1, \ldots, H_s}}\leq \prod_{\eps\in\{0,1\}^s}\nnorm{f_\eps}_{H_1, \ldots, H_s}.
    \end{align*}
\end{proposition}

One consequence is the standard monotonicity property of box seminorms.
\begin{corollary}[Monotonicity property]
    Let $f\in L^\infty(\mu)$. Then
    \begin{align*}
        \nnorm{f}_{H_1}\leq \nnorm{f}_{H_1, H_2}\leq \cdots \leq \nnorm{f}_{H_1, \ldots, H_s}.
    \end{align*}
\end{corollary}

\subsubsection{Dual functions and structured algebras}
To define factors from the seminorms, we will need the following class of ``structured'' functions.
\begin{definition}[Dual functions]
    Let $\mathfrak{f} = (f_\eps)_\eps\subseteq (L^\infty(\mu))^{\{0,1\}^s_*}$. We define the \textit{dual function} of $\mathfrak{f}$ along $H_1, \ldots, H_s$ via
    \begin{align*}
        \CD_{H_1, \ldots, H_s}\mathfrak{f}:= \E_{h_s,h_s'\in H_s}\cdots \E_{h_1,h_1'\in H_1}U_{h_s\cdots h_1}\inv\brac{\prod_{\eps\in\{0,1\}^s_*}U_{h_{s,\eps_s}\cdots h_{1,\eps_1}}f_\eps}.
    \end{align*}
    If $f_\eps = \CC^{|\eps|}f$ for all $\eps\in \{0,1\}^s_*$, 
    then we write $\CD_{H_1, \ldots, H_s}\mathfrak{f} = \CD_{H_1, \ldots, H_s}f$ whereas if $H:=H_1 = \cdots = H_s$, we write $\CD_{H_1, \ldots, H_s}\mathfrak{f} = \CD_{s, H}\mathfrak f$.
\end{definition}
A priori, it is not clear that the iterated limit in the definition of dual functions exists. This follows from the result below whose proof is identical to the one of \cite[Chapter 8, Proposition~25]{HK18}.
\begin{proposition}
    Let $\mathfrak{f} = (f_\eps)_\eps\subseteq (L^\infty(\mu))^{\{0,1\}^s_*}$. The iterated $L^2(\mu)$-limit in the definition of $\CD_{H_1, \ldots, H_s}\mathfrak{f}$ exists.
\end{proposition}

The next result shows that we have some flexibility in how we take the limit in the definition of dual functions.
\begin{proposition}\label{P: iterated limits}
     Let $\mathfrak{f} = (f_\eps)_\eps\subseteq (L^\infty(\mu))^{\{0,1\}^s_*}$. 
    Then     \begin{align}\label{E: dual as single limit}
        \CD_{H_1, \ldots, H_s}\mathfrak{f}= \E_{\substack{h,h'\in H_1 \times \cdots \times H_s}}U_{h_s\cdots h_1}\inv\brac{\prod_{\eps\in\{0,1\}^s_*}U_{h_{s,\eps_s}\cdots h_{1,\eps_1}}f_\eps}.
    \end{align}
\end{proposition}
\begin{proof}
Observe that for every index $1\leq j\leq s$ and group elements $h_{j+1}, h_{j+1}'\in H_{j+1},$ ..., $h_s,h'_s\in H_s$, the expression
    \begin{multline*}
        \E_{h_j,h_j'\in H_j}\cdots \E_{h_1,h_1'\in H_1}U_{h_s\cdots h_1}\inv\brac{\prod_{\eps\in\{0,1\}^s_*}U_{h_{s,\eps_s}\cdots h_{1,\eps_1}}f_\eps}\\
        = U_{h_s\cdots h_{j+1}}\inv \brac{\E_{h_j,h_j'\in H_j}\cdots \E_{h_1,h_1'\in H_1}U_{h_j\cdots h_1}\inv\brac{\prod_{\eps'\in\{0,1\}^j_*}U_{h_{j,\eps'_j}\cdots h_{1,\eps'_1}}f_{\eps'}'}}
    \end{multline*}
    with 
    \begin{align*}
        f_{\eps'}' := \prod_{\substack{\eps\in \{0,1\}^s_*\colon\\ \eps_1 = \eps_1',\ldots, \eps_j = \eps_j'}}U_{h_{s,\eps_s}\cdots h_{j+1,\eps_{j+1}}}f_\eps
    \end{align*}
    is a well-defined dual function, so in particular it is a uniform limit. We then obtain the formula \eqref{E: dual as single limit} by iteratively applying \cite[Theorem~1.1]{ZK16}, a Fubini-type theorem for averages. 
\end{proof}
\begin{corollary}
    In the definition of the seminorm $\nnorm{\cdot}_{H_1, \ldots, H_s}$ on $L^\infty(\mu)$, we can replace 
    \begin{align}\label{E: iterated limit}
    \E_{h_s,h_s'\in H_s}\cdots \E_{h_1,h_1'\in H_1}
\end{align}
by the single limit
\begin{align}\label{E: single limit}
\E_{h,h'\in H_1 \times \cdots \times H_s}.
\end{align}
\end{corollary}

To box seminorms we can associate factors in the standard way.
\begin{definition}
    For a sub $\sigma$-algebra $\CY$ of $\CX$, let $Z_{H_1, \ldots, H_s}(\CY)$ be the $L^2(\mu)$-closed linear span of functions $\CD_{H_1, \ldots, H_s}\mathfrak f$ for $\CY$-measurable $\mathfrak{f}\subseteq (L^\infty(\mu))^{\{0,1\}^s_*}$.
    
    If $\CY = \CX$, we simply denote $Z_{H_1, \ldots, H_s}(\CX)$ by $Z_{H_1, \ldots, H_s}$ or $Z(H_1, \ldots, H_s)$, and if $H:=H_1 = \cdots = H_s$, we also use the abbreviations $Z_{s-1,H}(\CY):=Z_{H_1, \ldots, H_s}(\CY)$ and 
    $Z_{s-1}(H)=Z_{s-1,H}=Z_{s-1,H}(\CX)$. 

    If we want to specify which $H$-action $U$ we are referring to, we also write $Z_{s-1, U} = \CZ_{s-1}(U)$.
\end{definition}

We recover the usual relationship between the seminorm $\nnorm{\cdot}_{H_1, \ldots, H_s}$ and the space $Z_{H_1, \ldots, H_s}$.
\begin{proposition}\label{P: seminorms vs. algebras}
    Let $f\in L^\infty(\mu)$. Then
    \begin{align*}
        \nnorm{f}_{H_1, \ldots, H_s} = 0\quad \iff \quad {P_{Z(H_1, \ldots, H_s)}f= 0,}
    \end{align*}
    where $P_{Z(H_1, \ldots, H_s)}$ is the orthogonal projection on $Z(H_1, \ldots, H_s).$
\end{proposition}
\begin{proof}
    By the definition of $Z(H_1, \ldots, H_s)$, the function $f$ is orthogonal to $Z(H_1, \ldots, H_s)$ if and only if it is orthogonal to $\CD_{H_1,\ldots, H_s}\mathfrak{f}$ for all $\mathfrak{f}\subseteq (L^\infty(\mu))^{\{0,1\}^s_*}$.
    The latter assertion is then equivalent to the vanishing of $\nnorm{f}_{H_1, \ldots, H_s}$, where one direction follows from taking correlations with $\CD_{H_1,\ldots, H_s}f$ and the other is a consequence of the Gowers-Cauchy-Schwarz inequality.
\end{proof}

To the space $Z_{H_1, \ldots, H_s}(\CY)$ we can associate a factor.

\begin{definition}[Box and Host-Kra factors]
    For $\CY\subseteq \CX$, let the \textit{box factor} $\CZ_{H_1, \ldots, H_s}(\CY)$ be the smallest $\sigma$-algebra with respect to which the elements of $Z_{H_1, \ldots, H_s}(\CY)$ are measurable. 

    We use the labels $\CZ_{H_1, \ldots, H_s}$, $\CZ(H_1, \ldots, H_s)$, $\CZ_{s-1,H}(\CY)$, $\CZ_{s-1}(H)$, $\CZ_{s-1,H}$, $\CZ_{s-1}(U)$, $\CZ(s-1,U)$  in analogy with the labels for the spaces $Z$. {We also refer to the factor $\CZ_{s-1}(H)$ as the \textit{degree-$(s-1)$ Host-Kra factor} along $H$.}
\end{definition}

Whenever the ambient group $H$ is abelian, $Z_{H_1, \ldots, H_s}(\CY)$ is known to be an algebra, i.e. it is closed under multiplication (see e.g. \cite[Proposition~2.3]{TZ16}). Then
\begin{align}\label{E: algebra vs. factor}
    Z_{H_1, \ldots, H_s}(\CY) = L^2(\CZ_{H_1, \ldots, H_s}(\CY)).
\end{align}
The same is true whenever $H_1 = \cdots = H_s$ is a finitely generated nilpotent group \cite[Lemma~5.8]{CS23}. As a consequence, we obtain a familiar correspondence between factors and seminorms
\begin{align}\label{E: seminorms vs. factors}
     \E(f|\CZ(H_1, \ldots, H_s)) = 0\quad \iff \quad \nnorm{f}_{H_1, \ldots, H_s} = 0.
\end{align}
It is not clear if \eqref{E: algebra vs. factor} and \eqref{E: seminorms vs. factors} hold for nonabelian groups in general; see Problem~\ref{Pr: nonobvious direction}. Outside of these two special cases, all we know is that
\begin{align}\label{E: seminorms vs. factors nonabelian}
        \E(f|\CZ(H_1, \ldots, H_s)) = 0\quad \Longrightarrow \quad \nnorm{f}_{H_1, \ldots, H_s} = 0.
\end{align}
In the abelian context, the missing direction \eqref{E: seminorms vs. factors nonabelian} can be proved either via the identity \eqref{E: algebra vs. factor} or by the invariance of the cubic measure $\mu_{H_1, \ldots, H_s}$ under the group of face transformations or the identity \eqref{E: algebra vs. factor}. Whether any of these two properties hold more generally is unclear.

\subsubsection{Further remarks on notations}
The variety of results and settings considered in this paper forces us to be as flexible as possible with the notations. Thus, for instance, whenever the action $U$ of a group $H$ is generated by transformations $T_1, \ldots, T_\ell$, then we use the labels
\begin{align*}
 \nnorm{f}_{s, T_j},\; \nnorm{f}_{T_j^{\times s}},\; \nnorm{f}_{s, e_j},\; \nnorm{f}_{e_j^{\times s}}   
\end{align*}
to denote the same object, depending on which choice of notation is most handy at a given moment.

\subsection{The commuting case}\label{SS: commuting}

In this subsection, we assume that the ambient group $H$ is abelian. In the commuting case, box seminorms enjoy a number of properties that we shall use profusely throughout the paper. First, we can replace the symmetric multiplicative derivative by its asymmetric version.
\begin{definition}[Asymmetric multiplicative derivatives]
    Given $h\in H$ and $f\in L^\infty(\mu)$, the \textit{multiplicative derivative} of $f$ along $h$ is $\Delta_h f:=f\cdot U_{h}\bar f$.
\end{definition}
Since $H$ is abelian, the order of taking derivatives iteratively does not matter, and so the simpler notation
\begin{align*}
    \Delta_{h_1, \ldots, h_s}f := \Delta_{h_1}\cdots \Delta_{h_s}f
\end{align*}
causes no ambiguity. For subgroups $H_1, \ldots, H_s\subseteq H$, the definition of the box seminorm takes the following simpler form:
    \begin{align*}
        \nnorm{f}_{H_1, \ldots, H_s} =&\brac{\E_{h_s\in H_s}\cdots \E_{h_1\in H_1}\int \Delta_{h_1, \ldots, h_s}f\; d\mu}^{1/2^s}.
    \end{align*}
On top of the properties listed in Section \ref{SS: general theory}, box seminorms enjoy a number of well-established properties in the abelian setting:
\begin{enumerate}
    \item (Permutation invariance) For
    any permutation $\sigma:[s]\to[s]$,  let
    \begin{align*}
        \nnorm{f}_{H_1, \ldots, H_s} = \nnorm{f}_{H_{\sigma(1)},\ldots, H_{\sigma(s)}}.
    \end{align*}
    \item \label{I:scaling} (Subgroup properties) For any subgroups $H_1'\subseteq H_1, \ldots, H_s'\subseteq H_s$,  
    \begin{align*}
        \nnorm{f}_{H_1, \ldots, H_s}\leq \nnorm{f}_{H_1', \ldots, H_s'}.
    \end{align*}
    Conversely, if $s\geq 2$ and $r_i:=|H_i/H'_i|<\infty$, then
    \begin{align*}
        \nnorm{f}_{H_1', \ldots, H_s'}\leq |r_1\cdots r_s|^{1/2^s}\nnorm{f}_{H_1, \ldots, H_s}.
    \end{align*}
    \item If $\CI(H_i') = \CI(H_i)$ for $1\leq i\leq s$, then
    \begin{align}\label{E: same invariant factors}
        \nnorm{f}_{H_1, \ldots, H_s} = \nnorm{f}_{H_1',\ldots, H_s'}.
    \end{align}
\end{enumerate}
 The proof of these properties can be found e.g. in \cite{DKKST24, FrKu22b, H09}. It is unclear whether their variants hold in the noncommutative universe, as the proofs rely crucially on being able to swap the order of taking averages and applying multiplicative derivative. 

Furthermore, the subspace $Z_{H_1, \ldots, H_s}(\CY)$ is an algebra \cite[Proposition~2.3]{TZ16}, and hence \eqref{E: algebra vs. factor} and \eqref{E: seminorms vs. factors} hold.

 We shall also need the following concatenation result proved by the authors jointly with Donoso and Tsinas in \cite{DKKST25} and generalizing an earlier result of Tao and Ziegler \cite{TZ16}.
\begin{theorem}[Relative concatenation, {\cite[Theorem~3.1, Corollary 3.2]{DKKST25}}]\label{T: relative concatenation}
    For any subgroups $H_1, \ldots, H_r, G_1, \ldots, G_s, G_1', \ldots, G_{s'}'\subseteq \Z^D$, we have
\begin{multline*}
    \CZ(H_1, \ldots, H_r, G_1, \ldots, G_s) \cap \CZ(H_1, \ldots, H_r, G_1', \ldots, G_{s'}')\\
    \subseteq \CZ(H_1, \ldots, H_r, G_i + G_j'\colon\; 1\leq i\leq s,\; 1\leq j\leq s').
\end{multline*}
In particular, if
\begin{align*}
    \E(f|\CZ(H_1, \ldots, H_r, G_i + G_j'\colon\; 1\leq i\leq s,\; 1\leq j\leq s')) = 0,
\end{align*}
then we can decompose $f = f_1 + f_2$, where
\begin{align*}
    \E(f_1|\CZ(H_1, \ldots, H_r, G_1, \ldots, G_s)) = \E(f_2|\CZ(H_1, \ldots, H_r, G_1', \ldots, G_{s'}')) = 0.
\end{align*}
\end{theorem}
To clarify the notation,
\begin{align*}
    \CZ(H, G_i + G_j'\colon\; 1\leq i,j\leq 2) = \CZ(H,G_1+G_1',G_2+G_1',G_1+G'_2,G_2+G'_2),
\end{align*}
and likewise for longer expressions.

We apply Theorem~\ref{T: relative concatenation} in the following fashion.
\begin{corollary}\label{C: relative concatenation}
    Let $(X, \CX, \mu, U)$ be a $\Z^D$-system, $a_1, \ldots, a_\ell:\Z^L\to\Z^D$ be sequences, and $\Phi$ be a F{\o}lner sequence on $\Z^L$. Let $H_1, \ldots, H_r, G_1, \ldots, G_s, G_1', \ldots, G_{s'}'\subseteq \Z^D$ be subgroups. Suppose that for all $f_1, \ldots, f_\ell\in L^\infty(\mu)$,
    \begin{align}\label{E: vanishing average in concatenation}
        \norm{\E_{\bn\in\Phi}U_{a_1(\bn)}f_1\cdots U_{a_\ell(\bn)}f_\ell}_{L^2(\mu)}=0
    \end{align}
    whenever 
    \begin{align}\label{E: original estimate 1 in concatenation}
        \norm{f_j}_{H_1, \ldots, H_r, G_1, \ldots, G_s} = 0
    \end{align}
    or 
    \begin{align}\label{E: original estimate 2 in concatenation}
        \norm{f_j}_{H_1, \ldots, H_r, G_1', \ldots, G_{s'}'} = 0
    \end{align}
    for some fixed $1\leq j\leq \ell$. Then \eqref{E: vanishing average in concatenation} also holds whenever
    \begin{align}\label{E: new estimate in concatenation}
        \norm{f_j}_{H_1, \ldots, H_r, G_i + G_{i'}'\colon\; 1\leq i\leq s,\; 1\leq i'\leq s'} = 0.
    \end{align}
\end{corollary}
\begin{proof}
    Suppose that \eqref{E: new estimate in concatenation} holds.
    By Theorem~\ref{T: relative concatenation}, we can split $f_j = f_j'+f_j''$, where $f_j',f_j''$ satisfy the estimates \eqref{E: original estimate 1 in concatenation} and \eqref{E: original estimate 2 in concatenation} respectively. By the original seminorm control, the contributions of $f_j',f_j''$ to the average in \eqref{E: vanishing average in concatenation} vanish, and hence \eqref{E: vanishing average in concatenation} holds by subadditivity.
\end{proof}

 \subsection{The nilpotent case}
When $H$ is nilpotent but not abelian, we specialize to the case $H_1 = \cdots = H_s$: partly because this is the only case that we need, and partly because even the most basic properties of box seminorms along distinct groups remain unknown without commutativity. We shall make crucial use of Theorem~\ref{T: Candela-Szegedy}, the structure theorem of Candela and Szegedy \cite[Theorem~5.12]{CS23} that vastly generalizes the classical structure theorem of Host and Kra \cite{HK05a}. This result applies in our context because our definition of the cubic measures $\mu_{s,H}$ for ergodic actions coincides with the Candela-Szegedy definition of \textit{cubic couplings} \cite[Definition 5.4]{CS23}. Hence our definition of the seminorm $\nnorm{\cdot}_{s, H}$ meets the Candela-Szegedy definition of $U^s$ seminorms \cite[Definitions 3.15 and 5.7]{CS23}. Then the proof of \cite[Lemma~5.8]{CS23} indicates that our notion of the Host-Kra factor $\CZ_{s-1}(H)$ agrees with theirs \cite[Definition 5.9]{CS23}.

\section{Seminorm control transfer in the commutative setting}\label{S: control transfer}
Having presented the rudiments of the structure theory, we move on to prove Theorem~\ref{T: new estimates}, the new seminorm estimates for ergodic averages involving multidimensional polynomials. The proof of this result is split into two sections. In the subsequent section, we introduce a robust induction scheme underpinning the argument, and then run the induction to derive the result. A key step in this argument is the transfer of seminorm control, which we present in this section. The methods discussed here generalize the \textit{ping-pong} arguments originating in \cite{FrKu22a} and refined in \cite{DKKST25, DKKST24, FrKu22b, Kuc23}. What distinguishes our presentation from those earlier formulations is that we abandon the \textit{ping-pong} duality underlying the seminorm smoothing in \cite{DKKST25, DKKST24, FrKu22b, FrKu22a, Kuc23}. Rather, we isolate the key analytic maneuver in the form of Proposition~\ref{P: seminorm transfer} below, and then we reduce the \textit{ping-pong} game to a simple iteration of this result. We believe that this formulation provides a useful alternative viewpoint that may lead to further applications.

\begin{proposition}[Single seminorm control transfer]\label{P: seminorm transfer}
    Let $D,\ell,L,s'\in\N$ and $s\in\N_0$. Let $a_1, \ldots, a_\ell, a_\ell':\Z^L\to\Z^D$ be maps, $(X, \CX, \mu, U)$ be a $\Z^D$-system, and $\Phi$ be a F{\o}lner sequence on $\Z^L$. Let $$H_1, \ldots, H_{s+1}, H'_1, \ldots, H'_{s'}\subseteq \Z^D$$ be subgroups with the following properties:
    \begin{enumerate}
        \item $H_{s+1}$ contains the subgroup $\langle a_\ell(\bn)-a'_\ell(\bn)\colon\; \bn\in\Z^L\rangle$;
        \item for all $f_1, \ldots, f_\ell\in L^\infty(\mu)$,
        \begin{align}\label{E: vanishing average 1}
            \norm{\E_{\bn\in\Phi}U_{a_1(\bn)}f_1\cdots U_{a_\ell(\bn)}f_\ell}_{L^2(\mu)} = 0
        \end{align}
        whenever $\nnorm{f_\ell}_{H_1, \ldots, H_{s+1}} = 0$;
        \item for all $f_1, \ldots, f_{\ell}\in L^\infty(\mu)$,
        \begin{align}\label{E: vanishing average 1.1}
            \norm{\E_{\bn\in\Phi}U_{a_1(\bn)}f_1\cdots U_{a_{\ell-1}(\bn)}f_{\ell-1}\cdot U_{a'_\ell(\bn)}f_\ell}_{L^2(\mu)} = 0
        \end{align}
        whenever $\nnorm{f_{\ell-1}}_{H'_1, \ldots, H'_{s'}} = 0$;
    \end{enumerate}
     Then \eqref{E: vanishing average 1} holds whenever $\nnorm{f_{\ell-1}}_{H_1, \ldots, H_{s}, H'_1, \ldots, H'_{s'}} = 0$.
\end{proposition}
A priori, the new estimate looks less useful and more formidable than the original one, as it involves a seminorm of degree $s+s'$ rather than $s+1$. However, in our applications, the degree of a seminorm is often less important than the groups that make it up. In particular, if the groups $H'_1, \ldots, H'_{s'}$ are somehow ``nicer'' than $H_{s+1}$, then this result enables us to remove the problematic subgroup.
\begin{proof}
Without loss of generality, we assume that all the functions are 1-bounded.
    Suppose that \eqref{E: vanishing average 1} fails. The stashing trick 
    (see e.g. \cite[Lemma~3.2]{FrKu22b}) gives
    \begin{align*}
        \norm{\oE_{\bn\in\Phi}\prod_{j=1}^{\ell-1} U_{a_j(\bn)}f_j\cdot U_{a_\ell(\bn)}\tilde f_\ell}_{L^2(\mu)} > 0
    \end{align*}
    for
    \begin{align*}
        \tilde f_\ell := \lim_{k\to\infty}\E_{\bn\in\Phi_{N_k}}U_{-a_\ell(\bn)}\bar f_{0,k}\cdot U_{a_1(\bn)-a_\ell(\bn)}\bar f_1\cdots U_{a_{\ell-1}(\bn)-a_\ell(\bn)}\bar f_{\ell-1}, 
    \end{align*}
    where the limit is a weak limit, $(N_k)_k$ is an increasing sequence of natural numbers, and $f_{0,k}$'s are 1-bounded for every $k\in\N$.
    Using the first seminorm control and the subgroup property of box seminorms, we infer that 
    \begin{align*}
        \nnorm{\tilde f_\ell}_{H_1, \ldots, H_s, G_{s+1}}\geq \nnorm{\tilde f_\ell}_{H_1, \ldots, H_{s+1}}>0,
    \end{align*}
    where
    \begin{align*}
        G_{s+1}:=\langle a_\ell(\bn)-a'_\ell(\bn)\colon\; \bn\in\Z^L\rangle.
    \end{align*}
    Then the inductive property of box seminorms and the inverse theorem for the degree-1 seminorm give
    \begin{align*}
        \E_{h_1\in H_1}\cdots \E_{h_s\in H_s}\int \Delta_{h_1,\ldots, h_s}\tilde f_\ell \cdot u_h\; d\mu >0
    \end{align*}
    for some 1-bounded, $G_{s+1}$-invariant $u_h\in L^\infty(\mu)$. By the dual-difference interchange (see e.g. \cite[Proposition~3.3]{FrKu22b}),
    \begin{align*}
        \oE_{h_1,h_1'\in H_1} \cdots \oE_{h_s,h'_s\in H_{s}}\norm{\oE_{\bn\in\Phi}\prod_{j=1}^{\ell-1} U_{a_j(\bn)}f_{j,h,h'}\cdot U_{a_\ell(\bn)}u_{h,h'}}_{L^2(\mu)} > 0
    \end{align*}
    for some 1-bounded, $G_{s+1}$-invariant $u_{h,h'}\in L^\infty(\mu)$ and functions
    \begin{align*}
        f_{j,h,h'}:=\Delta_{h_1-h_1', \ldots, h_s-h_s'}f_j.
    \end{align*}
    From the definition of $G_{s+1}$ and the invariance property of $u_{h,h'}$'s, we deduce that $$U_{a_\ell(\bn)}u_{h,h'} = U_{a_\ell'(\bn)}u_{h,h'}$$ for all $h,h'$, and so
    \begin{align*}
       \oE_{h_1,h_1'\in H_1} \cdots \oE_{h_s,h'_s\in H_{s}}\norm{\oE_{\bn\in\Phi}\prod_{j=1}^{\ell-1} U_{a_j(\bn)}f_{j,h,h'}\cdot U_{a_\ell'(\bn)}u_{h,h'}}_{L^2(\mu)} > 0.
    \end{align*}
    
    At this point, we apply the second seminorm control in a uniform version supplied by \cite[Proposition~A.2]{FrKu22a}, obtaining the lower bound
    \begin{align*}
        \oE_{h_1,h_1'\in H_1} \cdots \oE_{h_s,h'_s\in H_{s}}\nnorm{f_{\ell-1,h,h'}}_{H'_1, \ldots, H'_{s'}}>0.
    \end{align*}
    The H\"older inequality and the inductive formula for box seminorms allow us to conclude that 
    \begin{align*}
        \nnorm{f_{\ell-1}}_{H_1, \ldots, H_{s}, H'_1, \ldots, H'_{s'}}>0.
    \end{align*}
    The claim follows by contrapositive.
\end{proof}
Proposition~\ref{P: seminorm transfer} can be applied as it stands, or it can be iterated an arbitrary number of times. In the former case, it gives seminorm control in terms of $f_{\ell-1}$ rather than the original function $f_\ell$. If we want a new estimate involving the same function $f_\ell$ as originally, we can iterate Proposition~\ref{P: seminorm transfer}, transferring control from $f_\ell$ to $f_{\ell-1}$ in the first iteration and back to $f_\ell$ in the second. The following corollary describes this procedure formally. 
Phrased in the language of the seminorm smoothing arguments from \cite{DKKST25, DKKST24, FrKu22b, FrKu22a, Kuc23}, the first and second iteration of Proposition~\ref{P: seminorm transfer} in the proof correspond to the \textit{ping} and \textit{pong} steps. {One notable difference between our argument and those from earlier works is that our ping and pong arguments both rely on the inverse theorems for degree-1 box seminorms, and hence they only use invariant rather than dual functions.} 

\begin{corollary}[Double seminorm control transfer]\label{C: seminorm transfer}
    Let $D,\ell,s''\in\N$ and $s,s'\in\N_0$.
    Let $a_1, \ldots, a_\ell, a_{\ell-1}'', a_\ell':\Z^L\to\Z^D$ be sequences, $(X, \CX, \mu, U)$ be a $\Z^D$-system, and $\Phi$ be a F{\o}lner sequence on $\Z^L$. Let $$H_1, \ldots, H_{s+1}, H'_1, \ldots, H'_{s'+1}, H''_1, \ldots, H''_{s''}\subseteq \Z^D$$ be subgroups with the following properties:
    \begin{enumerate}
        \item $H_{s+1}$ contains the subgroup $\langle a_\ell(\bn)-a'_\ell(\bn)\colon\; \bn\in\Phi\rangle$;
        \item $H'_{s'+1}$ contains the subgroup $\langle a_{\ell-1}(\bn)-a''_{\ell-1}(\bn)\colon\; \bn\in\Phi\rangle$;
        \item for all $f_1, \ldots, f_\ell\in L^\infty(\mu)$, \eqref{E: vanishing average 1} holds
        whenever $\nnorm{f_\ell}_{H_1, \ldots, H_{s+1}} = 0$;
        \item for all $f_1, \ldots, f_{\ell}\in L^\infty(\mu)$, \eqref{E: vanishing average 1.1} holds
        whenever $\nnorm{f_{\ell-1}}_{H'_1, \ldots, H'_{s'+1}} = 0$;
        \item for all $f_1, \ldots, f_\ell\in L^\infty(\mu)$,
\begin{align*}
    \norm{\E_{\bn\in\Phi}U_{a_1(\bn)}f_1\cdots U_{a_{\ell-2}(\bn)}f_{\ell-2}\cdot U_{a''_{\ell-1}(\bn)}f_{\ell-1}\cdot U_{a_\ell(\bn)}f_\ell}_{L^2(\mu)} = 0
\end{align*} holds
whenever $\nnorm{f_\ell}_{H''_1, \ldots, H''_{s}} = 0$.
    \end{enumerate}
     Then \eqref{E: vanishing average 1} holds whenever 
     \begin{align*}
         \nnorm{f_\ell}_{H_1, \ldots, H_{s}, H'_1, \ldots, H'_{s'}, H''_1, \ldots, H''_{s''}} = 0.
     \end{align*}
\end{corollary}
Again, the new estimate looks more daunting than the original one, but it is of greater utility if the subgroups $H_i'$ and $H_i''$ are nicer than $H_{s+1}$.
\begin{proof}
    The proof consists of two applications of Proposition~\ref{P: seminorm transfer}. Applying it once, we deduce that
    \eqref{E: vanishing average 1} holds whenever $\nnorm{f_{\ell-1}}_{H_1, \ldots, H_{s}, H'_1, \ldots, H'_{s'+1}} = 0$. Applying it the second time with the role of $\ell-1$ and $\ell$ switched,  we obtain the claimed estimate. 
\end{proof}

\subsection{A simple case of Theorem~\ref{T: new estimates}}\label{SS: seminorm example}
In the proof of Theorem~\ref{T: new estimates}, we will apply Corollary \ref{C: relative concatenation}, Proposition~\ref{P: seminorm transfer}, and Corollary \ref{C: seminorm transfer} many times. 
To illustrate how these results come into play, we sketch the proof of a simple case of Theorem~\ref{T: new estimates}.
    Consider the polynomials
    \begin{align*}
    p_1(n) = b_1 n + b_{2,1}n^2\quad \textrm{and}\quad p_2(n) = b_1 n + b_{2,2}n^2,    
    \end{align*}
    where $b_1, b_{2,1},b_{2,2}\in\Z^D$ are distinct and nonzero (obviously, we could consider the more general case where the linear coefficients differ, but then the argument requires more steps). By Theorem~\ref{T: original estimates}, 
        \begin{align}\label{E: original average in example}
        \E_{n\in\Z}U_{b_1 n+ b_{1,2} n^2}f_1\cdot U_{b_{1}n + b_{2,2}n^2}f_2
    \end{align}
    is controlled by 
    \begin{align}\label{E: original seminorm in example}
     \nnorm{f_2}_{\tilde H_{2,1}^{\times s_1},\tilde H_{2,0}^{\times s_1}}   
    \end{align}
    for some $s_1\in\N$, where
    \begin{align*}
        \tilde H_{2,1} = \langle b_{2,2}-b_{1,2}\rangle\quad \textrm{and}\quad \tilde H_{2,0} = \langle b_{2,2}\rangle.
    \end{align*}
    By contrast, Theorem~\ref{T: new estimates} claims control in terms of a seminorm of $f_2$ (and $f_1$, but we focus on $f_2$ here) along    
    \begin{align}\label{E: subgroups in new estimates example}
        H_{1,0} = \langle b_1, b_{1,2}\rangle,\quad H_{1,2} = H_{2,1} = \langle b_{2,2}-b_{1,2}\rangle,\quad H_{2,0} = \langle b_1, b_{2,2}\rangle. 
    \end{align}
    Note that $\tilde H_{2,1} = H_{1,2} = H_{2,1}$. Hence, due to the subgroup and monotonicity properties of box seminorms, passing from the old to new estimates is equivalent to swapping all instances of $\tilde H_{2,0}$ for (possibly many copies of) $H_{1,0}, H_{2,1}, H_{2,0}$. 

    In order to do so, we apply Corollary \ref{C: seminorm transfer} successively to get rid of each copy of $\tilde H_{2,0}$. While doing so, we use two auxiliary estimates:
    \begin{enumerate}
        \item the control of 
        \begin{align*}
        \E_{n\in\Z}U_{b_1 n+ b_{1,2} n^2}f_1\cdot U_{b_{1}n}f_2
    \end{align*}        
    by $\nnorm{f_1}_{\langle b_{1,2}\rangle^{\times s_2}}$ for some $s_2\in\N$,
    which follows directly from Theorem~\ref{T: original estimates};
        \item the control of
                \begin{align*}
        \E_{n\in\Z}U_{b_{1}n + b_{2,2}n^2}f_2
    \end{align*}
    by $\nnorm{f_2}_{H_{2,0}^{\times 2}}$, which is easy to obtain from the spectral theorem (see Proposition~\ref{P: type 1} below). 
    \end{enumerate}
    By Corollary \ref{C: seminorm transfer} applied with
    \begin{align*}
        a_1(n) = b_1 n+ b_{1,2} n^2,\quad a_2(n) = b_1 n+ b_{2,2} n^2,\quad a_2'(n) = b_1 n, \quad  \textrm{and}\quad a_1''(n) = 0,
    \end{align*}
    the original average \eqref{E: original average in example} is controlled by
    \begin{align}\label{E: new seminorm in example}
             \nnorm{f_2}_{\tilde H_{2,1}^{\times s_1},\tilde H_{2,0}^{\times (s_1-1)}, \langle b_{1,2}\rangle^{\times s_2}, H_{2,0}^{\times 2}}.   
    \end{align}

    At this point, we use Corollary \ref{C: relative concatenation} to merge (``concatenate'') the original estimate \eqref{E: original seminorm in example} with the new one \eqref{E: new seminorm in example}. Note that both estimates ``share'' $s_1$ copies of $\tilde H_{2,1}$ and $s_1-1$ copies of $\tilde H_{2,0}$. So Corollary \ref{C: relative concatenation} concatenates the one copy of $\tilde H_{2,1}$ remaining in \eqref{E: original seminorm in example} with the $s_2$ copies of $\langle b_{1,2}\rangle$ and $2$ copies of $H_{2,0}$ remaining in \eqref{E: new seminorm in example}. Hence we obtain control by
    \begin{align*}
        \nnorm{f_2}_{\tilde H_{2,1}^{\times s_1},\tilde H_{2,0}^{\times (s_1-1)}, \langle \tilde H_{2,0}, b_{1,2}\rangle^{\times s_2}, \langle \tilde H_{2,0}, H_{2,0}\rangle^{\times 2}}.
    \end{align*}
    The crucial observation is that
    \begin{align*}
        \langle \tilde H_{2,0}, b_{1,2}\rangle = \langle b_{1,2}, b_{2,2} \rangle \supseteq  \langle b_{2,2}-b_{1,2}\rangle = H_{2,1}
    \end{align*}
    and $\langle \tilde H_{2,0}, H_{2,0}\rangle = H_{2,0}$, so by the subgroup property of seminorms, we can control the original average \eqref{E: original average in example} by
    \begin{align*}
        \nnorm{f_2}_{\tilde H_{2,1}^{\times s_1},\tilde H_{2,0}^{\times (s_1-1)}, H_{2,1}^{\times s_2}, H_{2,0}^{\times 2}}.
    \end{align*}
    On iterating this argument $s_1-1$ more times, we get ride of all copies of $\tilde H_{2,0}$. Using the identity $\tilde H_{2,1} = H_{2,1}$, we then obtain the claimed new estimates in terms of subgroups \eqref{E: subgroups in new estimates example}.

\section{Proof of Theorem~\ref{T: new estimates}}\label{S: seminorm estimates}
We move on to the main part of the proof of Theorem~\ref{T: new estimates}. Throughout this section, $(X, \CX, \mu, U)$ is a $\Z^D$-system.

\subsection{Induction formalism}
\subsubsection{Reducible and irreducible families}
The argument leading to Theorem~\ref{T: new estimates} relies on a robust induction scheme that we set out to present now.  
We start by introducing a handful of definitions for polynomial families arising in the subsequent discussion. 
{All the polynomial families { $\CP$} considered in this section are assumed to consist of finitely many distinct, nonconstant polynomials with zero constant terms that lie in $\Z^D[\bn]$ (where $\bn\in\Z^L$) for some fixed $D, L\in\N$.}
\begin{definition}
    We call $\CP$:
    \begin{enumerate}
        \item \textit{irreducible of type 1} if $\ell =1$;
        \item \textit{irreducible of type 2} if $\ell = 2$ and, up to permuting polynomials, 
        \begin{align*}
         p_1(\bn) = \sum_{\vert\bi\vert=d} b_{\bi} \bn^\bi+q(\bn) \quad \textrm{and}\quad p_2(\bn) = q(\bn)   
        \end{align*}
        for some $b_{\bi}\in\Z^D$ not all equal to zero and $q\in\Z^{D}[\bn]$ with $\deg q\leq d-1$; 
        \item \textit{reducible} otherwise.
    \end{enumerate}
\end{definition}

Irreducible families form the base case of our induction. For them, Theorem~\ref{T: new estimates} is relatively straightforward, as demonstrated by the following two results.
\begin{proposition}[Theorem~\ref{T: new estimates} for type-1 irreducible families]\label{P: type 1}
    Let $\CP$ be irreducible of type 1, i.e. $\CP = \{p\}$ for $p(\bn) = \sum\limits_{1\leq \vert\bi\vert\leq d} b_{\bi} \bn^\bi$. Then 
    \begin{align*}
        \norm{\E_{\bn\in\Z^{L}}U_{p(\bn)}f}_{L^2(\mu)} = 0
    \end{align*}
    whenever $\nnorm{f}_{H^{\times 2}} = 0$, where $H:=\langle b_\bi\colon 1\leq \vert\bi\vert\leq d\rangle$.
\end{proposition}
It is easy to see from the proof that the claimed control by $\CZ(H^{\times 2})$ can be upgraded to control by
\begin{align*}
    \Krat(H) := \langle E\in\CX\colon \exists r\in\N\; \textrm{s.t.}\; U_{b_\bi}^{-r} E =E\; \textrm{for all}\; 1\leq |\bi|\leq d\rangle.
\end{align*}
This is however not needed for our application. 
\begin{proof}
    By passing to ergodic components if needed, we can assume that $H$ acts ergodically. 
    By the spectral theorem, {for every F{\o}lner sequence $\Phi$ on $\Z^D$}  we can find a finite Borel measure $\sigma_f$ on $\T^d$ such that
    \begin{align}\label{E: spectral}
        \norm{\E_{\bn\in\Phi}U_{p(\bn)}f}_{L^2(\mu)}^2 = \int \abs{\E_{\bn\in\Phi}e(t\cdot p(\bn))}^2\; d\sigma_f(t).
    \end{align}
    Whenever at least one coordinate of $t$ is irrational, the exponential sum vanishes. Since $f$ is orthogonal to $\CZ(H^{\times 2})$, and this factor is generated by eigenfunctions of $H$ due to ergodicity, the measure $\sigma_f$ has no point masses. Since $\Q^d$ is countable, its total contribution of $\Q^d$ to the integral on the right-hand side of \eqref{E: spectral} is 0.
\end{proof}

Before proving Theorem~\ref{T: new estimates} for type-2 irreducible families, we recall the definition of various groups of coefficients appearing in the statements of Theorems \ref{T: original estimates} and \ref{T: new estimates}. 
Write
\begin{align}\label{E: p_j}
    p_j(\bn) = \sum\limits_{1\leq\vert\bi\vert\leq d} b_{j,\bi}\bn^\bi
\end{align}
for $1\leq j\leq \ell$ and $p_0 = 0$; then 
\begin{align}\label{E: H}
    H_{j,j'}&:=\langle b_{j,\bi}-b_{j',\bi}\colon\; 1\leq \vert\bi\vert\leq d \rangle\\
    \label{E: tilde H}
    \tilde H_{j,j'}&:=\langle b_{j,\bi}-b_{j',\bi}\colon\; \vert\bi\vert = \deg (p_j-p_{j'}) \rangle.
\end{align}
Equipped with these definitions, we are ready to tackle the other class of irreducible families.
\begin{proposition}[Theorem~\ref{T: new estimates} for type-2 irreducible families]\label{P: type 2}
    Let $\CP$ be irreducible of type 2, i.e. up to permuting indices, $\CP = \{p_1,p_2\}$ for 
    \begin{align*}
    p_1(\bn) = \sum_{\vert\bi\vert=d} b_{\bi} \bn^\bi+q(\bn), \quad p_2(\bn) = q(\bn) 
    \end{align*}
     for some $b_{\bi}\in\Z^D$ not all equal to zero and $q\in\Z^{D}[\bn]$ with $\deg q\leq d-1$. Then there exists a natural number $s=O_{d,L}(1)$ such that
    \begin{align*}
        \norm{\E_{\bn\in\Z^{L}}U_{p_1(\bn)}f_1\cdot U_{p_2(\bn)}f_2}_{L^2(\mu)} = 0
    \end{align*}
    whenever $\nnorm{f_j}_{H_{1,2}^{\times s}, H_{2,0}^{\times s}} = 0$ for some $1\leq j\leq 2$.
\end{proposition}
The seminorm control provided by Proposition~\ref{P: type 2} looks stronger than one claimed by Theorem~\ref{T: new estimates}. However, since the missing group $H_{1,0}$ contains both $H_{1,2}$ and $H_{2,0}$, the two are equivalent as long as we do not care about the value of $s$. 
\begin{proof}
    For $j=1$, Theorem~\ref{T: original estimates} gives control by $\nnorm{f_1}_{\tilde H_{1,0}^{\times s'}, \tilde H_{1,2}^{\times s'}}$ for some natural number $s'=O_{d,L}(1)$. Since 
    \begin{align}\label{E: equality of subgroups H and tilde H}
        \tilde H_{1,0} = \tilde H_{1,2} = H_{1,2} = \langle b_{\bi}\colon \vert\bi\vert=d\rangle, 
    \end{align}
    this gives us control by $\nnorm{f_1}_{H_{1,2}^{\times 2s'}}$, implying the claimed control in terms of $f_1$ for any $s\geq 2s'$. 

    For $j=2$, we first observe that the previous case and the monotonicity property of box seminorms give control by $\nnorm{f_1}_{H_{1,2}^{\times 2s'}, H_{1,0}}$.
    We combine this seminorm control with Propositions \ref{P: seminorm transfer} 
    (applied with $a_1 = p_2$, $a_2 = p_1$, $a_2'=0$) and \ref{P: type 1}  to deduce control by $\nnorm{f_2}_{H_{1,2}^{\times 2s'}, H_{2,0}^{\times 2}}$, which implies the claim for $s:= \max(2,2s')$. 
\end{proof}

\subsubsection{Complexity}\label{SSS: complexity}
Our goal now is to prove Theorem~\ref{T: new estimates} for a reducible family. The proof will follow a variant of the seminorm smoothing argument during which we will induct on the complexity of $\CP$. This section presents the relevant notion of complexity.

{Take $\CP = \{p_1, \ldots, p_\ell\}\subseteq\Z^D[\bn]$ with $p_j$'s as in \eqref{E: p_j}.}
For $0\leq j,j'\leq \ell, j\neq j'$ and $i\in\N$, let 
$$H_{j,j'}^{i}:=\langle b_{j,\bi}-b_{j',\bi}\colon \vert\bi\vert=i\rangle$$
be the group generated by the degree-$i$ coefficients of $p_j-p_j'$.
We abbreviate the group $H_{0,j}^{i}$ of the degree-$i$ coefficients of $p_j$ as $H_{j}^{i}$.

For $i\in\N$, we let $$\FB_{i} := \{H_{0,j}^{i}\colon 1\leq j\leq \ell\; \textrm{ and }\; H_{0,j}^{i}\neq \{0\}\}.$$
For all $B\in \FB_{i}$,  we define $$\tau_{i}(B) := \abs{\{1\leq j\leq \ell\colon H_{j}^{i}=B\}}.$$ We also set $\tau_i(\{0\})=\infty$ for later convenience.

We describe the \textit{complexity} of the family $\CP$ by the tuple $(\ell, \zeta, \omega)$, where: 
\begin{enumerate}
    \item $\ell$ denotes its \textit{length}, i.e. the number of polynomials in $\mathcal{P}$.
    \item $\zeta = (\zeta_1, \ldots, \zeta_d)$ denotes its \textit{degree}, and it is given by 
    $$\zeta_i :=|\{1\leq j\leq \ell:\; \deg p_j = i\}|$$ 
    for $1\leq i\leq d$. We order degrees colexicographically, i.e. $\zeta>\zeta'$ if $\zeta_{d'}>\zeta'_{d'}$ for some $1\leq d'\leq d$ and $\zeta_i = \zeta'_i$ for all $d'<i\leq d$. 
    \item $\omega = (\omega_1, \ldots, \omega_d)$ denotes its \textit{type}, and it is given by 
    \begin{align*}
    \omega_i:=\sum_{B\in \FB_i} \tau_{i}(B)^2 = \sum_{B\in \FB_i}\abs{\{1\leq j\leq \ell\colon H_{j}^{i}=B\}}^2
\end{align*}
for $1\leq i\leq d$.
We order types by $\omega>\omega'$ if $\omega_{d'}<\omega'_{d'}$ for some $1\leq d'\leq d$ and $\omega_i = \omega'_i$ for all $d'<i\leq d$.
\end{enumerate}
If needed, we denote the dependence of $\FB_i, \tau, \ell, \zeta, \omega$, etc. on $\CP$ in the subscript. 

We order the tuples $(\ell, \zeta, \omega)$ lexicographically. That is, $(\ell, \zeta, \omega)>(\ell', \zeta', \omega')$ if:
\begin{enumerate}
    \item $\ell > \ell'$;
    \item $\ell = \ell'$ and $\zeta > \zeta'$;
    \item $\ell = \ell'$, $\zeta = \zeta'$, and $\omega > \omega'$.
\end{enumerate}
The intuition behind this ordering is as follows. A longer polynomial family is more complex than a shorter one. For families of the same length, the one with more high-degree polynomials is more complex. If the degrees of the polynomials in both families are identically distributed, then the family with a larger variety of high-degree coefficients, in the way measured by type, will be more complex.

The inductive step of Theorem~\ref{T: new estimates} takes the following form.
\begin{proposition}\label{P: estimates for reducible families}
    Let $(X, \CX, \mu, U)$ be a $\Z^D$-system and $\CP\subseteq\Z^D[\bn]$ be a reducible family of complexity $(\ell, \zeta, \omega)$. If Theorem~\ref{T: new estimates} holds for all families of complexity less than $(\ell, \zeta, \omega)$, then it also holds for $\CP$.
\end{proposition}
\begin{proof}[Proof of Theorem~\ref{T: new estimates} assuming Proposition~\ref{P: estimates for reducible families}]
    If $\CP$ is irreducible, then Theorem~\ref{T: new estimates} holds thanks to Propositions \ref{P: type 1} and \ref{P: type 2}. Let $\CP$ then be reducible of complexity $(\ell, \zeta, \omega)$.
    We claim that there exist $s=O_{d,\ell,L}(1)$ complexities
    \begin{align*}
        (\ell, \zeta, \omega) := (\ell^{(1)}, \zeta^{(1)}, \omega^{(1)})>(\ell^{(2)}, \zeta^{(2)}, \omega^{(2)}) > \cdots > (\ell^{(s)}, \zeta^{(s)}, \omega^{(s)})
    \end{align*}
    such that $\ell^{(s)} = 1$, in which case we have an irreducible family of type 1. Indeed, there are only $\ell-1$ possible lengths $\ell'<\ell$; $O_{d,\ell}(1)$ possible degrees $\zeta'<\zeta$; and $O_{d,\ell,L}(1)$ possible types $\omega'<\omega$. Hence within $O_{d,\ell,L}(1)$ steps, we reduce the reducible tuple of complexity $(\ell, \zeta, \omega)$ to an irreducible one. Thus, Theorem~\ref{T: new estimates} for types $(\ell, \zeta, \omega)$ follows by induction from Proposition~\ref{P: estimates for reducible families}.
\end{proof}

\subsubsection{Choosing a dominant index}\label{SS: dominant index}

{For the remainder of Section \ref{S: seminorm estimates}, fix  $\CP = \{p_1, \ldots, p_\ell\}\subseteq\Z^D[\bn]$ with $p_j$'s as in \eqref{E: p_j}, and let $H_{j,j'}, \tilde H_{j,j'}$ be the groups defined in \eqref{E: H} and \eqref{E: tilde H} for that family.}

In deriving Proposition~\ref{P: estimates for reducible families} for a reducible family $\CP$, the main challenge is to achieve \textit{complexity reduction}, i.e. to reduce $\CP$ to a family of lower complexity and then lift the estimates for the lower-complexity family to $\CP$. In fact, we will reduce to two different lower-complexity families: $\CP'$ in the \textit{ping} step and $\CP''$ in the \textit{pong} step. While showing that $\CP''$ has lower complexity and satisfies all the desired properties is not a real issue, choosing the correct reduction in the \textit{ping} step is highly nontrivial.

Let $\CP$ be given as in Section \ref{SSS: complexity}. We will construct the family $\CP'$ by appropriately choosing indices $1\leq k\leq \ell$ and $1\leq i\leq d$, and then changing the degree-$i$ coefficients of $p_k$. 
In order to find such an index $k$, 
we impose a weak total order on the indices $1\leq j\leq \ell$, meaning that we divide them into equivalence classes and then compare every two equivalence classes. 
Given two distinct $1\leq j,j'\leq \ell$, we examine their degree-$i$ coefficients for $1\leq i\leq d$, starting with $i=d$. The relative position of $j,j'$ in this order depends on the frequency with which their degree-$i$ coefficients arise among $H_{1}^{i}, \ldots, H_{\ell}^{i}$. Suppose that up to level $i+1$, the relation between $j$ and $j'$ has not yet been determined. We compare their degree-$i$ coefficients as follows: 
\begin{enumerate}
    \item if $\tau_i(H_{j}^{i})<\tau_{i}(H_{j'}^{i})$, we set $j\succ_\CP j'$ (note that this also includes the case when ${H_{j}^{i}} \neq \{0\}$ while ${H_{j'}^{i}} = \{0\}$, and that the relative order on $j,j'$ is opposite to the order on $\tau_i(H_{j}^{i}),\tau_{i}(H_{j'}^{i})$);
    \item if  $\tau_i(H_{j}^{i})=\tau_{i}(H_{j'}^{i})$, then we move down to degree-$(i-1)$ coefficients and compare $\tau_{i-1}(H_{j}^{i-1}), \tau_{i-1}(H_{j'}^{i-1})$.
\end{enumerate}
If $\tau_i(H_{j}^{i}) = \tau_i(H_{j'}^{i})$ for all $1\leq i\leq d$, then $j\sim_\CP j'$. Otherwise either $j\succ_\CP j'$ or $j\prec_\CP j'$, depending on which one of $\tau_i(H_{j}^{i}),\tau_i(H_{j'}^{i})$ is smaller at the highest index $i$ at which these two quantities differ. It is clear that this relation is transitive.

We call an index $1\leq k\leq \ell$ \textit{dominant} if $k\succeq_\CP j$ (i.e. $k\succ_\CP j$ or $k\sim_\CP j$) for all $1\leq j\leq \ell.$ Such an index exists since our relation defines a weak total order. We also call a pair $(k,m)$ \emph{simple (with respect to $\CP$)} if for all $i\in\N$, either $H_{k,m}^{i}$ equals $H_{k,m}$ or is $\{0\}$. In other words, $p_k,p_m$ only differ at the level of degree-$i$ coefficients for a single value $i\in\N$.
Then the gist of complexity reduction is contained in the statement below. In what follows, we set $p_0 := 0$.

\begin{lemma}\label{L: P'}
    Let $1\leq k\leq \ell$ be a dominant index for the reducible family $\CP$, and suppose that there exists $0\leq m\leq \ell$ distinct from $k$ such that $(k,m)$ is not simple. Let $d' := \deg(p_k-p_m)$. Define 
\begin{align*}
    p'_j(\bn) := \begin{cases} \sum\limits_{\vert \bi\vert\neq d'} b_{k, \bi}\bn^\bi +  \sum\limits_{\vert \bi\vert=d'} b_{m,\bi}\bn^{\bi},\; &j= k,\\
    p_j(\bn),\; &j\neq k,
\end{cases}  
\end{align*}
and take $\CP'$ to be $\{p'_1, \ldots, p'_\ell\}$ if $p'_k$ is nonconstant and distinct from $p'_j$ for all $j\neq k$, and $\{p'_1, \ldots, p'_{k-1},p'_{k+1},\ldots, p'_\ell\}$ otherwise. Then the following hold:
\begin{enumerate}
    \item the family $\CP'$ has lower complexity than $\CP$;
    \item there exists $1\leq t\leq \ell$ distinct from $k$ such that $p'_t \neq p'_j$ for all $0\leq j\leq \ell$ different from $t$.
\end{enumerate}
\end{lemma}
\begin{proof}
    Let $(\ell, \zeta, \omega)$ and $(\ell', \zeta', \omega')$ be the complexities of $\CP$ and $\CP'$ respectively. Clearly, $\ell'\leq \ell$. If this is strict, we are done, so we assume that $\ell'=\ell$, i.e. $p'_k$ is nonconstant and distinct from $p'_j$ for all $j\neq k$. The next stage in comparing complexities is to look at $\zeta'$ and $\zeta$. Since $k$ is dominant, we must have $\deg p_k\geq \deg p_m$. Hence $H_{k}^{d'}\neq \{0\}$, and so $\zeta'_{d'} \leq \zeta_{d'}$ (with equality if and only if $H_{m}^{d'}\neq \{0\}$) and $\zeta'_i = \zeta_i$ for $d'<i\leq d$. This implies that $\zeta'\leq \zeta$, and this is not strict only if $H_{m}^{d'}\neq\{0\}$. We therefore compare the types $\omega'$ and $\omega$ under the assumption $H_{m}^{d'}\neq \{0\}$. The assumption $k\succeq_\CP m$ plus the fact that $H_{k}^{d'},H_{m}^{d'}\neq \{0\}$ jointly imply that $\tau_{d'}(H_{k}^{d'})\leq \tau_{d'}(H_{m}^{d'})$ (here and below, $\tau_{d'}:=\tau_{d',\CP}$). Hence 
\begin{align*}
    \omega'_{d'}-\omega_{d'} 
    &= \brac{(\tau_{d'}(H_{m}^{d'})+1)^2 + (\tau_{d'}(H_{k}^{d'})-1)^2}-\brac{\tau_{d'}(H_{m}^{d'})^2 + \tau_{d'}(H_{k}^{d'})^2}\\
    &= 2(\tau_{d'}(H_{m}^{d'})- \tau_{d'}(H_{k}^{d'})+1)\geq 2.
\end{align*}
    Since $\omega_i = \omega'_i$ for all $d'<i\leq d$, we have $\omega'<\omega$. This completes the proof that $\CP'$ has lower complexity than $\CP$.

We move on to the second claim. This is where we crucially use the reducibility of $\CP$. Suppose the claim fails, i.e. for every $1\leq j\leq \ell$ distinct from $k$ there exists $0\leq j'\leq \ell$ distinct from $j$ such that $p'_j = p'_{j'}$. 
Since $p'_j = p_j$ for $j\neq k$, and elements of $\CP$ are distinct and nonconstant, necessarily $\ell =2$ and $j'=k$. (If $\ell \geq 3$, and $k=3$ for simplicity, then $p_3'$ can equal $p_1=p_1'$ or $p_2=p_2'$, but not both.) To simplify notation, suppose that $k = 2$; then $p_1 = p'_1 = p'_2$. Because of the way $p'_2$ is constructed, this can only happen if $p_1, p_2$  have different degree-$d'$ coefficients for a single value $1\leq d'\leq d$. In this case, {$(1,2)$ is simple with respect to $\CP$}. Since the index $m$ used to construct $p'_2$ from $p_2$ is chosen so that $(m,2)$ is not simple, we must have $m=0$. From this follow two important observations. First, the degree-$d'$ coefficients of $p_1, p_2'$ are 0. Second, $d' = \deg p_2$. Hence $$p_2(\bn) = p_1(\bn) + \sum_{\vert\bi\vert=d'}b_{2,\bi}\bn^{\bi}$$ for at least one nonzero $b_{2,\bi}$ and $\deg p_1 < d'$. The assumption that $(m,2)$ is not simple with respect to $\CP$ further gives $\deg p_1 > 0$. Hence $\CP$ is an irreducible family of type 2, yielding a contradiction.
\end{proof}

\subsection{Reduction to seminorm estimates in terms of dominant-index function}
We derive the general case of Proposition~\ref{P: estimates for reducible families} from the following special case, which gives the desired seminorm estimates for a function with dominant index. From now on, we shall always assume that $\ell$ is a dominant index in order to simplify the notation.
\begin{proposition}\label{P: estimates for reducible families and dominant index}
    Let $(X, \CX, \mu, U)$ be a $\Z^D$-system and $\CP$ be reducible of complexity $(\ell, \zeta, \omega)$. 
    Suppose that $\ell$ is a dominant index, and that Theorem~\ref{T: new estimates} holds for all families of complexity less than $(\ell, \zeta, \omega)$. Then there exists a natural number $s=O_{d,\ell}(1)$ such that for all  $f_1, \ldots, f_\ell\in L^\infty(\mu)$,
    \begin{align}\label{E: vanishing average 2}
        \norm{\E_{\bn\in\Z^{L}} U_{p_1(\bn)}f_1\cdots U_{p_\ell(\bn)}f_\ell}_{L^2(\mu)} = 0
    \end{align}
    whenever $\nnorm{f_\ell}_{H_{j,{j'}}^{\times s}\colon\; 0\leq j<{j'}\leq \ell} = 0$.
\end{proposition}
\begin{proof}[Proof of Proposition~\ref{P: estimates for reducible families} assuming Proposition~\ref{P: estimates for reducible families and dominant index}]
Proposition~\ref{P: estimates for reducible families and dominant index} yields the claimed seminorm control in terms of $f_\ell,$ and so our goal is to extend it to $f_1, \ldots, f_{\ell-1}$. One fact that we shall implicitly take advantage of is that the subgroups appearing in the estimate given by Theorem~\ref{T: new estimates} are the same regardless of which function we use to control the average.

Consider the family $\CP'=\{p_1, \ldots, p_{\ell-1}\}$; it clearly has complexity less than $(\ell, \zeta, \omega)$. By the inductive hypothesis, we have
\begin{align*}
        \norm{\E_{\bn\in\Z^{L}} U_{p_1(\bn)}f_1\cdots U_{p_{\ell-1}(\bn)}f_{\ell-1}}_{L^2(\mu)} = 0
    \end{align*}
    whenever $\nnorm{f_{\ell-1}}_{H_{j,{j'}}^{\times s'}\colon\; 0\leq j<{j'}\leq \ell-1} = 0$ for some natural number $s'=O_{d,\ell,L}(1)$.
Combining this control with Propositions \ref{P: seminorm transfer} (applied with $a_1 = p_1, \ldots, a_\ell = p_\ell$ and $a'_\ell = 0$) and \ref{P: estimates for reducible families and dominant index} as well as the monotonicity property of box seminorms, we deduce that \eqref{E: vanishing average 2} holds whenever
\begin{align*}
    \nnorm{f_{\ell-1}}_{H_{j,{j'}}^{\times (s+s')}\colon\; 0\leq j<{j'}\leq \ell} = 0
\end{align*}
for the natural number $s=O_{d,\ell}(1)$ provided by Proposition~\ref{P: estimates for reducible families and dominant index}. The choice of index $\ell-1$ was arbitrary, and we can get analogous control for any of $f_1, \ldots, f_{\ell-2}$ in place of $f_{\ell-1}$. The claim follows with $s+s'$ playing the role of $s$ in the statement of Proposition~\ref{P: estimates for reducible families}.
\end{proof}

\subsection{Reduction to seminorm smoothing}

In the rest of this section, we set
\begin{align}\label{E: bolded G}
    \mathbf{H}:=\{H_{j,j'}\colon 0\leq j<j'\leq \ell\},
\end{align}
where the groups $H_{j,j'}$ are defined as in \eqref{E: H} for the family $\CP$. Proposition~\ref{P: estimates for reducible families and dominant index} will follow by iterating on the following result.
\begin{proposition}\label{P: smoothing}
    Let $(X, \CX, \mu, U)$ be a $\Z^D$-system and $\CP$ be reducible of complexity $(\ell, \zeta, \omega)$. Suppose that Theorem~\ref{T: new estimates} holds for all families of complexity less than $(\ell, \zeta, \omega)$. Suppose that $\ell$ is a dominant index, and that there exists an index $0\leq m < \ell$  for which {$(\ell,m)$ is not simple}. Lastly, suppose that there exist natural numbers $s_1, s_2=O_{d,\ell,L}(1)$ and subgroups $G_1, \ldots, G_{s_2}\subseteq \Z^D$ such that for all  $f_1, \ldots, f_\ell\in L^\infty(\mu)$, \eqref{E: vanishing average 2} holds whenever $$\nnorm{f_\ell}_{\mathbf{H}^{\times s_1}, G_1, \ldots, G_{s_2}, \tilde H_{\ell,m}} = 0.$$
    Then there exists a natural number $s_3=O_{d,\ell,L}(1)$ such that \eqref{E: vanishing average 2} also holds whenever
    $$\nnorm{f_\ell}_{\mathbf{H}^{\times s_3}, G_1, \ldots, G_{s_2}} = 0.$$
\end{proposition}
\begin{proof}[Proof of Proposition~\ref{P: estimates for reducible families and dominant index} assuming Proposition~\ref{P: smoothing}]
    Recall the definition of the groups $\tilde H_{\ell,j}$ from \eqref{E: tilde H}.
    By Theorem~\ref{T: original estimates}, we know that \eqref{E: vanishing average 2} holds whenever
    \begin{align*}
        \nnorm{f_\ell}_{\tilde H_{\ell,0}^{\times s}, \tilde H_{\ell, 1}^{\times s}, \ldots, \tilde H_{\ell, \ell-1}^{\times s}} = 0
    \end{align*}
    for some natural number $s=O_{d,\ell,L}(1)$. Now, if $(\ell,j)$ is simple for some $0\leq j<\ell$, then $\tilde H_{\ell, j} = H_{\ell,j}$, and so we can subsume all such groups $\tilde H_{\ell, j}$ into $\mathbf{H}$. Using this and the monotonicity property of box seminorms, we deduce that \eqref{E: vanishing average 2} holds whenever
    \begin{align*}
        \nnorm{f_\ell}_{\mathbf{H}^{\times \ell s}, \{\tilde H_{\ell, j}\colon\; 0\leq j <\ell,\; \text{$(\ell,j)$ is not simple}\}^{\times s}} = 0.
    \end{align*}
    Applying Proposition~\ref{P: smoothing} $$s\cdot \bigabs{\srem{\tilde H_{\ell, j}\colon\; 0\leq j <\ell,\; \text{$(\ell,j)$ is not simple}}}\leq \ell s$$ times, we remove each such subgroup one by one at the cost of increasing the number of subgroups from $\mathbf{H}$. 
\end{proof}
\subsection{Proof of Proposition~\ref{P: smoothing}}
Let $\CP'$ and $1\leq t< \ell$ be as in Lemma~\ref{L: P'}, and define $\CP''=\{p_j\colon\; 1\leq j\leq \ell,\; j\neq t\}$. Then both $\CP', \CP''$ have lower complexity than $\CP$; this is obvious for $\CP''$, and it follows from Lemma~\ref{L: P'} for $\CP'$. Let $\mathbf{H}$ be as in \eqref{E: bolded G}  and $$\mathbf{H'}:=\rem{H_{j,j'}'\colon\; 0\leq j<j'\leq \ell},$$ where the groups $H_{j,{j'}}'$ are as in \eqref{E: H} but defined for the family $\CP'$.

By the induction hypothesis applied to $\CP, \CP'$ and the fact that $p'_l\neq p'_j$ for any $j\neq l$, we have the following seminorm estimates for some natural numbers $s',s''=O_{d,\ell,L}(1)$:
\begin{gather*}
    \norm{\E_{\bn\in\Z^\ell}\prod_{j=1}^\ell U_{p'_j(\bn)}f_j}_{L^2(\mu)} = 0\quad \textrm{whenever}\quad \nnorm{f_t}_{\mathbf{H'}^{\times s'}} = 0\\
    \norm{\E_{\bn\in\Z^{L}}\prod_{j=1}^\ell U_{p_j(\bn)}f_j}_{L^2(\mu)} = 0\quad \textrm{whenever}\quad \nnorm{f_\ell}_{\mathbf{H}^{\times s''}} = 0 \quad\textrm{and}\quad f_t =1.
\end{gather*}
At the same time, our assumption gives that
\begin{align}\label{E: assumed estimate}
    \norm{\E_{\bn\in\Z^L}\prod_{j=1}^\ell U_{p_j(\bn)}f_j}_{L^2(\mu)} = 0\quad \textrm{whenever}\quad \nnorm{f_\ell}_{\mathbf{H}^{\times s_1}, \mathbf{G}, \tilde H_{\ell,m}} =0,
\end{align}
where 
\begin{align*}
    \mathbf{G} := \rem{G_1, \ldots, G_{s_2}}.
\end{align*}

Since $(\ell,m)$ is simple, we have $\tilde H_{\ell, m} = H_{\ell,m}$.
Applying Corollary \ref{C: seminorm transfer} to all these estimates with $a_1 = p_1, \ldots, a_\ell = p_\ell, a_\ell' = p'_\ell$, and $p_l'$ playing the role of $a''_{\ell-1}$, we deduce that
\begin{align}\label{E: pong}
    \norm{\E_{\bn\in\Z^L}\prod_{j=1}^\ell U_{p_j(\bn)}f_j}_{L^2(\mu)} = 0\quad \textrm{whenever}\quad \nnorm{f_\ell}_{\mathbf{H}^{\times (s_1+s'')}, \mathbf{H'}^{\times s'}, \mathbf{G}}=0.
\end{align}

We want to combine estimates \eqref{E: assumed estimate} and \eqref{E: pong} using Corollary \ref{C: relative concatenation}, the relative concatenation result. To this end, we concatenate $\tilde H_{\ell,m}$ with groups $H_{j,j'}$ appearing in $\mathbf{H}$ and $H_{j,j'}'$ appearing in $\mathbf{H'}$. 
For the groups $H_{j,j'}$, we simply observe that $H_{j,j'}+ \tilde H_{\ell,m} \supseteq H_{j,j'}$. For the groups $H_{j,j'}'$, we split into two cases. If $j,j'\neq \ell$, then the group $H'_{j,j'}$ equals $H_{j,j'}$, and we land in the previous case. Otherwise $$H'_{\ell,j} = \langle b_{m,\bi}-b_{j,\bi},\; b_{\ell,\bi'}-b_{j,\bi'}\colon\; \vert \bi\vert=d',\; \vert \bi'\vert\neq d'\rangle,$$ 
in which case $H_{\ell,j}'+ \tilde H_{\ell,m} \supseteq H_{\ell,j}$ since for all $\vert\bi\vert=d'$, the terms $b_{\ell, \bi}-b_{j,\bi}$, the only elements of $H_{\ell, j}$ potentially missing from $H'_{\ell,j}$, satisfy
$$b_{\ell, \bi}-b_{j,\bi} = (b_{\ell, \bi}-b_{m,\bi})+ (b_{m,\bi}-b_{j,\bi})\in \tilde H_{\ell,m} + H_{\ell,j}'.$$
Hence by Corollary \ref{C: relative concatenation} and the subgroup property of box seminorms, we deduce that \eqref{E: vanishing average 2} holds whenever
\begin{align*}
    \nnorm{f_\ell}_{\mathbf{H}^{\times (s_1+s'+s'')}, \mathbf{G}}=0.
\end{align*}
This completes the proof. \hfill $\Box$

\section{2-step nilpotent PET for $(T^n, S^{n^2})$}\label{S: PET n, n^2}
We depart from the abelian universe, plunging into the brave new world of 2-step nilpotent systems. Sections \ref{S: PET n, n^2} and \ref{S: PET} are dedicated to the proof of Theorem~\ref{T: nilpotent seminorm estimates}. While the next section covers the general case, this section illustrates the main ideas by proving the following special case.
\begin{theorem}\label{T: nilpotent joint ergodicity n n^2}
    There exists $s\in\N$ such that for all 2-step nilpotent systems $(X, \CX, \mu, T, S)$ and functions $f_1, f_2\in L^\infty(\mu)$, we have
    \begin{align}\label{E: n, n^2}
    \norm{\E_{n\in\Z}T^{n}f_1\cdot S^{n^2}f_2}_{L^2(\mu)} = 0
\end{align}
whenever $\nnorm{f_2}_{s,S} = 0$.
\end{theorem}

\begin{proof}
We assume that $f_1, f_2$ are 1-bounded.
Let $R:=[T,S]$, and set
\begin{align*}
    U_{(a,b,c)}:=T^{a}S^{b}R^{c}.
\end{align*}
On endowing $\Z^3$ with the operation
\begin{align*}
    (a,b,c)*(x,y,z) = (a + x, b + y, c + z - bx),
\end{align*}
the map $h\mapsto U_h$ defines an isomorphism between $(\Z^3, *)$ and the group $H=\langle T, S\rangle$. Note also that
\begin{align*}
    (x,y,z)\inv &= (-x,-y,-z - xy),\\
    (a,b,c)*(x,y,z)\inv &= (a - x, b - y, c -z +(b-y)x);
\end{align*}
we will use these identities freely in the ensuing argument.

In this notation, the average in \eqref{E: n, n^2} can be recast as
\begin{align}\label{E: n, n^2 recast}
    \E_{n\in\Z} U_{(n,0,0)}f_1\cdot U_{(0,n^2,0)}f_2.
\end{align}
The overall strategy is to apply the van der Corput inequality
to \eqref{E: n, n^2 recast} until we arrive at an average that meets two criteria:
\begin{enumerate}
    \item it has no contribution from $T$;
    \item each term has a nontrivial contribution from $S$ which in some sense is \textit{distinguishable} from the contribution of $R$.
\end{enumerate}
The first condition allows us to apply Theorem~\ref{T: new estimates} since the average involves only the commuting transformations $S, R$. The second one ensures that each group of coefficients appearing in the conclusion of Theorem~\ref{T: new estimates} contains a nonzero multiple of~$S$. The desired conclusion follows on invoking the subgroup property of box seminorms.

Let
\begin{align*}
    \CP_0 = \{(n,0,0),\; (0, n^2,0)\}
\end{align*}
be the original polynomial family arising in \eqref{E: n, n^2 recast}. Applying the asymmetric van der Corput inequality to \eqref{E: n, n^2 recast}, we get
\begin{multline*}
    \norm{\E_{n\in\Z} U_{(n,0,0)}f_1\cdot U_{(0,n^2,0)}f_2}_{L^2(\mu)}^2\\
    \leq \oE_{h_1\in\Z}\abs{\E_{n\in \Z}U_{(n,0,0)}f_1\cdot U_{(n+h_1,0,0)}\bar f_1\cdot U_{(0,n^2,0)}f_2\cdot U_{(0,(n+h_1)^2,0)}\bar f_2}.
\end{multline*}
After composing the average with $U_{(n,0,0)}\inv = U_{(-n,0,0)}$ and using
\begin{align*}
    U_{(0,(n+\eps_1 h_1)^2,0)}(U_{(n,0,0)}\inv f) = U_{(-n, (n+\eps_1 h_1)^2, n(n+\eps_1 h_1)^2)}f
\end{align*}
for $\eps_1\in\{0,1\}$,
the right-hand side of the above equals
\begin{align*}
    \oE_{h_1\in\Z}\abs{\E_{n\in \Z}f_1\cdot U_{(h_1,0,0)}\bar f_1\cdot U_{(-n,n^2,n^3)}f_2\cdot U_{(-n,(n+h_1)^2,n(n+h_1)^2)}\bar f_2}.
\end{align*}
By the Cauchy-Schwarz inequality, this is at most
\begin{align*}
    \oE_{h_1\in\Z}\norm{\E_{n\in \Z} U_{(-n,n^2,n^3)}f_2\cdot U_{(-n,(n+h_1)^2,n(n+h_1)^2)}\bar f_2}_{L^2(\mu)},
\end{align*}
and the new average corresponds to the family
\begin{align*}
        \CP_1 = \{(-n, n^2, n^3),\; (-n, (n+h_1)^2, n(n+h_1)^2)\}.
    \end{align*}
    We repeat this argument, applying the van der Corput inequality, composing the new average with $U_{(-n,n^2,n^3)}\inv$, and applying the Cauchy-Schwarz inequality. This gives the average
    \begin{multline}\label{E: average over h_1, h_2 0}
        \oE_{h_1,h_2\in\Z}\left|\!\left|\E_{n\in\Z}U_{(0, 2h_1 n +h_1^2,0)}\bar f_2\cdot U_{(-h_2,2h_2 n +h_2^2,h_2(n+h_2)^2)}\bar f_2\right.\right.\\
    \left.\left. U_{(-h_2, 2(h_1+h_2) n +(h_1+h_2)^2, h_2(n+h_1+h_2)^2)}f_2\vphantom{\E_{n\in[N]}}\right|\!\right|_{L^2(\mu)}
    \end{multline}
    involving the family
    \begin{multline*}
        \CP_2 = \{(0, 2h_1 n +h_1^2,0),\; (-h_2,2h_2 n +h_2^2,h_2(n+h_2)^2),\\ (-h_2, 2(h_1+h_2) n +(h_1+h_2)^2, h_2(n+h_1+h_2)^2)\}.
    \end{multline*}

    In the abelian case, i.e. when $R=\Id_X$, the family would reduce to
    \begin{align*}
        \{(0, 2h_1 n +h_1^2),\; (-h_2,2h_2 n +h_2^2),\; (-h_2, 2(h_1+h_2) n +(h_1+h_2)^2)\},
    \end{align*}
    and so for every fixed pair $(h_1,h_2)\in\Z^2$ with $h_1, h_2$ distinct and nonzero, we could control the corresponding average
    \begin{align*}
        \E_{n\in\Z}S^{2h_1 n +h_1^2}\bar f_2\cdot T^{-h_2}S^{2h_2 n +h_2^2}\bar f_2\cdot  T^{-h_2}S^{ 2(h_1+h_2) n +(h_1+h_2)^2}f_2\vphantom{\E_{n\in[N]}}
    \end{align*}
    by $\nnorm{f_2}_{3,S}$ using known Host-Kra seminorm estimates for the single-transformation case (just like it was done by Chu, Frantzikinakis, and Host in \cite{CFH11}).
    However, in the 2-step nilpotent setting, the presence of commutators complicates things in two ways. First, a seminorm estimate in terms of $T^{-h_2}f_2$ does not translate into a seminorm estimate in terms of $f_2$ due to the noncommutativity of $T,S$. Second, and more importantly, the presence of commutators $R$ means that we can only control each average 
    \begin{multline}\label{E: average over h_1, h_2}
        \E_{n\in\Z}U_{(0, 2h_1 n +h_1^2,0)}\bar f_2\cdot U_{(-h_2,2h_2 n +h_2^2,h_2(n+h_2)^2)}\bar f_2\\
        U_{(-h_2, 2(h_1+h_2) n +(h_1+h_2)^2, h_2(n+h_1+h_2)^2)}f_2
    \end{multline}
    by a box seminorm involving some powers of $R, S$ depending on $h_1, h_2$.

      To see how these complications play out, let us see what estimates Theorems \ref{T: original estimates} and \ref{T: new estimates} provide. To this end, rewrite (\ref{E: average over h_1, h_2}) as
     \begin{multline}\nonumber
        \E_{n\in\Z}U_{(0, 2h_1 n,0)}(U_{(0,h_1^2,0)}\bar f_2)\cdot U_{(0,2h_2 n,h_2(n+h_2)^2-h_{2}^{3})}(U_{(-h_{2},h_{2}^{2},h_{2}^{3})}\bar f_2)\\
        U_{(0, 2(h_1+h_2) n, h_2(n+h_1+h_2)^2-h_{2}(h_{1}+h_{2})^{2})}(U_{(-h_{2},(h_{1}+h_{2})^{2},h_{2}(h_{1}+h_{2})^{2})}f_2).
    \end{multline}
    If we apply the existing estimates in the form of Theorem~\ref{T: original estimates} to the first term (the only one with no contribution from $T$), we obtain control in terms of
    \begin{align}\label{E: naive control}
     \nnorm{U_{(0,h_1^2,0)} f_2}_{(0,2h_1,0)^{\times s}, (0,0,-h_{2})^{\times s}}=\nnorm{f_2}_{(0,2h_1,0)^{\times s}, (0,0,-h_{2})^{\times s}}
    \end{align}
     for some $s\in\N$. This does not seem particularly useful, as we know no way of connecting a seminorm in terms of $R$ to a seminorm in terms of $S$. Applying Theorem~\ref{T: original estimates} to the second term (and similarly to the third term) is not helpful, either, since we will get a control by
      $\nnorm{U_{(-h_{2},h_2^2,h_{2}^{3})} f_2}_{(0,0,h_2)^{\times s}, (0,-2h_{1},-2h_{1}h_{2})^{\times s}}$, which may not be the same as $\nnorm{f_2}_{(0,0,h_2)^{\times s}, (0,-2h_{1},-2h_{1}h_{2})^{\times s}}$ since $T$ and $S$ are not commuting.

      To overcome this difficulty, instead we use Theorem~\ref{T: new estimates}, the new seminorm estimate of the paper.  Theorem~\ref{T: new estimates} provides control in terms of
    \begin{align}\label{E: less naive control}
        \nnorm{f_2}_{\substack{(0,2h_1,0)^{\times s}, \langle (0,0,h_2), (0,2h_2,2h_2^2)\rangle^{\times s}, \langle (0,0,h_2), (0,2(h_{1}+h_{2}),2h_{2}(h_{1}+h_{2}))\rangle^{\times s},\\
        \langle (0,0,h_2), (0,{2(h_2-h_1)},2h_2^2)\rangle^{\times s}, \langle (0,0,h_2), (0,2h_{2},2h_2(h_1+h_2))\rangle^{\times s}, (0,2h_1,2h_1h_2)^{\times s}}}.
    \end{align}
    A priori, this looks rather formidable. However, every group appearing in the expression above has a generator of the form $(0,a(h),b(h))$, where $h=(h_{1},h_{2})$ and $a(h)$ is nonzero and is not a $\Q$-multiple of $b(h)$. Now, if we had a quantitative relationship between \eqref{E: average over h_1, h_2} and \eqref{E: less naive control}, it would imply that
    \begin{align}\label{E: naive control averaged}
\oE_{h_1, h_2\in \Z}\nnorm{f_2}_{\substack{(0,2h_1,0)^{\times s}, \langle (0,0,h_2), (0,2h_2,2h_2^2)\rangle^{\times s}, \langle (0,0,h_2), (0,2(h_{1}+h_{2}),2h_{2}(h_{1}+h_{2}))\rangle^{\times s},\\
        \langle (0,0,h_2), (0,{2(h_2-h_1)},2h_2^2)\rangle^{\times s}, \langle (0,0,h_2), (0,2h_{2},2h_2(h_1+h_2))\rangle^{\times s}, (0,2h_2,{0})^{\times s}}}
    \end{align}
    controls \eqref{E: average over h_1, h_2 0}. By concatenating \eqref{E: naive control averaged} as in e.g. \cite{DFKS22, DKKST24} and using the aforementioned property of the generators $(0,a(h),b(h))$ we could control \eqref{E: naive control averaged} in terms of $\nnorm{f_2}_{s, S}$ for some $s\in\N$. 

    Unfortunately, our Theorem~\ref{T: new estimates} provides no quantitative relationship between \eqref{E: less naive control} and \eqref{E: average over h_1, h_2}; hence the concatenation-based strategy above cannot be implemented. Instead, we proceed as follows. First, we apply the van der Corput inequality two more times to \eqref{E: average over h_1, h_2 0} to kill the contributions of $T$ altogether. Second, we repeat the van der Corput operation three more times, this time to kill terms in which the second coordinate is constant in $n$ while the third coordinate depends nontrivially on $n$. Third, we run one more van der Corput, this time to turn the average into one where the parameters $n, h_1, h_2,\ldots$ are on equal footing (it will become clear what this means as soon as we reach there). Lastly, we apply Theorem~\ref{T: new estimates} to the resulting multiparameter polynomial family. 
    
    Let us start with the first of these four steps. The application of the van der Corput inequality to \eqref{E: average over h_1, h_2 0}, composing with $U_{(0, 2h_1n+h_1^2, 0)}\inv$, and applying the Cauchy-Schwarz inequality brings us to the family
    \begin{align*}
        \CP_3 = \{&(-h_2, 2(h_2-h_1)n+h_2^2-h_1^2,h_2(n+h_2)^2),\\ 
        &(-h_2, 2(h_2-h_1)n+2h_2h_3+h_2^2-h_1^2, h_2(n+h_2+h_3)^2),\\ 
        &(-h_2, 2h_2n+2h_1h_2+h_2^2, h_2(n+h_1+h_2)^2),\\ 
        &(-h_2, 2h_2n+2(h_1+h_2)h_3+2h_1h_2+h_2^2, h_2(n+h_1+h_2+h_3)^2)\}.
    \end{align*}
    Yet another repetition of the same procedure yields
    \begin{align*}
        \CP_4 = \{&({0,2(h_2-h_1)h_4, 2h_2h_4(n+h_1) + h_2h_4^2),}\\
        &{(0, 2h_2h_3, 2h_2h_3n + h_2 h_3^2),}\\
        &{(0, 2(h_2-h_1)h_4 + 2h_2h_3, 2h_2(h_3+h_4)n + h_2 (h_3+h_4)^2+2h_1h_2h_4),}\\
        &(0, 2h_1n+2h_1h_2+h_1^2,0),\\ 
        &(0,2h_1n+2h_2h_4+2h_1h_2+h_1^2, 2h_2h_4(n+h_1)+h_2h_4^2),\\ 
        &(0, 2h_1n+2h_1h_3+2h_2h_3+2h_1h_2+h_1^2,2h_2h_3n + h_2h_3^2),\\ 
        &(0, 2h_1n+2h_1h_3+2h_2h_3+2h_1h_2+2h_2h_4+h_1^2,\\
        &\qquad\qquad 2h_2(h_3+h_4)n+h_2h_4(2h_1+2h_3+h_4)+h_2h_3^2)\}.
    \end{align*}
    The family $\CP_4$ does not depend on $T$, so we are now in the abelian setting. However, since we ultimately care about an estimate in terms of $S$, we need to deal with the issue that the first three members of $\CP_4$ have their second coordinates independent of $n$. 
    
    To fix this, we apply three more van der Corputs to the effect of killing the first three members of $\CP_4$. Name the terms of $\CP_4$ as $$(0,t_j(h)n + u_j(h),v_j(h)n+w_j(h))$$ for $1\leq j\leq 7$. Then the combined effect of these van der Corputs is the family
    \begin{multline*}
        \CQ := \{(0, t_j(h)(n +\eps_5h_5 + \eps_6 h_6 +\eps_7h_7) + u_j(h)-u_3(h),\\
        (v_j(h)-v_3(h))n + v_j(h)\eps_5h_5 + (v_j(h)-v_1(h))\eps_6 h_6 +(v_j(h)-v_2(h))\eps_7h_7\\ + w_j(h)-w_3(h))\colon\;
        4\leq j\leq 7,\; \eps_5,\eps_6,\eps_7\in\{0,1\}\}.
    \end{multline*}
    Since $t_j(h) = 2h_1$ for $4\leq j\leq 7$, we simply denote this quantity by $t(h)$. We can then recast $\CQ$ in the more compact form
    \begin{align*}
        \CQ = \{(0, t(h)n + b_j(h), c_j(n,h))\colon 1\leq j\leq 32\}, 
    \end{align*}
    where now $h = (h_1, \ldots, h_7)$. Hence the 128-th power of the $L^2(\mu)$ norm of \eqref{E: n, n^2 recast} is bounded by
    \begin{align}\label{E: average example 1}
        \oE_{h_1,\dots, h_7\in\Z}\left|\!\left|\E_{n\in\Z}\prod_{j=1}^{32} U_{(0, t(h)n + b_j(h),c_j(n,h))} f_2\vphantom{\E_{n\in[N]}}\right|\!\right|_{L^2(\mu)}
    \end{align}
    (with complex conjugates over some $f_2$'s that we choose to ignore for the rest of the proof). 

    It is rather inconvenient that the average over $n$ is inside the norm while the averages over $h$'s are outside, and that the limits are iterated. We want to convert \eqref{E: average example 1} into a single average over all variables so that we can think of elements of $\CQ$ as polynomials in $n, h_1, \ldots, h_7$ rather than polynomials in $n$ whose coefficients depend on $h_1, \ldots, h_7$.
    This can be done as follows. First, we apply the symmetric van der Corput inequality, 
     bounding the square of \eqref{E: average example 1} by
    \begin{multline}\label{E: average example 2}
        \oE_{h_1,\ldots, h_7,u_0,u_1\in\Z}\E_{n\in\Z}\\
        \int\prod_{j=1}^{32} \left(U_{(0, t(h)(n+u_0) + b_j(h),c_j(n+u_0,h))} f_2\cdot 
        U_{(0, t(h)(n+u_1) + b_j(h),c_j(n+u_1,h))}  f_2\right)\; d\mu.
    \end{multline}
    Using the uniform version of Walsh's convergence theorem \cite{W12} jointly with the Fubini-type principle of Bergelson and Leibman \cite[Lemma~1.1]{BL15}, we can replace the iterated limits by a single one, so that \eqref{E: average example 2} equals
    \begin{multline}\label{E: average example 3}
        \E_{(n,h_1, \ldots, h_7,u_0,u_1)\in\Z^{10}}\int\prod_{j=1}^{32} \left(U_{(0, t(h)(n+u_0) + b_j(h),c_j(n+u_0,h))} f_2\right.\\ 
        \left. \cdot U_{(0, t(h)(n+u_1) + b_j(h),c_j(n+u_1,h))}  f_2\right)\; d\mu.
    \end{multline}
    
    This is a properly multiparameter average that can be handled using Theorem~\ref{T: new estimates}. The key observation is that the groups
    \begin{multline}\label{E: example - first subgroup}
        \langle (0, t(h)(n+u_0) + b_j(h),c_j(n+u_0,h))\\ *(0, t(h)(n+u_1) + b_{j'}(h),c_{j'}(n+u_1,h))\inv\colon\;
        (n,h,u_0,u_1)\in\N^{10}\rangle 
    \end{multline}
    (for $1\leq j,j'\leq \ell$) and
    \begin{multline}\label{E: example - second subgroup}
        \langle (0, t(h)(n+u) + b_j(h),c_j(n+u,h))\\ * (0, t(h)(n+u) + b_{j'}(h),c_{j'}(n+u,h))\inv\colon\; (n,h,u)\in\N^9\rangle 
    \end{multline}
    (for $0\leq j<j'\leq \ell$) arising in the conclusion of Theorem~\ref{T: new estimates} both contain a nonzero multiple of $e_2 = (0,1,0)$. We refer to this crucial property as \textit{distinguishability}. 
    For the first subgroup \eqref{E: example - first subgroup}, this follows since the coordinate of the monomial $h_1u_0$ in 
    \begin{align*}
        (0, 2h_1(n+u_0) + b_j(h),c_j(n+u_0,h))* (0, 2h_1(n+u_1) + b_{j'}(h),c_{j'}(n+u_1,h))\inv
    \end{align*}
    is $2e_2$ (recall that $t(h) = 2h_1$). To see this, observe that all the monomials in the third coordinate  have total degree 3. 
    For the second subgroup \eqref{E: example - second subgroup}, this follows again from the fact that the second coordinate of 
    \begin{align*}
        (0, 2h_1(n+u) + b_j(h),c_j(n+u,h))* (0, 2h_1(n+u) + b_{j'}(h),c_{j'}(n+u,h))\inv
    \end{align*}
    is a nonzero linear combination of degree-2 monomials (this time, they do not involve $n$ or $u$) while the third coordinate is a linear combination of degree-3 monomials. Hence by linear algebra,\footnote{We are using here the basic fact that if $a_1, \ldots, a_\kappa\in\Z[\bn]$ are distinct monomials in an indeterminate $\bn\in \Z^L$, and $v_1, \ldots, v_\kappa\in\Z^D$, then the group
    \begin{align*}
        \langle v_1 a_1(\bn) + \cdots + v_\kappa a_\kappa(\bn)\colon\; \bn\in\Z^L\rangle
    \end{align*}
    contains some multiples of $v_1, \ldots, v_\kappa$.
    } some multiple of $e_2$ must lie in \eqref{E: example - second subgroup}. Using this observation, the subgroup property of box seminorms,\footnote{Since each of \eqref{E: example - first subgroup} and \eqref{E: example - second subgroup} contains a multiple of $e_2$, the seminorm along (many copies of) the subgroups \eqref{E: example - first subgroup} and \eqref{E: example - second subgroup} is bounded by the seminorm along (many copies of) $e_2$.} as well as Theorem~\ref{T: new estimates}, we obtain the claimed estimate.
\end{proof}

\section{2-step nilpotent PET: the general case}\label{S: PET}

We now present the 2-step nilpotent PET induction scheme that proves Theorem~\ref{T: nilpotent seminorm estimates}. To understand its goals and mechanics, we strongly advertise the reader to first familiarize themselves with the example in the previous section as well as the high-level explanation of the argument given in Section \ref{SS: outline}.

\subsection{Setup and local conventions}
Throughout Section \ref{S: PET}, 
 we fix a 2-step nilpotent system $(X, \CX, \mu, U)$ generated by the group $H = \langle T_1, \ldots, T_D\rangle$. 
For $1\leq k<l\leq D$, set $R_{k,l} := [T_k,T_l]$, and let $$D' :=D + \binom{D}{2}=\frac{D(D+1)}{2}$$ be (an upper bound on) the dimension of $H$. We then identify $H$ with $\Z^{D'}$ by defining the 2-step nilpotent group operation on $\Z^{D'}$ as follows. Take $h=(h_\ab,h_\nab)$ and $h'=(h'_\ab,h'_\nab)$, where
\begin{gather*}
 h_\ab = (h_1, \ldots, h_D),\quad h_\ab'= (h'_1, \ldots, h'_D),\\
 h_\nab = (h_{k,l}\colon 1\leq k<l\leq D),\quad h_\nab' = (h'_{k,l}\colon 1\leq k<l\leq D),   
\end{gather*}
and each coordinate is in $\Z$. 
Then define $h*h'$ coordinatewise by 
\begin{align*}
    (h*h)_k &:= h_k + h'_k\quad \textrm{for}\quad 1\leq k\leq D,\\
    (h*h')_{k,l} &:= h_{k,l}+h'_{k,l} -h'_k h_l \quad \textrm{for}\quad 1\leq k<l\leq D.
\end{align*}
By reordering $T_1, \ldots, T_D$ if needed, we can assume that $U$ acts on $(X, \CX, \mu)$ via
\begin{align}\label{E: group representation}
    U_h := \prod_{1\leq j\leq D} T_j^{h_{j}} \circ \prod_{1\leq k<l\leq D} R_{k,l}^{h_{k,l}}
\end{align}
for $h\in\Z^{D'}$.

    Whenever the order of the nonabelian part of $\Z^{D'}$ does not matter, we write elements of $\Z^{D'}$ or $\Z^{D'}[n,h]$ as $x = (x_1, \ldots, x_{D'})$, where $x_1, \ldots, x_D$ are the coordinates corresponding to $T_1, \ldots, T_D$ and $x_{D+1}, \ldots, x_{D'}$ are the coordinates corresponding to $R_{k,l}$ ordered in some arbitrary fashion. If we want to specify a particular coordinate, we use $x_{k,l}$ to denote the coordinate of $x$ corresponding to $R_{k,l}.$

Our goal now is to study subfamilies of $\Z^{D'}[n,h]$ 
that arise in the PET induction scheme for polynomial families appearing in the statement of Theorem~\ref{T: nilpotent seminorm estimates}. In what follows, $h$ is always an indeterminate taking values in $\Z^{r}$ for some $r\in\N_0$. 
\begin{definition}[Basic definitions for polynomials]\label{d61}
Write $p\in\Z^{D'}[n,h]$ as 
\begin{align*}
    p(n,h) = \sum_{j=0}^\infty b_j(h)n^j \quad \textrm{for some {finitely many nonzero}}\quad b_j\in \Z^{D'}[h].
\end{align*}
We define the \textit{degree} of $p$ to be the maximum $j\in\N_0$ for which $b_j$ is a nonzero polynomial. If no such $j$ exists, i.e. $p$ is identically 0, we say that $\deg p  = - \infty$. If $p$ has degree $d\geq 0$, then we call $b_d$ the \textit{leading coefficient} of $p$. We call $p$ \textit{constant} if it is constant in $n$, i.e. its degree is 0 or $-\infty$. 
We call two polynomials $p,q\in\Z^{D'}[n,h]$ \textit{essentially distinct} if $p*q\inv$ is nonconstant (i.e. it depends nontrivially on $n$). Since $(p*q\inv)\inv = q*p\inv$, this is equivalent to saying that $q*p\inv$ is nonconstant, or that $p(n,h) \neq  r(h)*q(n,h)$ for any $r\in \Z^{D'}[h]$. 
We call $p,q$ \textit{distinct} if $p*q\inv\neq 0$, otherwise we call them \textit{identical}.
    
\end{definition}

\begin{definition}[Abelian and nonabelian parts of a polynomial]
    A polynomial $p\in\Z^{D'}[n,h]$ can be written as $p=(p_\ab, p_\nonab)$, where we call 
    \begin{align*}
     p_\ab = (p_1, \ldots, p_D)\quad \textrm{and} \quad p_\nonab = (p_{k,l}\colon\; 1\leq k<l\leq D)   
    \end{align*}
    the \textit{abelian} and \textit{nonabelian parts} of $p$ respectively.
\end{definition}

Given a polynomial family $\CP = \{p_1, \ldots, p_\ell\}\subseteq \Z^{D'}[n,h]$,
write $p_{j}=(p_{j,1},\dots,p_{j,D'})$ 
and let $$d := \max_{1\leq j\leq \ell}\max_{1\leq i\leq D'}\deg p_{j,i}$$ be the maximum degree of polynomials appearing in $\CP$.

While discussing the complexity of the family $\CP$, we use the following definitions.
\begin{definition}[Degree classes]\label{D: degree classes}
    Let $\CP = \{p_1, \ldots, p_\ell\}\subseteq\Z^{D'}[n,h]$.
    For $1\leq k\leq D$, set
    $$\FI_k:=\rem{1\leq j\leq \ell\colon\; {p_{j,k}\; \textrm{is nonconstant while }} p_{j,i}\; \textrm{constant for}\; k<i\leq D}.$$
    Define also
    $$\FI_\nab :=\{1\leq j\leq \ell\colon\; {p_{j,\nonab}\; \textrm{is nonconstant while }} p_{j,\ab}\; \textrm{is constant}\}.$$
\end{definition}
Thus, $\FI_k$ captures those polynomials that witness {a contribution from $T_k$ but} no contribution from $T_{k+1}, \ldots, T_D$ 
while $\FI_\nab$ takes care of those that only see the commutators. 

\begin{example}
    Let $\CP = \{p_1, \ldots, p_{D'}\}$, where $p_j = q_je_j$ for some  nonconstant $q_1, \ldots, q_{D'}\in\Z[n,h]$ and $e_1, \ldots, e_{D'}$ are the coordinate vectors. Then 
    $\FI_k = \{k\}$ for all $1\leq k\leq D$ 
    and $\FI_\nab = \{D+1, \ldots, D'\}$.
\end{example}
\begin{definition}[Abelian type]
    Let $\CP = \{p_1, \ldots, p_\ell\}\subseteq\Z^{D'}[n,h]$.
    For $1\leq k\leq D$, let $\omega_{k} := (\omega_{k, 1}, \ldots, \omega_{k, d})$, where for each $1\leq i\leq d$, the quantity $\omega_{k,i}$ denotes the number of distinct leading coefficients of degree $i$ among $\{p_{j,k}\colon\; j\in\FI_k\}$. 
    We then let $\omega_\ab := (\omega_{1}, \ldots, \omega_{D})$ denote the \textit{abelian type} of $\CP$.
\end{definition}

\begin{definition}[Nonabelian type]
     Let $\CP = \{p_1, \ldots, p_\ell\}\subseteq\Z^{D'}[n,h]$. For $1\leq i\leq d$, let $\omega_{\nab, i}$ be the number of distinct leading coefficients of degree $i$ among $$\{p_{j,\nab}\colon\; j\in\FI_\nab\}.$$ Then $\omega_{\nab} := (\omega_{\nab, 1}, \ldots, \omega_{\nab, d})$ denotes the \textit{nonabelian type} of $\CP$.
\end{definition}

An attentive reader may note that the definitions above describe the abelian parts of polynomial with significantly greater precision than their nonabelian parts. For instance, while the abelian type has a separate entry for each coordinate $1\leq k\leq D$, this is not the case for the nonabelian type. This is because the PET argument below does not require us to keep track of the nonabelian part of the polynomials to the same extent as it does for the abelian part.

\begin{definition}[Type]
    Let $\CP = \{p_1, \ldots, p_\ell\}\subseteq\Z^{D'}[n,h]$. Then we call $\omega := (\omega_\ab, \omega_\nab)$ its \textit{type}. 
    We call a type \textit{nondegenerate} if $\omega_\ab \neq 0$.  We call a type \textit{basic} if
    \begin{align*}
        \omega_\nonab = \omega_{1}=\dots=\omega_{D-1}={0}\quad \textrm{while}\quad \omega_{D}=(1,0,\dots,0).
    \end{align*} 
\end{definition}

We emphasize here that the notion of types in this section has nothing to do with the corresponding notion from Section \ref{S: seminorm estimates}. 

We order types as follows:
\begin{enumerate}
    \item we let $\omega_{k}>\omega_{k}'$ if there exists $1\leq d'\leq d$ such that $\omega_{k,d'}>\omega_{k,d'}'$ while $\omega_{k,i} = \omega_{k,i}'$ for $d'<i\leq d$; 
        \item we let $\omega_\ab>\omega'_\ab$ if there exists $1\leq k\leq D$ such that $\omega_{k}>\omega_{k}'$ while $\omega_{i} =\omega_{i}'$ for all $k<i\leq D$;
    \item we let $\omega_{\nab}>\omega_{\nab}'$ if there exists $1\leq d'\leq d$ such that $\omega_{\nab,d'}>\omega_{\nab,d'}'$ while $\omega_{\nab,i} = \omega_{\nab,i}'$ for $d'<i\leq d$;
    \item we let $\omega > \omega'$ if either $\omega_\ab > \omega'_\ab$, or $\omega_\ab = \omega_\ab'$ and $\omega_\nab > \omega_\nab'$.
\end{enumerate}
 The ordering on types reveals that the abelian part takes precedence over the nonabelian part in determining the type. The abelian type is ordered collexicographically: within it, the contribution of $T_D$ weighs more than the contribution of $T_{D-1}$, which in turn is more important than the contribution of $T_{D-2}$, etc. For fixed $1\leq k\leq D$, higher degree terms matter more than the lower degree terms. The nonabelian type is likewise ordered collexicographically.

\begin{example}\label{Ex: nilpotent PET type}
    To illustrate the definitions above, we present the simplest nontrivial example of a polynomial family covered by our setup. Take $D=2$ and consider 
    the average
    \begin{align*}
        \E_{n\in\Z}T^n f_1\cdot S^{n^2}f_2
    \end{align*}
    studied in Theorem~\ref{T: nilpotent joint ergodicity n n^2}. For $0\leq i\leq 4$, let $\omega^{(i)} = (\omega^{(i)}_\ab, \omega^{(i)}_\nonab)$ be the type of the families $\CP_i$ arising in the proof of Theorem~\ref{T: nilpotent joint ergodicity n n^2}. Then their types are as follows:
    \begin{enumerate}
        \item $\omega^{(0)}_\ab = ((1,0), (0,1))$ and $\omega^{(0)}_\nab = 0$; here, $(1,0)$ corresponds to $T^n$ and $(0,1)$ refers to $S^{n^2}$;
        \item $\omega^{(1)}_\ab = ((0,0), (0,1))$ and $\omega^{(1)}_\nab = 0$; the change in type compared to $\CP_0$ corresponds to $e_2 n^2$ now being the leading term of the abelian parts of both terms;
        \item $\omega^{(2)}_\ab = ((0,0), (3,0))$ and $\omega^{(2)}_\nab = 0$; now, the leading terms of the abelian parts are $2e_2 h_1 n$, $2e_2 h_2 n$, $2e_2 (h_1+h_2)n$;
        \item $\omega^{(3)}_\ab = ((0,0), (2,0))$ and $\omega^{(3)}_\nab = 0$; we have reduced to only two distinct linear terms of the abelian parts, $2e_2(h_2-h_1)n$ and $2e_2h_2n$;
        \item $\omega^{(4)}_\ab = ((0,0), (1,0))$ and $\omega^{(4)}_\nab = 3$; the only nonconstant leading term of the abelian part is $2e_2 h_1 n$, but now we have three terms with constant abelian parts and nonconstant nonabelian parts with three different leading terms: $2e_3h_2h_4n$, $2e_3h_2h_3n$, and $2e_3h_2(h_3+h_4)n$.
    \end{enumerate}
    The family $\CQ$ obtained after three more van der Corput operations has basic type; each of these operations lowers the nonabelian type by exactly one. 
\end{example}

We aim at seminorm estimates in terms of the transformation $T_D$, hence in polynomial families under consideration, we need to isolate this coordinate from the other ones. The precise property needed is captured in the following definition; its utility will become evident later on.
\begin{definition}[Distinguishable polynomials]\label{D: distinguishable}
    Let $p=(p_1, \ldots, p_{D'})\in\Z^{D'}[n,h]$ be a polynomial such that $\deg p_D = d'$. We say that $p$ is \emph{$D$-distinguishable} if 
    $p_D\neq  0$ and the polynomial\footnote{If $\deg p_D = 0$, i.e. $p_D(n,h) = q(h)$ for a nonzero $q\in\Z^{D'}[h]$, then we interpret $\Delta_{m_{d'}}\dots\Delta_{m_{1}}p$ as $p$.} $(\Delta_{m_{d'}}\dots\Delta_{m_{1}}p)_{D}\in\Z[n,h,m_{1},\dots,m_{d'}]$ is not a $\Q$-linear combination of 
    \begin{align*}
        \rem{(\Delta_{m_{d'}}\dots\Delta_{m_{1}}p){_{k,l}\colon\; 1\leq k<l\leq D}}.
    \end{align*}
    Here $\Delta_{m}p(n,h):=(\sigma_{m}p\ast p^{-1})(n,h)$ { and $\sigma_{m}p(n,h):=p(n+m,h)$.}
\end{definition}

For a polynomial $p\in\Z^{D'}[n,h]$, define
\begin{align}\label{E: tilde G}
    {G}(p):=\langle p(n,h)\colon\; (n,h)\in \Z^{r+1}\rangle. 
\end{align}

\begin{lemma}\label{lnh}
If $p$ is a $D$-distinguishable polynomial whose first $D-1$ coordinates are 0, then {some nonzero multiple of} $e_{D}$ is contained in ${G}(p)$.
\end{lemma}
 \begin{proof} 
 For $q\in\Z[n,h]$, set $\partial_mq(n,h):=q(n+m,h)-q(n,h)$.
 The key observation is that since the first $D-1$ coordinates of $p$ are 0, we are in the commuting case, i.e.
 \begin{align}\label{E: iterated derivatives of i,D}
     (\Delta_{m_{d}}\dots\Delta_{m_{1}}p)_{k,l} = \partial_{m_{d}}\dots\partial_{m_{1}}p_{k,l}
 \end{align}
 for any $1\leq k<l\leq D$.
 For instance, 
 \begin{align*}
     (\Delta_m p)_{k,l}(n,h) &= (p(n+m,h)*p(n,h)\inv)_{k,l}\\
     &= p_{k,l}(n+m,h)-p_{k,l}(n,h)\\
     &= (\partial_m p_{k,l})(n,h),
 \end{align*} 
 i.e. there are no extra terms coming from the commutators. It is then easy to see by induction that if 
 \begin{align*}
     p_D = \sum_{1\leq k < l\leq D}c_{k,l} \cdot p_{k,l}
 \end{align*}
 for some $c_{k,l}\in\Q$, then
 \begin{align*}
     \partial_m p_D &= \sum_{1\leq k < l\leq D}c_{k,l} \cdot \partial_m p_{k,l},\\
     \partial_{m_2}\partial_{m_1} p_D &= \sum_{1\leq k < l\leq D}c_{k,l} \cdot \partial_{m_2}\partial_{m_1} p_{k,l},\\
     &\vdots\\
     \partial_{m_{d'}}\dots\partial_{m_{1}}p_D &=\sum_{1\leq k < l\leq D}c_{k,l} \cdot \partial_{m_{d'}}\dots\partial_{m_{1}} p_{k,l},
 \end{align*}
 which by the identity \eqref{E: iterated derivatives of i,D} contradicts the $D$-distinguishability of $p$.
Hence $p_D$ is $\Q$-linearly independent from $\{p_{k,l}\colon 1\leq k<l\leq D\}$, and so ${G}(p)$ must contain {some nonzero multiple of} $e_{D}$.
 \end{proof}

Our seminorm estimates in this section will work for families satisfying the following condition, which is inspired by - but goes significantly beyond - a similar definition in \cite[Section 5.2.1]{CFH11}. 
\begin{definition}[Nice families]\label{D: nice families}
Denote $p_{0}:=0$.
    We call the polynomial family $\CP= \{p_1, \ldots, p_\ell\}\subseteq\Z^{D'}[n, h]$ \textit{nice} if it satisfies the following conditions: 
    \begin{enumerate}
        \item $\FI_D$ is nonempty (so in particular $\CP$ is nondegenerate);
        \item $\deg (p_{j,D}-p_{j',D}) > \deg (p_{j,k}-p_{j',k})$ for all $j\in\FI_D$, ${0}\leq j'\leq \ell$ distinct from $j$, and $1\leq k <D$;
        {\item $\deg (p_{j,D}-p_{j',D})\geq \max(\deg p_{j,D}, \deg p_{j',D})-1$ for all distinct $j,j'\in\FI_D$;}
        \item $p_{j}\ast p_{j'}^{-1}$ is $D$-distinguishable for all $j\in \FI_D$ and $0\leq j'\leq \ell$ distinct from $j$. 
    \end{enumerate}
    We refer to these conditions as the \textit{first, second, third}, and \textit{fourth nice property}.
\end{definition}
We remark that members of nice families need not be essentially distinct in the sense considered in this paper, i.e. up to $\Z^{D'}[h]$ terms. This is one of many notable differences between the abelian and nilpotent PETs.

The conditions above play the following role:
\begin{enumerate}
    \item The first nice property ensures that some $p_j$ witnesses contribution from $T_D$.
    \item The second nice property ensures that as we run the van der Corput operations defined below, the terms with a nontrivial contribution from $T_D$ remain of higher degree than the terms with contributions from $T_1, \ldots, T_{D-1}$ but not $T_D$. In the long run, this allows us to annihilate the contributions of $T_1, \ldots, T_{D-1}$ while maintaining a contribution from $T_D$. This condition is modeled on the analogous properties constituting the definition of niceness in \cite[Section 5.2.1]{CFH11}. One divergence from the latter work is that over there, it sufficed to choose a single $j\in\FI_D$ in the definition of niceness. We need to modify this condition because we will invoke Theorem~\ref{T: new estimates} for an average with $\FI_D = \{1, \ldots, \ell\}$, and these seminorms require us to compare all polynomials with each other.
    \item {The last two conditions are specific to the nilpotent setting and have no analogs in \cite{CFH11} or any other abelian PET arguments. They guarantee that the contribution of $T_D$ in the polynomials in $\FI_D$ can always be disentangled from the contributions of the commutators, in that the $D$-th coordinate of $p_j*p_{j'}\inv$ is $\Q$-independent from the coordinates coming from the commutators.\footnote{We do not know whether this exact property remains preserved by the van der Corput operation, so instead we impose the third and fourth nice properties, both of which pass from one stage of PET to another. {These two conditions are stronger than they should be; for instance, we believe that the average \begin{align*}
        \E_{n\in\Z} T_1^{n}f_1\cdot T_2^{n^4 + n^3}f_2\cdot T_2^{n^4 + n^3 + n^2}f_3
    \end{align*}
    also admits Host-Kra seminorm estimates even though it does not satisfy the third nice property. However, for our applications, the third and fourth nice properties are good enough.}
    } 
    This ensures that our final seminorm estimate can be framed purely in terms of $T_D$ rather than a mixture of $T_D$ and commutators.} 
\end{enumerate}

     As we run the PET argument, we want to make sure that an initial function associated with the polynomial of highest-degree $D$-th coordinate survives. The next definition allows us to keep track of this function throughout the whole process.
     \begin{definition}[Pinned polynomial family]
     A \emph{pinned (polynomial) family} is a pair $(\mathcal{P},p_{j})$ where $\mathcal{P}=\{p_{1},\dots,p_{\ell}\}\subseteq\Z^{D'}[n,h]$ is a family of distinct polynomials and $j\in \{1,\dots,\ell\}$. We say that $(\mathcal{P},p_{j_0})$ is \emph{nice} if the following two conditions hold:
     \begin{enumerate}
         \item {$\CP$ is nice;}
         \item $\deg p_{j_0,D}\geq \deg p_{j,D}$ for all $1\leq j\leq \ell$.    
     \end{enumerate}
          We say that $(\CP, p_{j_0})$ is \textit{nondegenerate/basic} if $\CP$ is, and we let the (abelian/nonabelian) type of $(\CP, p_{j_0})$ be the (abelian/nonabelian) type of $\CP$.
     \end{definition}

The next lemma ensures the niceness of the families of interest to us.
\begin{lemma}\label{L: niceness}
    Let $q_1, \ldots, q_D\in\Z[n]$ be nonconstant polynomials of increasing degrees. Set $q_{j,\eps}(n,h):=q_j(n+\eps\cdot h)$ and $p_{j,\eps}:=q_{j,\eps}e_j$ for $\eps\in\{0,1\}^{r}$ and $1\leq j \leq D$, and define
    \begin{align*}
        \CP := \{p_{j,\eps}\colon\; 1\leq j\leq D,\; \eps\in\{0,1\}^{r}\}.
    \end{align*}
     {Then $(\CP, p_{D,\eps})$ is nice for every $\eps\in\{0,1\}^{r}$.}
\end{lemma}
\begin{proof}
    Observe first that $\FI_k = \{(k,\eps)\colon\; \eps\in\{0,1\}^{r}\}$ for every $1\leq k\leq D$. Clearly, $\FI_D$ is nonempty and
    \begin{align}\label{E: fourth}
        \deg (q_{D,\eps} - q_{D,\eps'}) = \deg q_{D} - 1
    \end{align}
    for all distinct $\eps\neq \eps'$. This implies the first and third nice properties.
    The second nice property follows from \eqref{E: fourth} and the fact that $q_j$'s have increasing degrees.
     For the fourth nice property, we observe that $p_{D,\eps}$, $p_{D,\eps'}$ commute for any two distinct $\eps,\eps'\in\{0,1\}^{r}$, so the nonabelian part of their differences is trivial. If by contrast we look at $p_{D,\eps}*p_{j,\eps'}\inv$ for some $1\leq j<D$ and (not necessarily distinct) $\eps,\eps'\in\{0,1\}^{r}$, then it has a nontrivial commutator only at the index $(j,D)$. That coordinate equals $-q_{D,\eps}q_{j,\eps'}$, so it is a polynomial of higher degree than $q_{D,\eps}$ and hence the two cannot be scalar multiples of each other.

     Thus, the family $\CP$ is nice. The second condition from the definition of nice pinned families clearly follows from the assumption on the polynomials, yielding the niceness of $(\CP, p_{D,\eps})$ for every $\eps\in\{0,1\}^{r}$.
\end{proof}

Our main theorem in this section takes the following form.
\begin{theorem}[Host-Kra seminorm control for nice families]\label{T: nilpotent seminorm estimates for nice families}
    Let $d, D,\ell\in\N$ and $r\in\N_0$. There exists a natural number $s = O_{d,\ell, r}(1)$ such that for any 2-step nilpotent system $(X, \CX, \mu, U)$ given by \eqref{E: group representation}, {nice pinned family $(\CP, p_{\ell})$ of degree at most $d$} (where $\CP := \{p_1, \ldots, p_\ell\}\subseteq\Z^{D'}[n,h]$), and functions $f_1, \ldots, f_\ell\in L^\infty(\mu)$, 
    \begin{align*}
        \oE_{h\in \Z^{r}}\norm{\E_{n\in\Z}U_{p_1(n,h)}f_1\cdots U_{p_\ell(n,h)}f_\ell}_{L^2(\mu)} = 0
    \end{align*}
    whenever $\nnorm{f_\ell}_{s,T_D} = 0$. 
\end{theorem}

\subsection{Complexity reduction via the van der Corput operation}

All PET arguments rely on complexity reduction, accomplished via applying the van der Corput inequality to every intermediate average. In our setting, this induces the following operation on polynomial families under consideration.
\begin{definition}[Van der Corput operation]\label{D: vdC operation}
Given $\CP= \{p_1, \ldots, p_\ell\}\subseteq\Z^{D'}[n, h]$ and an index $1\leq m\leq \ell$, we define the new polynomial family
      \begin{align*}
            \partial_m\CP := &\{\sigma_{h_{r+1}}{p}_1 *{p}_m\inv, \ldots, \sigma_{h_{r+1}}{p}_\ell * {p}_m\inv, {p}_1*{p}_m\inv, \ldots,  {p}_\ell*{p}_m\inv\}^\sharp,
      \end{align*}
      where
           \begin{align*}
          \sigma_{h_{r+1}}p(n, h) := p(n+h_{r+1}, h),
      \end{align*} while the $^\sharp$ operation removes the polynomials independent of $n$.
      However, unlike in the commutative setting, we do not group polynomials that are not essentially distinct.
     We refer to $\CP\mapsto \partial_m \CP$ as a \textit{van der Corput operation}.  

     If $(\CP, p_{j_0})$ is a pinned polynomial family, we define the van der Corput operation via 
     $$\partial_{m}(\mathcal{P},p_{j_0}):=(\partial_{m}\mathcal{P},\sigma_{h_{r+1}}p_{j_0}\ast p_{m}^{-1})$$ as long as $\sigma_{h_{r+1}}p_{j_0}\ast p_{m}^{-1}$ is not constant.
     \end{definition}

We multiply by $p_m\inv$ on the right because on composing an integral under consideration with $U_{p_m}\inv$, we act on every term via
\begin{align*}
    U_{p_m}\inv (U_p f) 
    = U_{p*p_m\inv}f.
\end{align*}

       We choose the index $1\leq m\leq \ell$ in the van der Corput operation according to the algorithm provided by the following definition.

       \begin{definition}[Good index for the van der Corput operation]
       Let $\CP= \{p_1, \ldots, p_\ell\}\subseteq \Z^{D'}[n, h]$, and suppose that $(\CP,p_{j_{0}})$ is a nice pinned polynomial family. We say that $1\leq m\leq \ell$ is \textit{good for the van der Corput operation} if it is chosen as follows:
        \begin{enumerate}
            \item if $\omega_\nab \neq 0$, or equivalently $\FI_\nab$ is nonempty, then $m\in\FI_\nab$ and moreover $\deg p_{m,\nab}$ has the least degree among all possible such choices;
            \item if $\omega_\nab = 0$, i.e. $p_{j,\ab}$ is nonconstant for every $1\leq j\leq \ell$, then $m\in \FI_k$ for the smallest $1\leq k\leq D$ for which $\omega_k \neq 0$ (such $k$ exists since $\CP$ is nondegenerate), and moreover $m$ is picked so that $p_{m,k}$ has the smallest degree among $\{p_{j,k}\colon\; j\in\FI_k\}$;
            \item in the second case, if additionally $k=D$ and all $p_{1,D},\dots,p_{\ell,D}$ have the same degree, then we pick $m$ for which $p_{m,D}$ has a different leading coefficient from $p_{j_{0},D}$ if such a term exists (otherwise we take any $m$). 
        \end{enumerate}
       \end{definition}

        By taking $m$ good for the van der Corput operation, we can lower the type of the polynomial family.
       \begin{lemma}\label{L: type reduction}
         Let $\CP= \{p_1, \ldots, p_\ell\}\subseteq\Z^{D'}[n, h]$ be a polynomial family, and take $1\leq m\leq \ell$ \textit{good for the van der Corput operation}. Then $\partial_m\CP$ has smaller type than $\CP$.   
       \end{lemma}
       \begin{proof}
        Let $\omega = (\omega_\ab, \omega_\nab)$ be the type of $\CP$. In the first scenario, i.e. when $m\in\FI_\nab$, the choice of $m$ ensures that the new family has type $(\omega_\ab, \omega_\nab')$ for some $\omega_\nab'<\omega_\nab$, as it lowers the $(\deg p_{m,\nab})$-coordinate of $\omega_\nab$ while keeping all the higher ones intact. In the second and third case, i.e. when $m\in \FI_k$ for the smallest $1\leq k\leq D$ for which $\omega_k \neq 0$, the new family has type $(\omega_\ab', \omega_\nab')$ where $\omega'_{k} < \omega_{k}$ (because $\omega'_{k,d'}<\omega_{k,d'}$ for $d' = \deg p_{m,k}$ and $\omega'_{k,i}=\omega_{k,i}$ for $d'<i\leq d$) and $\omega'_{i} = \omega_{i}$ for $k<i\leq D$; hence $\omega'_\ab <\omega_\ab$. 
        \end{proof}

To see how Lemma~\ref{L: type reduction} works in practice, consult Example \ref{Ex: nilpotent PET type}.

\begin{proposition}\label{P: niceness preserved}
    {Let $(\CP,p_{j_{0}})$ be a nice pinned polynomial family, suppose that $\CP$ is not basic, and let $m$ be a good index for the van der Corput operation. Then $\partial_{m}(\CP,p_{j_{0}})$ is also a nice pinned polynomial family.}
\end{proposition}

\begin{proof}
    By construction, the members of $\partial_m \CP$ are nonconstant, so we need to establish the nice properties from Definition \ref{D: nice families} for $\partial_m \CP$ and then show that $\sigma_{h_{r+1}}p_{j_0}*p_m\inv$ is nonconstant and has the highest-degree $D$-th coordinate among the members of $\partial_{m}\CP$.

        We start with the first nice property. That is, we want to show that at least one polynomial from $\partial_m\CP$ has a nonconstant contribution at coordinate $D$. The only way this can fail is if all the polynomials 
        $p_{j,D}$ are linear with the same leading coefficient, i.e. $\omega_{D}=(1,0,\dots,0)$.
        By the second nice property, we must have $\omega_{1}=\dots=\omega_{D-1}={0}$.
        This is precluded by the assumption that $\CP$ is not basic.

    We move on to prove the second nice property. Notice that for any $1\leq j,j'\leq \ell$ and $1\leq k\leq D$,
    \begin{gather}\label{E: composition 1}
        (p_{j,k}-p_{m,k})-(p_{j',k}-p_{m,k}) = p_{j,k}-p_{j',k},\\
        \label{E: composition 2}
        (\sigma_{h_{r+1}}p_{j,k}-p_{m,k})-(\sigma_{h_{r+1}}p_{j',k}-p_{m,k}) = \sigma_{h_{r+1}}(p_{j,k}-p_{j',k}),
    \end{gather}
    both of which have the same degree. Hence in these two cases, we immediately deduce the second nice property for $\partial_m\CP$ by invoking it for $\CP$.
    
    Suppose now that we are in the mixed case of comparing the coordinates of $\sigma_{h_{r+1}}p_{j}-p_{m}$ and $p_{j'}-p_m$ for some $1\leq j,j'\leq \ell$. In what follows, we say that two polynomials $p,p'\in\Z[n,h]$ satisfy $p\sim p'$ if they have the same leading coefficient. Now, if $p_{j,D}\not\sim p_{j',D}$, then
    \begin{align*}
        \deg \brac{(\sigma_{h_{r+1}}p_{j,D}-p_{m,D}) - (p_{j',D}-p_{m,D})} &= \deg (\sigma_{h_{r+1}}p_{j,D}-p_{j',D})\\
        &= \max(\deg p_{j,D}, \deg p_{j',D}).
    \end{align*}
    If one of $j,j'$ is in $\FI_D$, then for any  $1\leq k<D$, the second nice property for $\CP$ gives
    \begin{align*}
        \max(\deg p_{j,D},\deg p_{j',D}) &>\max(\deg p_{j,k}, \deg p_{j',k}) \geq \deg ({\sigma_{h_{r+1}}}p_{j,k}- p_{j',k})\\
        &={\deg\brac{(\sigma_{h_{r+1}}p_{j,k}-p_{m,k}) - (p_{j',k}-p_{m,k})}}.
    \end{align*}
    
     On the other hand, if $p_{j,D}\sim p_{j',D}$, then $j,j'\in\FI_D$ and
    \begin{align}\label{E: niceness problematic0}
        \deg \brac{(\sigma_{h_{r+1}}p_{j,D}-p_{m,D}) - (p_{j',D}-p_{m,D})} &= \deg (\sigma_{h_{r+1}}p_{j,D}-p_{j',D})= \deg p_{j,D}-1\\\nonumber
        &= \deg (p_{j,D}-p_{j',D}),
    \end{align} 
    where the last equality follows from the third nice property.
    Fix $1\leq k<D$. Now, two things may happen. If 
    \begin{align}\label{E: niceness problematic}
        \deg \brac{(\sigma_{h_{r+1}}p_{j,k}-p_{m,k}) - (p_{j',k}-p_{m,k})} \leq \deg (p_{j,k} -p_{j',k}),
    \end{align}
    then we have
     \begin{align}\label{E: niceness problematic2}
         \deg \brac{(\sigma_{h_{r+1}}p_{j,D}-p_{m,D}) - (p_{j',D}-p_{m,D})}>\deg \brac{(\sigma_{h_{r+1}}p_{j,k}-p_{m,k}) - (p_{j',k}-p_{m,k})}
    \end{align}
    by
     combining (\ref{E: niceness problematic0}) and (\ref{E: niceness problematic}) and invoking the second nice property for $\CP$. Otherwise, the failure of \eqref{E: niceness problematic} forces $p_{j,k}\sim p_{j',k}$, in which case
    \begin{align*}
        \deg \brac{(\sigma_{h_{r+1}}p_{j,k}-p_{m,k}) - (p_{j',k}-p_{m,k})}  &= \deg p_{j,k} - 1 < \deg p_{j,D} - 1\\
        &= \deg \brac{(\sigma_{h_{r+1}}p_{j,D}-p_{m,D}) - (p_{j',D}-p_{m,D})}.
    \end{align*}
    Hence the second nice property follows for this case. 

{The last remaining case of the second nice property is when $j' = 0$, i.e. we look at the coordinates of $p_j - p_m$ (for $j\neq m$) as well as $\sigma_{h_{r+1}}p_j-p_m$. If $p_{j,D}\nsim p_{m,D}$, then it follows from the second nice property for $\CP$ that
    \begin{multline*}
        \min\{\deg(p_{j,D}-p_{m,D}),\deg(\partial_{h_{r+1}}p_{j,D}-p_{m,D})\}=\max(\deg p_{j,D}, \deg p_{m,D})
        \\>\max\{\deg p_{j,k},\deg p_{m,k}\}\geq \max\{\deg(p_{j,k}-p_{m,k}),\deg(\partial_{h_{r+1}}p_{j,k}-p_{m,k})\}.
    \end{multline*}

    On the other hand, suppose that $p_{j,D}\sim p_{m,D}$. If $j\neq m$, then $\deg (p_{j,D}-p_{m,D})>\deg(p_{j,k}-p_{m,k})$ by the second nice property for $\CP$.
    The argument is more tricky for $\sigma_{h_{r+1}}p_j-p_m$. Under the assumption $p_{j,D}\sim p_{m,D}$, we have
    \begin{align*}
        \deg (\sigma_{h_{r+1}}p_{j,D} - p_{m,D}) &= \deg p_{j,D}-1 \geq \max(\deg p_{j,k},\deg p_{m,k})\\
        &\geq \deg(\sigma_{h_{r+1}}p_{j,k}-p_{m,k}).
    \end{align*}
    The only possibility in which these inequalities are not strict is if $p_{j,k}\not\sim p_{m,k}$ and $$\deg p_{j,D}-1 = \max(\deg p_{j,k},\deg p_{m,k}).$$
    Then, however,
    \begin{align*}
        \deg (p_{j,D} - p_{m,D}) &\leq \deg p_{j,D}-1 = \deg(p_{j,k}-p_{m,k}),
    \end{align*}
    contradicting the second nice property for $\CP$.} This proves the second nice property for $\partial_{m}\CP$.

A similar case-by-case analysis gives the third nice property. For the pairs $p_{j,D}-p_{m,D}$, $p_{j',D}-p_{m,D}$ as well as $\sigma_{h_{r+1}}p_{j,D}-p_{m,D}$, $\sigma_{h_{r+1}}p_{j',D}-p_{m,D}$ (both under the assumption $j\neq j'$), the third nice property follows from the corresponding property for $\CP$ as well as 
the identities \eqref{E: composition 1} and \eqref{E: composition 2}. 
For the pair $\sigma_{h_{r+1}}p_{j,D}{-p_{m,D}}$, $p_{j,D}{-p_{m,D}}$, it is clear that their difference has degree {$\deg p_{j,D}-1\geq \deg (p_{j,D}-p_{m,D})-1$, where the inequality follows from the assumption that $m$ is a good index}. 
Lastly, for the pair $\sigma_{h_{r+1}}p_{j,D}{-p_{m,D}}$, $p_{j',D}{-p_{m,D}}$ with $j\neq j'$, a direct computation shows that the third nice property can only fail if it fails for $p_{j,D}, p_{j',D}$, contradicting the third nice property for $\CP$.

     We conclude with the proof of the fourth nice property. Let $j\in \FI_D$ and $0\leq j'\leq \ell$ be distinct from $j$. Using the identities
     \begin{align*}
 (p_{j}\ast p_m^{-1})*(p_{j'}\ast p_m^{-1})^{-1}=p_{j}\ast p_{j'}^{-1} 
\end{align*}
and 
\begin{align*}
 (\sigma_{h_{r+1}} p_{j}\ast p_m^{-1})*(\sigma_{h_{r+1}} p_{j'}\ast p_m^{-1})\inv = \sigma_{h_{r+1}} (p_{j}\ast p_{j'}^{-1}),    
\end{align*}
the $D$-distinguishability of both polynomials follows from the fourth nice property for $\CP$. 

Likewise, we observe that 
\begin{align*}
    (\sigma_{h_{r+1}}p_{j}\ast p_m^{-1})*(p_{j'}\ast p_m^{-1})^{-1}=\sigma_{h_{r+1}}p_{j}\ast p_{j'}^{-1}.
\end{align*}
Moreover, it follows from the third nice property for $\CP$ that the $D$-th coordinates of the polynomials $\sigma_{h_{r+1}}p_{j}\ast p_{j'}^{-1}$ and $p_{j}\ast p_{j'}^{-1}$ have the same degree (which equals $\max (\deg p_{j,D}, \deg p_{j',D})$ if $p_{j,D}\not\sim p_{j',D}$ and $\deg p_{j,D}-1$ otherwise). This is important because $D$-distinguishability is defined in terms of the degree of the $D$-th coordinates. Using these two facts, it is then easy to see on setting $h_{r+1}=0$ that if $\sigma_{h_{r+1}}p_{j}\ast p_{j'}^{-1}$ fails to be $D$-distinguishable, then so does $p_{j}\ast p_{j'}^{-1}$,  contradicting the fourth nice property for $\CP$.

All that remains to be seen is that 
\begin{align*}
    (\sigma_{h_{r+1}}p_{j}\ast p_m^{-1})*(p_{j}*p_m\inv)^{-1}=\sigma_{h_{r+1}}p_{j}\ast p_{j}^{-1}
\end{align*}
is $D$-distinguishable for every $j\in \FI_D.$ Recall the definitions of $\Delta$ and $\partial$ from Definition \ref{D: distinguishable} and Lemma~\ref{lnh} respectively.
Observe that $\sigma_{h_{r+1}}p_{j}\ast p_{j}^{-1} = \Delta_{h_{r+1}}p_j,$
and hence
\begin{align}\label{E: derivative identity for fourth nice property}
    \Delta_{m_{d'-1}}\cdots \Delta_{m_{1}}(\sigma_{h_{r+1}}p_{j}\ast p_{j}^{-1}) = \Delta_{m_{d'-1}}\cdots \Delta_{m_{1}}\Delta_{h_{r+1}}p_j
\end{align}
 for any $d'\in\N$. By the assumption $j\in \FI_D$, the polynomial $p_{j,D}$ has some positive degree $d'$,\footnote{In fact, all that we need for the argument of this paragraph to work is that $p_{j,D}$ is a nonzero polynomial.} and hence $\partial_{h_{r+1}}p_{j,D}$ has degree $d'-1$. It follows from this and the identity \eqref{E: derivative identity for fourth nice property} that $\sigma_{h_{r+1}}p_{j}\ast p_{j}^{-1}$ is $D$-distinguishable if and only if $p_j$ is. The latter holds by the fourth nice property for $\CP$. This case completes the proof of the fourth nice property for $\partial_m \CP$.

{We have thus established that $\partial_m \CP$ is nice; it remains to be seen that $\partial_m (\CP, p_{j_0})$ is nice, i.e. that $\sigma_{h_{r+1}}p_{j_0}*p_m\inv$ is nonconstant and has the highest-degree $D$-th coordinate among the members of $\partial_{m}\CP$. In the case when $\omega_\nab = 0$, $k=D$ and all of $p_{1,D},\dots,p_{\ell,D}$ have the same degree and leading coefficients, we have $\deg((\sigma_{h}p_{j_{0}}\ast p_{m}^{-1})_{D})=\deg p_{j_{0},D}-1$, and the degrees of the $D$-th coordinates of the other polynomials in  $\partial_{m}\CP$ are also at most $\deg p_{j_{0},D}-1$. In all other cases, $\deg((\sigma_{h}p_{j_{0}}\ast p_{m}^{-1})_{D})=\deg p_{j_{0},D}$. So in all cases, $\sigma_{h}p_{j_{0}}\ast p_{m}^{-1}$ has the highest-degree $D$-th coordinate among all the polynomials in $\partial_{m}\CP$.

Finally, since $\CP$ is not basic, the family $\partial_{m}\CP$ is {nondegenerate, i.e. it contains at least one polynomial with a nonconstant abelian part. By the conclusion of the previous paragraph, $\sigma_{h}p_{j_{0}}\ast p_{m}^{-1}$ is such a polynomial}, and thus it is not removed using the $^\sharp$ operation. }
\end{proof}

The next result summarizes the arguments in this section and shows that after finitely many van der Corput operations, every nice family turns into one amenable to Theorem~\ref{T: new estimates}. Recall that $r$ is the dimension of the group over which the indeterminate $h$ runs. 

\begin{proposition}\label{P: types stabilize}
       Let $\CP= \{p_1, \ldots, p_\ell\}\subseteq \Z^{D'}[n, h]$, and suppose that $(\CP,p_{\ell})$ is a nice pinned polynomial family of degree at most $d$.
       Then there exists a nonnegative integer $s = O_{d,\ell,r}(1)$ and a sequence of indices $m_1, \ldots, m_s\in\N$ such that the families
       \begin{align*}
           (\CP_i,q_{i}) := \begin{cases}(\CP,p_{\ell}),\; &i=0\\
           \sigma_{m_i}\cdots \sigma_{m_1}(\CP,p_{\ell}),\; &i>0
           \end{cases}
       \end{align*}
       are nice and have decreasing types while $(\CP_s,q_s)$ is additionally basic. 

       Moreover, for any 1-bounded functions $f_1, \ldots, f_\ell\in L^\infty(\mu)$
       \begin{align}\label{E: vdc in PET}
           \brac{\oE_{h\in\Z^r}\norm{\E_{n\in\Z}U_{p_1(n,h)}f_1\cdots U_{p_{\ell}(n,h)}f_\ell}_{L^2(\mu)}}^{2^s}\leq \oE_{h\in\Z^{r+s}}\norm{\E_{n\in\Z}\prod_{p\in \CP_s}U_{p(n,h)}g_p}_{L^2(\mu)},
       \end{align}
       where $g_{q_{s}}=f_{\ell}$.
\end{proposition}

In this and other results in this section, the quantitative dependence on $r$ can be negotiated. Rather, what $s$ depends on is $d, \ell$, and the dimension of the subgroup generated by $p_1, \ldots, p_\ell$. The latter quantity affects the number of commutators that appear in the intermediate steps of PET, and it can be bounded either as a function of $d,\ell, r$ or $d,D,\ell$. Since in our main application (Theorem~\ref{T: nilpotent seminorm estimates}) we have $r = 0$, we opted to go for the former. 
\begin{proof}
We first reduce $(\CP,p_\ell)$ to a nice and basic family.
    If $(\CP,p_\ell)$ is already nice and basic, there is nothing to do. Otherwise we inductively construct the families $(\CP_1, q_1), (\CP_2, q_2), \ldots$ by finding indices $m_1, m_2, \ldots$ such that $m_{i+1}$ is good for the van der Corput operation for $(\CP_i,q_i)$. We terminate the algorithm once we reach a nice and basic family $(\CP_s,q_s)$. By Lemma~\ref{L: type reduction}, each subsequent family has lower type than the previous one, and by Proposition~\ref{P: niceness preserved}, they are all nice. 
    
     The point is to prove that this process terminates at some $s=O_{d,\ell,r}(1)$ steps. First, for each fixed type $(\omega_\ab,\omega_\nab)$ arising on the way, only $O_{d, \ell,r}(1)$ steps are required to reach a family with type $(\omega_\ab,0)$. 
    Second, by invoking the nondegeneracy of $\CP$ and mimicking the proof of \cite[Lemma~4.5]{CFH11}, we deduce that only $O_{d, \ell,r}(1)$ van der Corput operations as above are required to bring the abelian type $\omega_\ab$ of $\CP$ to $\omega'_\ab$ with $\omega'_1 = \cdots = \omega'_{D-1} = 0$ and $\omega'_D\neq 0$. Third, further $O_{d, \ell,r}(1)$ van der Corput operations as above bring $\omega'_\ab$ to $\omega_\ab''$ in which $\omega''_1 = \cdots = \omega''_{D-1} = 0$ while $\omega''_{D}=(1,0,\ldots, 0)$. That we do not ``skip over'' such a family while lowering the abelian type can be justified as follows. Suppose first that $\CP_i$ contains a polynomial whose $D$-th coordinate has degree at least 2. Since
    \begin{align*}
     \deg (\sigma_{h_{r+1}} p_D-q_D)\geq \max(\deg p_D,\deg q_D)-1   
    \end{align*}
    for any $p,q\in\Z^{D'}[n,h]$, the family $\CP_{i+1}$ has a polynomial whose $D$-th coordinate has degree at least 1. If instead all elements of $\CP_i$ have $D$-th coordinates of degree 0 or 1 (at least one of which is 0), and the nonabelian type of $\CP_i$ is 0, then the successive van der Corput operation lowers the number of distinct degree-1 $D$-th coordinates by exactly one.

    It remains to show \eqref{E: vdc in PET}. Let $\CP_i= \{p_{l,1}, \ldots, p_{i,\ell_i}\},$ and let $\FI_{i,k}, \FI_{i,\nonab}$ be the sets from Definition \ref{D: degree classes} for the family $\CP_i.$ By applying the van der Corput inequality every time, for every $0\leq i\leq s$ we get
    \begin{multline*}
    \brac{\oE_{h\in\Z^{r}}\norm{\E_{n\in\Z}U_{p_1(n,h)}f_1\cdots U_{p_\ell(n,h)}f_\ell}_{L^2(\mu)}}^{2^i}\\
    \leq \oE_{h\in \Z^{r+i}}\norm{\E_{n\in\Z}U_{p_{i,1}(n,h)}g_{i,1}\cdots U_{p_{i,\ell_i}(n,h)}g_{i,\ell_i}}_{L^2(\mu)}
    \end{multline*}
    for some 1-bounded functions $g_{i,j}\in L^\infty(\mu)$. 
    Moreover, from the definition of the van der Corput operation and the fact that the pinned polynomial $q_{i}$ is never removed by the $^\sharp$ operation, it is not hard to see that $g_{i,q_{i}}\in\{f_{\ell},\bar f_{\ell}\}$ for all $i$. By conjugating the expression inside the $L^2(\mu)$ norm, we can assume that $g_{q_i} = f_\ell$.
\end{proof}

\subsection{Proof of Theorem~\ref{T: nilpotent seminorm estimates for nice families}}
We finally prove Theorem~\ref{T: nilpotent seminorm estimates for nice families}. 
   Fix 1-bounded $f_1, \ldots, f_\ell\in L^\infty(\mu)$. By Proposition~\ref{P: types stabilize}, there exists a nonnegative integer ${s_1}=O_{d,\ell,r}(1)$, natural numbers $m_1, \ldots, m_{s_1}$, a nice basic {pinned family $(\CQ, q_\kappa)$}, where $$\CQ:=\partial_{m_{s_1}}\cdots\partial_{m_1}\CP = \{q_1, \ldots, q_{\kappa}\},$$ and functions $g_1, \ldots, g_\kappa\in L^\infty(\mu)$ with {$g_\kappa = f_\ell$}, such that

   \begin{multline}\label{E: vdC}
 \brac{\oE_{h\in\Z^{r}}\norm{\E_{n\in\Z}U_{p_1(n,h)}f_1\cdots U_{p_\ell(n,h)}f_\ell}_{L^2(\mu)}}^{2^{s_1}}\\
    \leq \oE_{h\in \Z^{r+{s_1}}}\norm{\E_{n\in\Z}U_{q_1(n,h)}g_{1}\cdots U_{q_\kappa(n,h)}g_\kappa}_{L^2(\mu)}.
   \end{multline}    
    Since $\CQ$ is basic, we may write its terms as 
    \begin{align*}
        q_j(n,h) = (a_j(h), t(h)n+b_j(h), c_j(n,h))
    \end{align*}
    for some $a_j\in\Z^{D-1}[h]$, $t,b_j\in\Z[h]$ (where $t$ is nonzero), and $c_j\in\Z^{D'-D}[n,h]$. 
The second nice property ensures that $b_{j}\neq b_{j'}$ for all $1\leq j<j'\leq \kappa$ and that all $a_{j}$'s equal the same $a$.

Write
$$q'_{j}(n,h):=q_{j}(n,h)\ast (a_j(h),0,0)^{-1}=(0,t(h)n+b_{j}(h),c'_{j}(n,h))$$
for some $c_j'\in \Z^{D'-D}[n,h]$.
It is not hard to see that 
$$\CQ':=\{q'_{j}\colon\; 1\leq j\leq \kappa\}$$
is also a nice and basic family.
  Translating the right hand side of \eqref{E: vdC} by $(a(h),0,0)\inv$, we get
        \begin{multline}\label{E: vdC 2}
     \brac{\oE_{h\in\Z^{r}}\norm{\E_{n\in\Z}U_{p_1(n,h)}f_1\cdots U_{p_\ell(n,h)}f_\ell}_{L^2(\mu)}}^{2^{s_1}}\\
    \leq \oE_{h\in \Z^{r+{s_1}}}\norm{\E_{n\in\Z}U_{q'_1(n,h)}g_{1}\cdots U_{q'_\kappa(n,h)}g_\kappa}_{L^2(\mu)}.
    \end{multline}
 
 It is somewhat inconvenient that the average over $n$ is inside the norm while the average over $h$ is outside. To amend this, we apply the symmetric van der Corput inequality.
 It gives that
    the square of the right-hand side of \eqref{E: vdC 2}   is bounded by
    \begin{align}\label{E: PET iterated limit}
        \oE_{(h,u_0,u_1)\in\Z^{r+{s_1}+2}}\E_{n\in\Z}\int \prod_{j=1}^{\kappa}\brac{U_{q'_{j}(n+u_{0},h)}g_{j}\cdot U_{q'_{j}(n+u_{1},h)}\overline{g_{j}}}\; d\mu.
    \end{align}
    By Walsh's convergence theorem \cite{W12} as well as the Fubini-type principle for averages \cite[Lemma~1.1]{BL15}, the iterated limit \eqref{E: PET iterated limit} equals the single limit
    \begin{align}\label{E: PET iterated limit 2}
      \E_{(h,u_{0},u_{1},n)\in\Z^{r+{s_1}+3}}\int \prod_{j=1}^{\kappa}\brac{U_{q'_{j}(n+u_{0},h)}g_{j}\cdot U_{q'_{j}(n+u_{1},h)}\overline{g_j}}\; d\mu.
    \end{align}

For distinct pairs $(j,\eps), (j',\eps')\in\{1,\dots,\kappa\}\times \{0,1\}$, define
$$q_{j,\eps,j',\eps'}(n,h,u_{0},u_{1}):=q'_{j}(n+u_{\eps},h)\ast q'_{j'}(n+u_{\eps'},h)^{-1}$$
and let $G(q_{j,\eps,j',\eps'})$ be as in \eqref{E: tilde G}. By Theorem~\ref{T: new estimates} and the fact that {$g_\kappa = f_\ell$},
there exists a natural number $s_2=O_{d,\ell,r}(1)$ such that \eqref{E: PET iterated limit 2} vanishes whenever
\begin{align*}
    \nnorm{f_\ell}_{G(q_{j,\eps,j',\eps'})^{\times s_2}\colon\; \textrm{distinct}\; (j,\eps), (j',\eps')\in\{1,\dots,\kappa\}\times \{0,1\}} = 0.
\end{align*}
The claimed result will then follow from the subgroup property of seminorms once we show that each group $G(q_{j,\eps,j',\eps'})$ contains {some nonzero multiple of} $e_D$. A key tool in that will be the fourth nice property.

Observe that all the $q'_{j}$'s are supported on the abelian subgroup $\{0\}^{D-1}\times \Z^{D'-(D-1)}$.
Whenever $j\neq j'$, the group ${G}(q_{j,\eps,j',\eps'})$ contains ${G}(q'_{j}\ast {q'_{j'}}^{-1})$ (on setting $u_0=u_1 = 0$), which contains $e_{D}$ by the fourth nice property and Lemma~\ref{lnh}. If instead  $j=j'$ and $\eps\neq \eps'$,  then ${G}(q_{j,\eps,j',\eps'})$ contains ${G}(q'_{j}\ast q'_{j}(0,\cdot)^{-1})$ on taking $n=u_1 = 0$.
Note that 
$$q'_{j}(n,h)\ast q'_{j}(0,h)^{-1}=(0,t(h)n,c_{j}(n,h)-c_{j}(0,h)).$$
For this polynomial and $q'_{j}$, their $D$-th coordinates are linear in $n$ with the same linear coefficient $t(h)$. Hence the $D$-distinguishability of $q'_j$ implies that $t$ is not a $\Q$-linear combination of the coordinates of $\Delta_{m}c_{j}$, 
which are the same as the coordinates of $\Delta_{m}(c_{j}-c_{j}(0,\cdot))$.
Once again, ${G}(q'_{j}\ast q'_{j}(0,\cdot)^{-1})$ (and hence also ${G}(q_{j,\eps,j',\eps'})$) contains {some nonzero multiple of} $e_{D}$ by Lemma~\ref{lnh}. \hfill $\Box$

\subsection{Proof of Theorem~\ref{T: nilpotent seminorm estimates}}
Let $p_1, \ldots, p_\ell\in\Z[n]$ be nonconstant and have increasing degrees, and let $\CP = \{p_1 e_1, \ldots, p_\ell e_\ell\}$.  
As shown in Lemma~\ref{L: niceness}, the family $(\CP, p_\ell e_\ell)$ is nice. Hence, by Theorem~\ref{T: nilpotent seminorm estimates for nice families}, there exists $s_1=O_{d,\ell}(1)$ such that 
  \begin{align}\label{E: original average in PET}
    \norm{\E_{n\in\Z}T_1^{p_1(n)}f_1\cdots T_\ell^{p_\ell(n)}f_\ell}_{L^2(\mu)} = 0
\end{align}
whenever $\nnorm{f_\ell}_{s_1+1,T_\ell} = 0$. We want to derive a similar conclusion in terms of $f_1, \ldots, f_{\ell-1}$. Using the newly established control, we can replace $f_\ell$ by $\E(f_\ell|\CZ_{s_1}(T_\ell))$, and then approximate by it a linear combination of dual functions $\CD_{s_1,T_\ell}\mathfrak{f}$ with $\mathfrak{f}\subseteq (L^\infty(\mu))^{\{0,1\}^{s_1}_*}$. In particular, if \eqref{E: original average in PET} fails, then it also fails with $f_\ell$ replaced by $\CD_{s_1,T_\ell}\mathfrak{f}$ for some $\mathfrak{f}\subseteq (L^\infty(\mu))^{\{0,1\}^{s_1}_*}$. By \cite[Proposition~6.1]{Fr12}, there exists a natural number $r=O_{d,\ell}(1)$ such that
\begin{align*}
    \oE_{h\in\Z^r}\norm{\E_{n\in\Z}\prod_{\eps\in\{0,1\}^r}T_1^{p_1(n+\eps\cdot h)}f_1\cdots T_{\ell-1}^{p_{\ell-1}(n+\eps\cdot h)}f_{\ell-1}}_{L^2(\mu)}>0.
\end{align*}
Let 
\begin{align*}
    \CQ := \{p_{j,\eps}\colon\; 1\leq j < \ell,\; \eps\in\{0,1\}^r\}
\end{align*}
be the new family, where $p_{j,\eps}(n,h):=p_j(n+\eps\cdot h)$.
Again using Lemma~\ref{L: niceness}, this time applied to the reduced system $(X, \CX, \mu, T_1, \ldots, T_{\ell-1})$, we deduce that $(\CQ, p_{\ell, \eps})$  is nice for every $\eps\in\{0,1\}^r$. Hence by Theorem~\ref{T: nilpotent seminorm estimates for nice families} again, we can find $s_2 = O_{d,\ell}(1)$ such that
$\nnorm{f_{\ell-1}}_{s_2,T_\ell} > 0$. This implies the desired estimate in terms of $f_{\ell-1}$ by contrapositive. By iterating this argument, replacing the functions $f_{\ell-2}, \ldots, f_2$ by dual functions one-by-one, we get the claimed result.  \hfill $\Box$

\section{Relationships between Host-Kra factors of nilpotent actions}\label{S: comparison of factors}
This section investigates the relationship between the factors $\CZ_{s,H}$ and $\CZ_{s',H'}$ whenever $H'\leq H$ are finitely generated and nilpotent. In the commuting case, this relationship is well-understood; see Section \ref{SS: commuting}. In the nilpotent case, it seems far more tricky to compare these two factors to each other. The main result of this section is summarized below.
\begin{proposition}\label{P: comparing HK factors in nilpotent systems}
Let $(X, \CX, \mu, T_1, \ldots, T_\ell)$ be a 2-step nilpotent system with $T_1, \ldots, T_{\ell}$ ergodic. Then for all $s\in\N$ there exists an natural number $s'=O_{\ell,s}(1)$ such that
  $\CZ_{s,T_j}\subseteq\CZ_{s',H}$ for all $1\leq j\leq \ell$, where $H := \langle T_1, \ldots, T_\ell\rangle$.
\end{proposition}
In the commuting case, one can simply take $s'=s$ by combining the properties of box seminorms listed in Section \ref{SS: commuting}. It is unclear what is the optimal value of $s'$ is in the 2-step nilpotent setting, nor whether the result extends to higher-step nilpotent systems (see Problem~\ref{Pr: comparing factors}). 

\subsection{Nilpotent actions on nilsystems}
Our first results asserts that if we have a nilpotent action on a nilsystem, and some generators act by nilrotations, then the entire action is by nilrotation, possibly at the cost of recasting the underlying nilmanifold as one of higher step. We are grateful to Bryna Kra for help with its proof. All the nilsystem-related definitions can be found in Appendix \ref{A: nilsystems}.
\begin{proposition}\label{P: Bryna}
Let $Y$ be an $s$-step nilmanifold. Suppose that $(Y, \CB_Y, m_Y, S_1, \ldots, S_{\ell+1})$ is a system with the following properties:
\begin{enumerate}
    \item each $S_1, \ldots, S_\ell$ acts ergodically on $Y$ via nilrotations;
    \item the group $\langle S_1, \ldots, S_{\ell+1}\rangle$ is $k$-step nilpotent.
\end{enumerate}
Then $(Y, \CB_Y, m_Y, S_1, \ldots, S_{\ell+1})$ is isomorphic to a $ks^2$-step nilsystem on which (the image of) $\langle S_1, \ldots, S_{\ell+1}\rangle$ acts via nilrotations.
\end{proposition}
Unless stated otherwise, an isomorphism between two measure-preserving systems always refers to a measure-preserving isomorphism.

This result could likely be extended to the case when the group $\langle S_1, \ldots, S_{\ell}\rangle$ acts ergodically on $Y$ via nilrotations (which is a weaker assumption than the ergodicity of \textit{each} $S_j$), but that would require extending a result of Parry used in the proof below from ergodic $\Z$-nilsystems to ergodic $H$-nilsystems for a general finitely generated group $H$.

\begin{proof}
Let $Y=G/\Gamma$, and for $1\leq j\leq \ell,$ define $\hat S_j:=S_{\ell+1} S_j S_{\ell+1}\inv$.
Observe that the systems $(Y, \CB_Y, m_Y, S_1, \ldots, S_\ell)$ and $(Y, \CB_Y, m_Y, \hat S_1, \ldots, \hat S_\ell)$ are isomorphic via the factor map $y\mapsto S_\ell y$. By the ergodicity of $S_j$'s and a result of Parry \cite{Parry71} (see also \cite[Theorem~1.4]{Parry73}), the map $S_\ell$ agrees $m_Y$-a.e. with an affine map. This means that after modifying $S_{\ell+1}$ on a null set, there exist $b\in G$ and $\sigma\in\Aut(G)$ with $\sigma(\Gamma) = \Gamma$ such that  
$$S_{\ell+1}(g\Gamma) = b\sigma(g)\Gamma$$ 
for all $g\in G$. Our goal is to show that 
the automorphism $\sigma$ is \textit{$ks$-step unipotent}, meaning that the map $\sigma'(g):=\sigma(g)g\inv$ satisfies $(\sigma')^{(ks)}(G) \subseteq \Gamma$. 
We will then show like in
\cite[Chapter 11, Proposition~19]{HK18} that 
the affine system $(Y, \CB_Y, m_Y, S_1, \ldots, S_{\ell+1})$ is a $ks^2$-step nilsystem. 

Let $G = G_1 \supseteq G_2 \supseteq \cdots$ be the lower central series of $G$. For $1\leq i\leq s$, let $Y_i := G_i/G_{i+1}\Gamma_i$, where $\Gamma_i:=G_i\cap \Gamma$.
Notice that $\sigma([g_1,g_2])=[\sigma(g_1),\sigma(g_2)]$ for any $g_1,g_2\in G$, and so $\sigma(G_i) = G_i$ for every $1\leq i\leq s$. It follows that the induced map $\sigma_i:Y_i \to Y_i$ given by 
$$\sigma_i(gG_{i+1}\Gamma_i) := \sigma(g)G_{i+1}\Gamma_i$$ 
is a well-defined automorphism of $Y_i$. We will show by induction on $i$ that $\sigma_i$ is \textit{$k$-step unipotent}, i.e. the related map\footnote{Notice that while $\sigma'$ is not a group homomorphism in general, $\sigma'_i$ is because $Y_i$ is an abelian group.} $\sigma'_i:Y_i \to Y_i$ defined by $$\sigma_i'(gG_{i+1}\Gamma_i):= \sigma_i(gG_{i+1}\Gamma_i)(gG_{i+1}\Gamma_i)\inv =\sigma(g)g\inv G_{i+1}\Gamma_i$$ satisfies $(\sigma_i')^{(k)}(gG_{i+1}\Gamma_i) = G_{i+1}\Gamma_i$ for every $g\in G_i$. Note that the $k$-step unipotence of $\sigma_i$ implies that $$(\sigma')^{(k)}(G_i)\subseteq G_{i+1}\Gamma_i.$$ Hence if $\sigma_1,\ldots, \sigma_s$ are all $k$-step unipotent, then $\sigma$ is $ks$-step unipotent as claimed because\footnote{We use here the fact that $\sigma'(gh) = \sigma'(g)\sigma'(h)[\sigma'(h),g\inv]$, and hence $\sigma'(G_i\Gamma)\subseteq \sigma'(G_i) G_{i+1}\Gamma$. Iterating, we get 
    $\sigma'^{(k)}(G_i\Gamma)\subseteq G_{i+1}\Gamma$ for any $i$.}
\begin{align*}
    {\sigma'}^{(ks)}(G)\subseteq {\sigma'}^{((k-1)s)}(G_2\Gamma)\subseteq \cdots \subseteq\sigma'^{(k)}(G_s\Gamma)\subseteq \Gamma.
\end{align*}

To prove the $k$-step unipotence of $\sigma_i$'s, we start with $i=1$. Take $S_j$ for any $1\leq j\leq \ell$. Since $Y_1$ is a compact abelian group, $S_j$ induces the rotation 
${S'_j} z := \alpha_j z$ on $Y_1$. 
As $S_j$ is assumed to be ergodic, the orbit of $\alpha_j$ is dense in ${Y_1}$. 
The image of $S_{\ell+1}$ on ${Y_1}$ takes the form
$${S'_{\ell+1}}z := b\sigma_1(z).$$

For $v\in G/G_2$, if $L_v(z) := vz$ is the rotation by $v$, then  
\begin{align*}
    {S'_{\ell+1}} L_v {S'_{\ell+1}}^{-1} = L_{\sigma_1(v)},
\end{align*}
and so 
\begin{align*}
    [{S'_{\ell+1}}\inv, L_v\inv] = {S'_{\ell+1}}L_v{S'_{\ell+1}}^{-1}L_v\inv 
    = L_{\sigma_1'(v)}.
\end{align*}
Iterating this $n$ times for $L_v = S_j'$ (for any $1\leq j\leq\ell$), we obtain 
\begin{align*}
  [{S'_{\ell+1}}\inv, [{S'_{\ell+1}}\inv, \dots, [{S'_{\ell+1}}\inv, {{S'_j}}\inv]\dots]] = L_{(\sigma_1')^{(n)}(\alpha_j)}.
\end{align*}
Since $\langle {S'_j}, {S'_{\ell+1}}\rangle$ is $k$-step nilpotent, the left-hand side equals $\Id_{Y_1}$ for any $n\geq k$.  This means that $(\sigma_1')^{(k)}\alpha_j = e_{Y_1}$.
Because $\alpha_j$ generates a dense subgroup of ${Y_1}$ and $(\sigma_1')^{(k)}$ 
is a continuous group homomorphism, we have that 
$(\sigma_1')^{(k)} = \id_{Y_1}$.
This means that $\sigma_1$ is $k$-step unipotent on ${Y_1}$. 

Suppose now that $\sigma_i$ is $k$-step unipotent on $Y_i$; we want to show that $\sigma_{i+1}$ is $k$-step unipotent on $Y_{i+1}$. Recall that the group $G_{i+1} = [G_i, G]$ is generated by elements of the form $[g_i,g]$ for some $g_i\in G_i, g\in G$ (in fact, every element of $G_{i+1}$ takes such a form). Then
\begin{align*}
    \sigma_{i+1}'([g_i, g]G_{i+2}\Gamma_{i+1}) &= \sigma([g_i,g])[g_i,g]\inv G_{i+2}\Gamma_{i+1} = [\sigma(g_i),\sigma(g)][g_i\inv, g\inv]G_{i+2}\Gamma_{i+1}\\
    &= [\sigma(g_i)g_i\inv,\sigma(g)g\inv]G_{i+2}\Gamma_{i+1} = [\sigma'(g_i),\sigma'(g)]G_{i+2}\Gamma_{i+1}.
\end{align*}
Iterating, we get that for any $n$,
\begin{align*}
    (\sigma_{i+1}')^{(n)}([g_i, g]G_{i+2}\Gamma_{i+1}) = [(\sigma')^{(n)}(g_i),(\sigma')^{(n)}(g)]G_{i+2}\Gamma_{i+1}.
\end{align*}
In particular, it follows from the induction hypothesis that for any $n\geq k$, we have $(\sigma')^{(n)}(g_i)\in G_{i+1}\Gamma_{i}$ and $(\sigma')^{(n)}(g)\in G_2\Gamma$, implying that
$$[(\sigma')^{(n)}(g_i),(\sigma')^{(n)}(g)]\in G_{i+2}\Gamma_{i+1}.$$
The $k$-step unipotence of $\sigma_{i+1}$ follows. 

We have thus shown that $\sigma_1, \ldots, \sigma_s$ are $k$-step unipotent, and hence $\sigma$ is $ks$-step unipotent. 
Now we have to build the nilmanifold on which $S_1, \ldots, S_{\ell+1}$ act by nilrotations.

By \cite[Chapter 2, Lemma~6]{HK18}, the semidirect product\footnote{I.e. the group operation on $\tilde{G}$ is given by $(g,m)*(g',m') = (g\sigma^m(g'),m+m').$} $\tilde{G} = G \rtimes_\sigma \Z$ is a $ks^2$-step nilpotent Lie group (that is where we use the $ks$-step unipotence of $\sigma$) and $\tilde{\Gamma} = \Gamma \rtimes_\sigma \Z$ is discrete and cocompact in $\tilde{G}$.  
Define $\Phi\colon G/\Gamma \to \tilde{G}/\tilde{\Gamma}$ by setting $\Phi(g\Gamma) = (g,0)\tilde{\Gamma}$.  
It is immediate to check that this is a homeomorphism. 

Set $\tilde{a}_j :=(a_j,0)$, $\tilde b := (b,1)\in\tilde{G}$, where $a_j\in G$ is the element defining the nilrotation $S_j$ for $1\leq j\leq \ell$. Then 
\begin{align*}
    \Phi(S_j(g\Gamma)) = \begin{cases} (a_j g,0)\tilde{\Gamma} = \tilde L_{\tilde a_j}\Phi(g\Gamma),\; &1\leq j\leq\ell\\
    (b\sigma(g), 0)\tilde{\Gamma} = \tilde L_{\tilde{b}}\Phi(g\Gamma),\; &j=\ell+1,
    \end{cases} 
\end{align*}
where $\tilde L_{\tilde a_j}$,  $\tilde L_{\tilde b}$ are corresponding left multiplication maps on $\tilde G/\tilde\Gamma$. Set $\tilde S_j:= \tilde L_{\tilde a_j}$ for $1\leq j\leq \ell$ and $\tilde S_{\ell+1}:=\tilde L_{\tilde b}$.
It follows that $(G/\Gamma, S_1, \ldots, S_{\ell+1})$ is (topologically) conjugate to the $ks^2$-step nilsystem $(\tilde{G}/\tilde{\Gamma}, \tilde S_1, \ldots, \tilde S_{\ell+1})$. The measure $\Phi_*m_{G/\Gamma}$ is $\tilde{G}$-invariant, and so it is the Haar measure $m_{\tilde{G}/\tilde{\Gamma}}$. This gives the claimed isomorphism of the (measurable) nilsystems. 
\end{proof}

\begin{corollary}\label{C: Bryna}
Let $(X, \CX, \mu, T_1, \ldots, T_{\ell+1})$ be a $k$-step nilpotent system. Suppose that $(X, \CX, \mu, T_1, \ldots, T_{\ell})$ is isomorphic to an $s$-step nilsystem. Then $(X, \CX, \mu, T_1, \ldots, T_{\ell+1})$ is isomorphic to a $ks^2$-step nilsystem.
\end{corollary}
\begin{proof}
    Let 
\begin{align*}
    \pi: (X, \CX, \mu, T_1, \ldots, T_\ell)\to (Y, \CB_Y, m_{Y}, S_1, \ldots, S_\ell)
\end{align*}
be the assumed isomorphism onto an $s$-step nilsystem. Define $S_{\ell+1}:=\pi\circ T_{\ell+1}\circ \pi\inv$ on $(Y, \CB_Y, m_{Y})$. Then
\begin{align*}
    \pi: (X, \CX, \mu, T_1, \ldots, T_{\ell+1})\to (Y, \CB_Y, m_{Y}, S_1, \ldots, S_{\ell+1})
\end{align*}
is also an isomorphism. By Proposition~\ref{P: Bryna}, the latter is isomorphic to a $ks^2$-step nilsystem.
\end{proof}

\subsection{Invariance of Host-Kra factors in 2-step nilpotent systems}
In this section, we prove the following result on the invariance of Host-Kra factors in the 2-step nilpotent setting. Our proof uses the 2-step nilpotence, and it is unclear whether the result holds in the higher-step nilpotent universe.
\begin{proposition}\label{P: Z_s(T_1) is T_2-invariant}
    Let $(X, \CX, \mu, T_1, \ldots, T_\ell)$ be 2-step nilpotent. Let $H:=\langle T_1, \ldots, T_{\ell-1}\rangle$. If $T_1, \ldots, T_{\ell-1}$ are ergodic, then $\CZ_s(H)$ is $T_\ell$-invariant for  every $s\in\N_0$.
\end{proposition}

We start with the following simple lemma that does not require nilpotence.
\begin{lemma}\label{L: conjugating factors}
    Let $H$ be a countable amenable group, $U:=(U_h)_{h\in H}$ be an $H$-action, and $T$ be a $\Z$-action on a standard probability space $(X, \CX, \mu)$.
    Let $V := (T\inv U_h T)_{h\in H}$. Then 
    {$$T Z_s(U) = Z_s(V)$$}
    for every $s\in\N_0$.
\end{lemma}
\begin{proof}
    Recall that $Z_s(U)$ is generated by dual functions $\CD_{s,U}\mathfrak{f}$. Given $h,h'\in H^s$, write $h_\eps:=h_{1,\eps_1}+\cdots + h_{s,\eps_s}$, where $h_{i,0} =h_i$ and $h_{i,1}=h_i'$ as before. Using the relation $U_h T = T(T\inv U_hT) = T V_h$, we deduce that
    \begin{align*}
        T \CD_{s, U}\mathfrak{f}:= \E_{h,h'\in H^s}\prod_{\eps\in\{0,1\}^s_*}T (U_{h_\eps}f_\eps)= \E_{h,h'\in H^s}\prod_{\eps\in\{0,1\}^s_*}V_{h_\eps}(Tf_\eps) = \CD_{s, V}(T\mathfrak{f}).
    \end{align*}
    Hence $T Z_s(U) = Z_s(V)$, from which the result follows.
\end{proof}

Below is the second ingredient for the proof of Proposition~\ref{P: Z_s(T_1) is T_2-invariant}. It relies on a result from Appendix \ref{A: Kronecker} that uses 2-step nilpotence. It would be interesting to see if the result holds or fails for higher-step nilpotent systems.

\begin{proposition}\label{P: invariance of Host-Kra factors}
    Let $H$ be a countable amenable group, $U:=(U_h)_{h\in H}$ be an $H$-action, $T$ be a $\Z$-action, and $V := (T\inv U_h T)_{h\in H}$ be an $H$-action on a standard probability space $(X, \CX, \mu)$. Suppose that the joint action of $U,T$ is 2-step nilpotent, and that {$I(U)=I(V)$}. Then
    \begin{align}\label{E: degree-s property}
      \mu_{s,U} = \mu_{s,V},\quad \CI_{\mu_{s,U}}(U^\cube{s}) = \CI_{\mu_{s,V}}(V^\cube{s}),\quad \textrm{and}\quad {Z_{s}(U) = Z_{s}(V)}
    \end{align}
    for every $s\in\N_0$.
\end{proposition}
\begin{proof}
    For $s=0$, the first claim is trivial while the other two are equivalent to the assumption $\CI(U)=\CI(V)$.
    
    Now, suppose that \eqref{E: degree-s property} holds for some $s\in\N_0$, and let $\mu_s  := \mu_{s,U} = \mu_{s,V}$. 
    The first claim follows since
    \begin{align*}
        \mu_{s+1,U} = \mu_s\times_{\CI_{\mu_{s}}(U^\cube{s})}\mu_s = \mu_s\times_{\CI_{\mu_{s}}(V^\cube{s})}\mu_s = \mu_{s+1,V}
    \end{align*}
    by induction.
    The second claim follows from Corollary \ref{C: equal invariant factors 3} with $U^\cube{s}, T^\cube{s}, V^\cube{s}, \mu_s$ in place of $U, T, V, \mu$. {For the last claim, we first infer from the first claim and the definition of Host-Kra seminorms that $\nnorm{\cdot}_{s+1, U}=\nnorm{\cdot}_{s+1, V}$, and then we derive the identity between the vector spaces from Proposition~\ref{P: seminorms vs. algebras}.}
\end{proof}

With these lemmas, we are in the position to prove Proposition~\ref{P: Z_s(T_1) is T_2-invariant}.

\begin{proof}[Proof of Proposition~\ref{P: Z_s(T_1) is T_2-invariant}]
    If $T_j$ is ergodic, then so is $T_\ell\inv T_j T_\ell$. Hence it is easy to see that the actions of $H$ and $$H' := \langle T_\ell\inv T_j T_\ell\colon\; 1\leq j\leq \ell-1\rangle$$ 
    has trivial invariant factors. 
    {By Proposition~\ref{P: invariance of Host-Kra factors} and Lemma~\ref{L: conjugating factors}, $$Z_s(H) = Z_s(H')=T_\ell Z_s(H)$$ for every $s\in\N$. It follows that $Z_s(H)$ is $T_\ell$-invariant, and hence the same is true for the corresponding factor $\CZ_s(H)$.}
\end{proof}

\subsection{Proof of Proposition~\ref{P: comparing HK factors in nilpotent systems}}
For $1\leq m\leq \ell$, let $H_m := \langle T_1, \ldots, T_m\rangle$. We prove by induction on $m$ that there exists $s_m\in\N$ such that $\CZ_{s,T_j}(\CX)\subseteq \CZ_{s_m, H_m}(\CX)$ for all $1\leq j\leq m$. (We emphasize that these Host-Kra factors are defined with respect to $\CX$ since in the proof, we will see Host-Kra factors defined with respect to factors of $\CX$.) The case $m=1$ is trivial with $s_m = s$. 

Suppose that $m>1$. By Theorem~\ref{T: Candela-Szegedy}, the structure theorem for nilpotent actions, the system
$(X,\CZ_{s_{m-1},H_{m-1}}(\CX),\mu,H_{m-1})$ is an inverse limit of systems $(X,\mathcal{X}_{i},\mu,H_{m-1})$, each of which is isomorphic to an $s_{m-1}$-step $H_{m-1}$-nilsystem. Here, the factor maps $$\pi_{ i,i'}\colon (X,\mathcal{X}_{i},\mu,H_{m-1})\to (X,\mathcal{X}_{i'},\mu,H_{m-1})$$ for $i'<i$ are pointwise identities. 

Our first claim is that  $\CZ_{s_{m-1},H_{m-1}}(\CX)$ is $T_m$-invariant, and hence $(X,\mathcal{X}_{i},\mu,H_m)$ is a well-defined system (recall from the definition that $H_m = \langle H_{m-1},T_m\rangle$). Indeed, $$T_{m}^{-1}\mathcal{X}_{i}=T_{m}^{-1}\CZ_{s_{m-1}, H_{m-1}}(\mathcal{X}_{i})=\CZ_{s_{m-1}, H_{m-1}}(\CX_i)=\mathcal{X}_{i}$$
by Propositions \ref{P: Z_s(T_1) is T_2-invariant} and \ref{P: nilsystems are systems of order k}, and the assumption that $(X,\mathcal{X}_{i},\mu,H_{m-1})$ is isomorphic to an $s_{m-1}$-step $H_{m-1}$-nilsystem.

By Corollary \ref{C: Bryna}, $(X,\mathcal{X}_{i},\mu,H_m)$ is isomorphic to an $H_m$-nilsystem of step $s_m:=2s_{m-1}^2$. 
We claim that $\mathcal{X}_{i}\subseteq \CZ_{s_m,H_m}(\CX)$ for every $i$. This is because
\begin{align*}
    \CZ_{s_m,H_m}(\mathcal{X})\supseteq \CZ_{s_m,H_m}(\mathcal{X}_{i}){=\mathcal{X}_{i}}.
\end{align*}
The first inclusion follows from the fact that the $\sigma$-algebras $\CZ_{s_m,H_m}(\mathcal{X})$ and $\CZ_{s_m,H_m}(\mathcal{X}_i)$ are generated by degree-$s$ dual functions along $H_m$ coming from $\CX$ and $\CX_i$-measurable functions respectively, and so the first is clearly at least as large as the second. The second equality follows from Proposition~\ref{P: nilsystems are systems of order k}. Hence  $(X,\mathcal{X}_{i},\mu,H_{m})$ is a factor of $(X,\CZ_{s_m,H_m}(\CX),\mu,H_{m})$ for every $i$, and so $(X,\CZ_{s_{m-1},H_{m-1}}(\CX),\mu,H_{m})$ is a factor of $(X,\CZ_{s_m,H_m}(\CX),\mu,H_{m})$ by the universal property of inverse limits. Therefore
\begin{align*}
    \CZ_{s,T_j}(\CX)\subseteq \CZ_{s_{m-1}, H_{m-1}}(\CX)\subseteq \CZ_{s_m, H_m}(\CX)
\end{align*}
for every $1\leq j\leq m-1$ by induction. To obtain the missing inclusion $\CZ_{s,T_m}(\CX)\subseteq \CZ_{s_m,H_m}(\CX)$, we repeat the entire argument with the roles of $T_{m-1}$ (say) and $T_m$ swapped.

    \section{Equidistribution on nilsystems}\label{S: equidistribution}
In this section, we collect positive and negative results for equidistribution problems on nilsystems.

\subsection{Equidistribution for independent polynomials}

Our first goal is to prove Theorem~\ref{T: equidistribution}, the last missing ingredient in the proof of Theorem~\ref{T: nilpotent joint ergodicity}. That is, given totally ergodic nilrotations $S_1, \ldots, S_\ell$ on a nilmanifold $Y$  as well as independent polynomials $p_1, \ldots, p_\ell\in\Z[n]$, we show that the orbit 
\begin{align}\label{E: orbit}
         (S_1^{p_1(n)}y, \ldots, S_\ell^{p_\ell(n)}y)_{n\in\Z}
\end{align}
is equidistributed in $Y^\ell$ for $m_Y$-a.e. $y\in Y$.
Our proof follows the strategy laid out by Frantzikinakis and Kra \cite{FrKr05}, who established Theorem~\ref{T: equidistribution} in the single-transformation case. That is, it can be split into three steps:
\begin{enumerate}
    \item pass to a nilmanifold with abelian connected component using  Leibman's equidistribution theorem \cite{L05c, L05b};
    \item pass to unipotent affine transformations on a torus following \cite[Proposition~3.1]{FrKr05};
    \item establish an equidistribution result on a torus using Weyl's equidistribution theorem.
\end{enumerate}

One difficulty present in our setup that did not arise in \cite{FrKr05} is that the unipotent affine transformations on the torus to which we reduce in step (ii) need to be compatible with each other in that, for instance, they still generate a nilpotent group. This does not a priori follow from \cite{FrKr05}. This motivates our first result: a strengthening of \cite[Proposition~3.1]{FrKr05} that provides additional information on the unipotent affine transformation.

\begin{proposition}\label{P: isom onto unipotent affine}
    Let $Y=G/\Gamma$ be a connected $s$-step nilmanifold and $(G_i)_{i\in\N_0}$ be the lower central series of $G$.
    Suppose that the connected component $G^o\subseteq G$ of the identity is abelian, and 
    let $G_\bullet = (G_i^o)_{i\in\N_0}$ and $\Gamma_\bullet = (\Gamma_i^o)_{i\in\N_0}$ be the filtrations on $G^o$ and $\Gamma^o:=\Gamma\cap G^o$ given by $G^o_i :=G^o\cap G_i$ and $\Gamma^o_i :=\Gamma^o\cap G_i$. 
    Let $m_i :=\dim G_i^o/G_{i+1}^o$ for every $1\leq i\leq s$ as well as $m:=m_1 + \cdots + m_s$.
    Then there exists a map $\phi\colon Y\to\T^m$ such that:
    \begin{enumerate}
        \item $\phi$ induces a group isomorphism from $G_i^o/\Gamma_i^o$ to $\T^{m_i + \cdots + m_s}$ for every $1\leq i\leq s$;
        \item for any $g\in G$, there exist $\alpha\in\T^m$ and a continuous group homomorphism $A:\T^m\to\T^m$ satisfying
    \begin{align}\label{E: group inclusion}
        A(\phi(G_i^o/\Gamma_i^o))\subseteq \phi(G_{i+1}^o/\Gamma_{i+1}^o)
    \end{align}
    for all $1\leq i\leq s$, and such that
    \begin{align*}
        \phi(gy) = (I+A + \cdots + A^{s-1})\phi(y) + \alpha
    \end{align*}
    for all $y\in Y$.
    \end{enumerate}
\end{proposition}

\begin{proof}
    Since $Y$ is connected, we have $G=G^o\Gamma$, and hence $Y=G^{o}\Gamma/\Gamma$ by \cite[Chapter 10, Lemma~11]{HK18}. 
Consequently, we can write $g = g_o \gamma$ for some $g_o\in G^o$ and $\gamma\in \Gamma$. 

Define $B(u) := [\gamma,u]$ for any $u\in G^o$; in particular,
\begin{align}\label{E: group inclusion 0}
 B(G_{j}^o)\subseteq G_{j+1}^o   
\end{align}
for every $1\leq j\leq s$ from the definition of the filtration. An important point is that $B(u)\in G^o$ for every $u\in G^o$ due to the normality of $G^o$, and hence all $B(u)$'s commute with each other and with elements of $G^o$ owing to the abelianness of $G^o$. We claim that $B$ is a continuous group homomorphism. The continuity follows from the continuity of the group operation. To see that $B$ is a group homomorphism, take any $u,v\in G^o$. Then
    \begin{align*}
        B(u v) = B(v)B(u)[B(u),v] = B(u)B(v),
    \end{align*}
    where the first identity is a consequence of the commutator identity $$[a,bc] = [a,c][a,b][[a,b],c]$$
    while the second one follows from the fact that $B(u)$, $v$, and $B(v)$ all commute as elements of the abelian group $G^o$. 
    Finally, from applying \eqref{E: group inclusion 0} iteratively we deduce that the homomorphism $B^{(s)}$ is trivial. 

Using these considerations, we can recast the nilrotation by $g$ as follows. Set any $y\in Y$ as $y=u\Gamma$ for some $u\in G^o$. Moving $\gamma$ to the right and generating commutators on the way, we get
    \begin{align*}
        gy &= g_o\gamma u\Gamma = g_o u\gamma B(u)\Gamma = g_o u B(u)\gamma B^{(2)}(u)\Gamma = \cdots
        = g_o u B(u)\cdots B^{(s-1)}(u) \Gamma,    
    \end{align*}
    using the fact that $B^{(s)}$ vanishes and $\gamma\in\Gamma$. In particular, the elements $g_o, u, B(u), \ldots,$ $B^{(s-1)}(u)$ in the rightmost expression all lie in $G^o$, and so they commute with each other. 

We proceed to construct the map $\phi:Y\to\T^m$. Since $G^{o}$ is connected and simply connected, by \cite[Chapter 10, Theorem~18]{HK18} it admits a Mal'cev coordinate map $\psi\colon\R^{m}\to G^{o}$ with
\begin{align*}
    \psi(\{0\}^{m_{1}+\dots+m_{i-1}}\times\R^{m_{i}+\dots+m_{s}})&=G_i^o\\ \textrm{and}\quad \psi(\{0\}^{m_{1}+\dots+m_{i-1}}\times\Z^{m_{i}+\dots+m_{s}}) &= \Gamma^o_i.
\end{align*}
Since $G^o$ is abelian, $\psi$ is in fact a (continuous) group isomorphism.
We then construct the map $\phi: Y\to \T^m$ by mapping $y = u\Gamma$ (for $u\in G^o$) to $\psi\inv(u)+\Z^m$. All the possible representations of $y$ in this form differ by right-multiplying $u$ by elements of $\Gamma^o=\psi(\Z^m)$, and hence the map $\phi$ is well-defined. In fact, it defines a group isomorphism between $Y$ endowed with the group operation $u\Gamma * v\Gamma = (uv)\Gamma$ (for $u,v\in G^o)$ and $\T^m$.
We then take
\begin{align*}
    \alpha:=\phi(g_o \Gamma)\quad \textrm{and}\quad A(t + \Z^m):=\phi \circ B\circ\psi(t),
\end{align*}
noting that $A$ does not depend on the representation $t$, and it is a continuous group homomorphism because $\phi, B, \psi$ are. 
The property \eqref{E: group inclusion} follows from \eqref{E: group inclusion 0}, and the fact that $\phi$ is a group isomorphism implies that
\begin{align*}
    \phi(gy) &= \phi(g_o\Gamma) + \phi(u\Gamma)+\phi(B(u)\Gamma) + \cdots +\phi(B^{s-1}(u) \Gamma)\\
    &= (I+ A + \cdots + A^{s-1})\phi(y) + \alpha.
\end{align*}
 \end{proof}

 Next, we establish a version of Theorem~\ref{T: equidistribution} for unipotent affine transformations that arise from nilsystems via Proposition~\ref{P: isom onto unipotent affine}.

\begin{proposition}\label{P: equidistribution for unipotent affine}
    For $1\leq i\leq s$, let $m_i\in\T$ and $Y_i := \{0\}^{m_{i-1}}\times \T^{m_{i}+ \cdots + m_s}$ (with $m_0:=0$). Set also $m:=m_1 + \cdots + m_s$.
    Let $p_1, \ldots, p_\ell\in\Z[n]$ be independent polynomials, and let $S_1, \ldots, S_\ell:\T^m\to\T^m$ be ergodic unipotent affine transformations of the form
    \begin{align*}
        S_jy = \alpha_j + (I+A_j + \cdots + A_j^{s-1})y
    \end{align*}
    for   $\alpha_{j}\in\T^{m}$ and some continuous group homomorphism $A_j:\T^m\to\T^m$ satisfying
    \begin{align}\label{E: shifting property of A}
     A_j(Y_i)\subseteq Y_{i+1}   
    \end{align}
    for every $1\leq i\leq s$.    
    Then for almost every $y\in\T^m$, the orbit
    \begin{align}
        (S_1^{p_1(n)}y, \ldots, S_\ell^{p_\ell(n)}y)_{n\in\Z}
    \end{align}
    is dense in $\T^{m\ell}$.
\end{proposition}
\begin{proof}
    Let $R_j:=I+ A_j + \cdots + A_j^{s-1}$. Then $R_j^n=\sum_{i=0}^{s-1}\binom{n+i-1}{i}A_j^i$ (for $n\geq 1)$ and 
    \begin{align*}
        S_j^n y &= R_j^n y + \sum_{i=0}^{n-1}R^i\alpha_j 
        = \sum_{i=0}^{s-1}\binom{n+i-1}{i}A_j^iy + n\alpha_j + \sum_{i=1}^{s-1}\binom{n+i-1}{i+1}A_j^i\alpha_j.
    \end{align*}

    Choose $y=(y_1, \ldots, y_s)$ (with $y_{j}\in\T^{m_{j}}$ for $1\leq j\leq s$) so that the coordinates of $y$ are all $\Q$-independent from $\alpha_1, \ldots, \alpha_s, 1$, in that
    \begin{align*}
        k\cdot y + l_1 \alpha_1 + \cdots + l_s \alpha_s + l_{s+1} = 0
    \end{align*}
    for some $k\in\Z^m$ and $l_1, \ldots, l_{s+1}\in\Z$ only if $k = 0$.
    Such points form a full measure set. Suppose that
    \begin{align*}
        k_1\cdot S_1^{p_1(n)}y + \cdots + k_\ell \cdot S_\ell^{p_\ell(n)}y = 0
    \end{align*}
    for some $k_1, \ldots, k_\ell\in\Z^m$. Given the formula for $S_j^n y$'s above, this becomes
    \begin{multline}\label{E: linear combo 2}
        \sum_{j=1}^\ell \left(\sum_{i=0}^{s-1}\binom{p_j(n)+i-1}{i}(k_j\cdot A_j^i)y\right.\\
        \left.+ p_j(n)(k_j\cdot\alpha_j) + \sum_{i=1}^{s-1}\binom{p_j(n)+i-1}{i+1}(k_j\cdot A_j^i)\alpha_j\right) = 0.
    \end{multline}
    Since all the coordinates are polynomial in $n$, by Weyl's equidistribution theorem it suffices to show that $k_1 = \cdots = k_\ell = 0$.

    Because of the $\Q$-independence assumption on the coordinates, the contribution of each coordinate of $y$ to \eqref{E: linear combo 2} is 0. We analyze one-by-one the contribution of coordinates $y_{s-1}, \ldots, y_1$. Since $A_j^2(Y_{s-1})\subseteq A_j(Y_s) = \{0\}$, the coordinate $y_{s-1}$ does not see the contribution of the $A_j^2, A_j^3,\ldots$
    Therefore, this coordinate contributes exactly
    \begin{align}\label{E: contribution of s-1}
        (p_1(n) k_{1}\cdot A_{1} + \cdots + p_\ell(n) k_{\ell}\cdot A_{\ell})(0,\ldots, 0, y_{s-1}, 0)
    \end{align}
    plus terms independent of $n$ that are of no interest to us. Because of the assumption on the rational independence of coordinates, \eqref{E: linear combo 2} and \eqref{E: contribution of s-1}, and the fact that $A_{j}(Y_{s})=\{0\}$, we have
    \begin{align*}
        (p_1(n) k_{1}\cdot A_{1} + \cdots + p_\ell(n) k_{\ell}\cdot A_{\ell})({Y_{s-1}}) = \{0\}
    \end{align*}
     for all $n$. By the independence of polynomials, it follows that $(k_j\cdot A_j)({Y_{s-1}}) = 0$ for every $1\leq j\leq \ell$. Hence also
    \begin{align*}
        (k_j\cdot A_j^{s-1})(Y) \subseteq (k_j\cdot A_j)(Y_{s-1}) = \{0\}
    \end{align*}
    due to the property \eqref{E: shifting property of A} of the maps $A_j$.
    The consequence of this is that we can lower the upper indices in both sums in \eqref{E: linear combo 2} from $s-1$ to $s-2$. 

    We now repeat the same procedure, this time analyzing the contribution of $y_{s-2}$ to \eqref{E: linear combo 2}. Because of the previous case, the contribution of $y_{s-2}$ is the same as \eqref{E: contribution of s-1}, with $(0, \ldots, 0, y_{s-1},0)$ replaced by $(0,\ldots, 0, y_{s-2},0,0)$, plus terms constant in $n$. Arguing exactly as in the previous case, we conclude that
        \begin{align*}
        (k_j\cdot A_j^{s-2})(Y) \subseteq (k_j\cdot A_j)(Y_{s-2}) = \{0\}.
    \end{align*}
    
    Iterating this procedure, we eventually deduce that $k_j\cdot A_j = 0$ for every $1\leq j\leq \ell$. With that, \eqref{E: contribution of s-1} reduces to
    \begin{align*}
        \sum_{j=1}^\ell p_j(n)(k_j\cdot\alpha_j) = 0.
    \end{align*}
    The linear independence of $p_j$'s implies that $k_j\cdot \alpha_j = 0$ for every $1\leq j\leq \ell$.

    We have thus established that $k_j \cdot A_j = k_j \cdot \alpha_j = 0$ for every $1\leq j\leq \ell$. In particular, $k_j\cdot (S_j^n y) = k_j\cdot y$ for every $n\in\Z$. The unique ergodicity of $S_j$ implies that $k_j = 0$, and so the equidistribution of the orbit follows from Weyl's equidistribution theorem. 
 \end{proof}

    \begin{proof}[Proof of Theorem~\ref{T: equidistribution}]
        Let $Y = G/\Gamma$ and $G^o\subseteq G$ be the connected component of the identity. Since $Y$ admits a totally ergodic nilrotation, it is connected \cite[Chapter 11, Corollary 7]{HK18}. By Leibman's equidistribution theorem \cite{L05b}, it suffices to show that the orbit \eqref{E: orbit} is equidistributed when projected down onto the torus $G/([G^o,G^o]\Gamma)$. Hence on replacing $G^o$ with $G^o/[G^o,G^o]$ and $G/\Gamma$ with $G/([G^o,G^o]\Gamma)$, we can assume that $G^o$ is abelian. We then apply Proposition~\ref{P: isom onto unipotent affine} to obtain an isomorphism
        \begin{align*}
            \phi:(Y, \CB_Y, m_Y, S_1, \ldots, S_\ell)\to (\T^m, \CB_{\T^m}, m_{\T^m}, S_1', \ldots, S_\ell')
        \end{align*}
        for some unipotent affine transformations $S_j'$ satisfying the conditions of Proposition~\ref{P: equidistribution for unipotent affine}. By this result, the orbit
        \begin{align*}
            ({S_1'}^{p_1(n)}y, \ldots, {S_\ell'}^{p_1(n)}y)_{n\in\Z}
        \end{align*}
        is equidistributed on $\T^m$ for $m_{\T^m}$-a.e. $y\in \T^m$, implying the equidistribution of the original orbit on $Y^m$.
    \end{proof}
    
    The next result is an ingredient in the proof of Theorem~\ref{T: joint ergodicity conjecture2}.

    \begin{proposition}\label{P:dcc}
     Let $(Y, \CB_Y, m_Y, S_1, \ldots, S_\ell)$ be a nilsystem and $p_1, \ldots, p_\ell\in\Z[n]$ be polynomials. If $Y$ is not connected and $\ell\geq 2$, then  
  \begin{align}\label{E: orbit0}
        (S_1^{p_1(n)}y_{1}, \ldots, S_\ell^{p_\ell(n)}y_{\ell})_{n\in\Z}
    \end{align}
    is  not equidistributed in $Y^{\ell}$ for any $y_{1},\dots,y_{\ell}\in Y$.
\end{proposition}
 \begin{proof}
   Let $Y = G/\Gamma$, and suppose that it is transitive (if not, the statement follows immediately). 
   Let $Y_0$ be the connected component of $Y$ containing the identity coset $\Gamma$. By transitivity, we can find $g_{\ast}\in G$ for which the orbit $(g_{\ast}^{n}\Gamma)_{n\in\Z}$ is dense in $Y$. On letting $W\in\N$ be the smallest natural number for which $g_{\ast}^{W}\in Y_{0}$, we may write $$Y=\bigcup\limits_{i=0}^{W-1}g^{i}_{\ast}Y_{0}=\bigcup_{i=0}^{W-1}Y_{i},$$
    where $Y_i:= g^i_\ast Y_0$, and the union is disjoint. Since $Y$ is not connected, we have $W\geq 2$. Our assumptions thus say that $W, \ell\geq 2$, and we will show that this forces \eqref{E: orbit0} to lie in a proper subset of $Y^\ell$.

 For any $g\in G$, let $\Omega(g)$ be the unique element in $[W]$ such that $g\Gamma\in Y_{\Omega(g)}$ (i.e. $g^{-\Omega(g)}_{\ast}g\Gamma\in Y_{0}$). Since $G^o$ is a normal subgroup of $G$, it is clear that $\Omega(gg')=\Omega(g)+\Omega(g')$ and $\Omega(g^W) = 1$ for all $g,g'\in W$. Thus, $\Omega\colon G\to \Z/W\Z$ is a group homomorphism.

   Assume that the nilrotations $S_1, \ldots, S_\ell$ take the form $S_{j}y=g_{j}y$ for some $g_j\in G$, and let the elements $y_1, \ldots, y_\ell$ be given by $y_{j}=h_{j}\Gamma$ for some $h_j\in G$. Then $$\Omega(S_{j}^{p_{j}(n)}y_{j})=p_{j}(n)\Omega(g_j)+\Omega(h_j),$$ and thus $\Omega(S_j^{p_j(n+W)}y_{j})=\Omega(S_j^{p_j(n)}y_{j})$ for every $n\in\Z$. It follows that for all $r\in [W]$ and $n\in\Z$, 
      \begin{align*}
   \vec{\omega}(r):=(\Omega(S_{1}^{p_{1}(Wn+r)}y_{1}),\dots,\Omega(S_{\ell}^{p_{\ell}(Wn+r)}y_{\ell}))
   =(\Omega(S_{1}^{p_{1}(r)}y_{1}),\dots,\Omega(S_{\ell}^{p_{\ell}(r)}y_{\ell}))    
   \end{align*}
   lies in $(\Z/W\Z)^{\ell}$ and is independent of $n$.

   
    For $\vec{a}=(a_{1},\dots,a_{\ell})\in (\Z/W\Z)^{\ell}$, define $$\vec{Y}_{\vec{a}}:=Y_{a_{1}}\times\dots \times Y_{a_{\ell}}.$$ Then for every $r\in [W]$, the orbit 
    \begin{align*}
        (S_1^{p_1(Wn+r)}y_{1}, \ldots, S_\ell^{p_\ell(Wn+r)}y_{\ell})_{n\in\Z}
    \end{align*}
    takes values in $Y_{\vec{\omega}(r)}$.
Hence
    \begin{align}\label{E: inclusion}
        \overline{(S_1^{p_1(n)}y_{1}, \ldots, S_\ell^{p_\ell(n)}y_{\ell})_{n\in\Z}}\subseteq\bigcup\limits_{r\in [W]}\vec{Y}_{\vec{\omega}(r)} \subseteq \bigcup\limits_{\vec{a}\in (\Z/W\Z)^{\ell}}\vec{Y}_{\vec{a}}=Y^\ell.
    \end{align}
    Now, the sets $\{\vec{Y}_{\vec{a}}\colon\; \vec{a}\in (\Z/W\Z)^{\ell}\}$ are all disjoint, and there are $W^\ell$ of them. By assumption, $W, \ell\geq 2$, hence $W^\ell>W$, and hence the second inclusion in \eqref{E: inclusion} is strict. 
    \end{proof}

       Proposition~\ref{P:dcc} may be the unique result in this paper that does not extend to integer-valued polynomials. Let $Y=\Z/2\Z$ and $S_1 y = S_2 y = y + 1$ mod 1; then the orbit $(S_1^n y_1, S_2^{\binom{n}{2}}y_2)_{n\in\Z}$ is equidistributed in $Y^2$ for every $y_1,y_2\in Y$.
       However, there exists a more complicated version of Proposition~\ref{P:dcc} that delivers the equidistribution result sufficient to establish Theorem~\ref{T: joint ergodicity conjecture2} for independent integer-valued polynomials, which we omit for the sake of conciseness.

\subsection{Counterexamples to joint ergodicity}\label{SS: counterexamples}
We conclude with the proofs of Theorems \ref{T: counterexample to joint ergodicity conjecture} and \ref{T: counterexample to joint ergodicity criteria}, i.e. counterexamples to various joint ergodicity statements in the 2-step nilpotent setting. 
\begin{proof}[Proof of Theorem~\ref{T: counterexample to joint ergodicity conjecture}]
    Let
    \begin{align*}
        T_{1}(x,y)=(x+a,y+2x+a)\quad\textrm{and}\quad T_{2}(x,y)=(x+b,y+2x+b)
    \end{align*}
     with $1,a,b$ being $\mathbb{Q}$-independent, and take $p_1(n) = p_3(n) = n$, $p_2(n) = p_4(n) = n^2$. 
    Since
    $$T_{2}^{n}T_{1}^{-n}(x,y)=(x+(b-a)n,y+(b-a)n^2)$$
    is equidistributed on $\T^2$,
     the sequence $(T_{2}^{n}T_{1}^{-n})_{n\in\Z}$ is ergodic. The ergodicity of $(T_{2}^{n^{2}}T_{1}^{-n^{2}})_{n\in\Z}$ follows from the same computation upon substituting $n^2$ for $n$. 
     The sequence $(T_{2}^{n^{2}}T_{1}^{-n})_{n\in\Z}$ is ergodic because 
     \begin{align*}
         T_2^{n^2}T_1^{-n}(x,y) = (x + n^2 b - na, y + 2n(n-1)x + n^2(1-2n)a + n^4 b)
     \end{align*}
     is equidistributed on $\T^2$,
     and likewise for the other three cross-terms. Hence the difference ergodicity condition is satisfied. 

     To verify the other two claims, consider the orbit
     \begin{multline*}
         (T_1^n(x_1,y_1), T_2^n(x_2,y_2), T_3^{n^2}(x_3,y_3), T_4^{n^2}(x_4,y_4))\\ = (x_{1}+na,y_{1}+2nx_{1}+n^{2}a,
          x_{2}+nb,y_{2}+2nx_{2}+n^{2}b,\\ x_{3}+n^{2}a,y_{3}+2n^{2}x_{3}+n^{4}a, x_{4}+n^{2}b,y_{4}+2n^{2}x_{4}+n^{4}b).
     \end{multline*}
     It is equidistributed on $\T^8$ for $m_{\T^8}$-a.e. $(x_1, y_1, \ldots, x_4, y_4)\in \T^8$, hence the product ergodicity condition is also satisfied.
    On the other hand, whenever
    \begin{align*}
     (x_1,y_1) = (x_2,y_2) = (x_3,y_3) = (x_4,y_4),
    \end{align*}
    the orbit lies inside the subtorus
    \begin{align*}
        \{(t_1, \ldots, t_8)\in\T^8\colon\; t_{2}-t_{4}=t_{5}-t_{7}\}.
    \end{align*}
    Hence we do not have joint ergodicity, as this would correspond to the equidistribution of the orbit
    \begin{align}\label{E: counterexample orbit 2}
        (T_1^n(x,y), T_2^n(x,y), T_3^{n^2}(x,y), T_4^{n^2}(x,y))_{n\in\Z}
    \end{align}
    on $\T^8$ for $m_{\T^2}$-a.e. $(x,y)\in\T^2.$
\end{proof}

\begin{proof}[Proof of Theorem~\ref{T: counterexample to joint ergodicity criteria}]
    Take the same system and polynomials as in the proof of Theorem~\ref{T: counterexample to joint ergodicity conjecture}. The failure of joint ergodicity has already been established there. On the other hand, the projection of the orbit \eqref{E: counterexample orbit 2} onto the first coordinate is 
\begin{align*}
    (x+na, x+nb, x+n^{2}a, x+n^{2}b),
\end{align*}
and this is clearly equidistributed on $\T^4$ due to the rational independence assumption. For $T_1, T_2, T_3, T_4$, the projection onto the first coordinate corresponds to acting by $T_1, T_2, T_3, T_4$ on the Kronecker factor of each of these transformations.
\end{proof}

\section{Proofs of the joint ergodicity results}\label{S: proofs of main results}
We now combine all the pieces to conclude the proof of Theorem~\ref{T: nilpotent joint ergodicity}.
\begin{proof}[Proof of Theorem~\ref{T: nilpotent joint ergodicity}]
Let $T_1, \ldots, T_\ell$ be totally ergodic transformations on $(X, \CX, \mu)$ generating a 2-step nilpotent group and $p_1, \ldots, p_\ell\in\Z[n]$ be nonconstant polynomials of distinct degrees. By Theorem~\ref{T: nilpotent seminorm estimates}, there exists $s\in\N$ so that 
  \begin{align}\label{E: average ultimate}
    \E_{n\in\Z}T_1^{p_1(n)}f_1\cdots T_\ell^{p_\ell(n)}f_\ell
\end{align}
vanishes in $L^2(\mu)$
whenever $\nnorm{f_j}_{s+1,T_j} = 0$ for some $1\leq j\leq \ell$. Hence we can assume without loss of generality that every $f_j$ in \eqref{E: average ultimate} is $\CZ_s(T_j)$-measurable. Let $H := \langle T_1, \ldots, T_\ell\rangle$. By Proposition~\ref{P: comparing HK factors in nilpotent systems}, we can find $s'\in\N$ for which
\begin{align*}
    \bigvee_{j=1}^\ell \CZ_s(T_j)\subseteq \CZ_{s'}(H),
\end{align*}
and hence we can assume that all $f_j$'s are $\CZ_{s'}(H)$-measurable. By Theorem~\ref{T: Candela-Szegedy}, the system $(X, \CZ_s(H), \mu, T_1, \ldots, T_\ell)$ is an $s'$-step $H$-pronilsystem. By an $L^2(\mu)$-approximation argument, we can thus assume that our system is an $H$-nilsystem. The result then follows from Theorem~\ref{T: equidistribution}, which implies a pointwise version of Theorem~\ref{T: nilpotent joint ergodicity} for nilsystems. 
\end{proof}

We move on to derive Theorem~\ref{T: joint ergodicity conjecture} from Theorem~\ref{T: new estimates}.
\begin{proof}[Proof of Theorem~\ref{T: joint ergodicity conjecture}]
    Fix a $\Z^D$-system $(X, \CX, \mu, U)$ and polynomials $p_1, \ldots, p_\ell\in\Z^D[\bn]$ (where $\bn$ is an indeterminate in $\Z^L$). In one direction, joint ergodicity implies the difference ergodicity condition by \cite[Proposition~6.2]{DFKS22} and the product ergodicity condition by essentially the same proof as in \cite[Proposition~6.2]{DFKS22} and \cite[Lemma~12.6]{DKKST24}. Conversely, suppose that the product and difference ergodicity conditions hold. We want to show that the average
    \begin{align}\label{E: average in jec}
        \E_{\bn\in\Z^L}U_{p_1(\bn)}f_1\cdots U_{p_\ell(\bn)}f_\ell
    \end{align}
    admits Host-Kra seminorm control and is good for equidistribution (see \cite[Section 1.1]{BFM22} for the precise definitions), and then apply the joint ergodicity criteria for $\Z^D$-actions from \cite[Theorem~1.1]{BFM21}. The property of being good for equidistribution follows easily from the product ergodicity condition (by adapting the proof of \cite[Corollary 12.5]{DKKST24}), so we focus on establishing Host-Kra seminorm control, which is the only nontrivial part of this proof.
    
    To this end, we conclude from Theorem~\ref{T: new estimates} that \eqref{E: average in jec} is controlled by box seminorms along (many copies of) the subgroups $H_{j,{j'}}$ for $0\leq j<{j'}\leq \ell$. Now, by the product ergodicity condition, the subgroup $H_{0,j'}$ acts ergodically; and by the difference ergodicity condition, the subgroup $\tilde H_{j,j'}$ is ergodic for every $1\leq j<j'\leq \ell$. Using this input, the inclusions $\tilde H_{j,j'}\subseteq H_{j,j'}$, and the subgroup property of box seminorms, it follows that \eqref{E: average in jec} is controlled by a box seminorm (in terms of any of $f_1, \ldots, f_\ell$) that involves only ergodic subgroups. 
    From \eqref{E: same invariant factors} it follows that for any $1\leq j\leq \ell$, the average \eqref{E: average in jec} is controlled by $\nnorm{f_j}_{s,T_j}$ for some fixed $s\in\N$. This concludes the proof of Host-Kra seminorm estimates, and hence also of the equivalence between joint ergodicity and the product and difference ergodicity conditions.
\end{proof}

We conclude with the proof of Theorem~\ref{T: joint ergodicity conjecture2}.

\begin{proof}[Proof of Theorem~\ref{T: joint ergodicity conjecture2}]

As in the preceding proof, joint ergodicity implies the product ergodicity condition by essentially the same argument as in \cite[Proposition~6.2]{DFKS22} and \cite[Lemma~12.6]{DKKST24}.\footnote{They are stated for commuting transformations, but the proofs extend to general transformations.} Conversely, assume the product ergodicity condition and suppose that $\ell>1$ since otherwise joint ergodicity and product ergodicity condition are the same. 
Arguing as in the proof of Theorem~\ref{T: nilpotent joint ergodicity}, we may assume without loss that our system is a nilsystem $(Y, \CB_Y, m_Y, S_1, \ldots, S_\ell)$. By the product ergodicity condition, the orbit
\begin{align}\nonumber
        (S_1^{p_1(n)}y_{1}, \ldots, S_\ell^{p_\ell(n)}y_{\ell})_{n\in\Z}
    \end{align}
    is equidistributed in $Y^{\ell}$ for some (in fact, for $m_{Y^\ell}$-a.e.) $(y_{1},\dots,y_{\ell})\in Y^\ell$. By Proposition~\ref{P:dcc}, the nilsystem $Y$ is connected, and so $S_1, \ldots, S_\ell$ are all totally ergodic by \cite[Chapter 11, Corollary 7]{HK18}. Then
    \begin{align}\nonumber
        (S_1^{p_1(n)}y, \ldots, S_\ell^{p_\ell(n)}y)_{n\in\Z}
    \end{align}
    is equidistributed in $Y^\ell$ for $m_Y$-a.e. $y\in Y$ by Theorem~\ref{T: equidistribution}, which implies joint ergodicity just like in the proof of Theorem~\ref{T: nilpotent joint ergodicity}. 
\end{proof}

\section{Open questions}\label{S: open problems}

The investigation conducted in this paper leads to a wealth of open problems; we collect some of them in this section. Other open problems on nilpotent systems can be found in \cite[Section 5]{BL02} and \cite[Section 7]{Kuc26}.

\subsection{Conjecture \ref{C: main conjecture} and related problems}\label{SS: main open problems}
In our opinion, the most tantalizing open problem arising from this work is Conjecture \ref{C: main conjecture}. Since the derivation of the seminorm estimates in Theorem~\ref{T: nilpotent seminorm estimates} relies strongly on dealing with 2-step nilpotent systems and distinct-degree polynomials, the natural place to start the hunt for Conjecture \ref{C: main conjecture} are the two problems below.

\begin{problem}\label{Pr: stronger seminorm estimates 1}
    Extend Theorem~\ref{T: nilpotent seminorm estimates} to higher-step nilpotent systems.
\end{problem}
\begin{problem}\label{Pr: stronger seminorm estimates 2}
    Extend Theorem~\ref{T: nilpotent seminorm estimates} to all independent (rather than distinct-degree) polynomials.
\end{problem}
The simplest open case of Problem~\ref{Pr: stronger seminorm estimates 1} is the average
\begin{align}\label{E: n, n^2 problem}
    \E_{n\in\Z} T^n f_1\cdot S^{n^2}f_2
\end{align}
with $T, S$ generating a 3-step nilpotent group while the model case for Problem~\ref{Pr: stronger seminorm estimates 2} is
\begin{align}\label{E: n^2, n^2 + n problem}
    \E_{n\in\Z} T^{n^2} f_1\cdot S^{n^2+n}f_2
\end{align}
for a 2-step nilpotent action of $T, S$. 

We find the second of these model problems more approachable given the amount of group-theoretic intricacies inherent in the study of 3-step nilpotent groups. Even so, tackling \eqref{E: n^2, n^2 + n problem} in the 2-step nilpotent regime will likely require significant new insights, in analogy to the new ideas needed to pass from \eqref{E: n, n^2 problem} to \eqref{E: n^2, n^2 + n problem} in the commuting setting. A major stumbling block is the failure of the seminorm smoothing argument for nilpotent systems, which, like degree lowering, uses dual-difference interchange in a crucial way. This leads to the following meta-problem.
\begin{problem}\label{Pr: degree lowering}
    Find a workable version of or a substitute for the degree lowering and seminorm smoothing arguments in the (2-step) nilpotent setting. 
\end{problem}

As a preparation for the 3-step nilpotent version of Theorem~\ref{T: nilpotent seminorm estimates}, one probably needs to address the following problem first.
\begin{problem}\label{Pr: nilpotent box seminorm estimates}
    Find box seminorm estimates for polynomial actions on 2-step nilpotent systems in the spirit of Theorems \ref{T: original estimates} and \ref{T: new estimates}.
\end{problem}

A related problem is to find a quantitative version of Theorem~\ref{T: nilpotent seminorm estimates}.
\begin{problem}\label{Pr: quantitative estimates}
    Quantify Theorem~\ref{T: nilpotent seminorm estimates}. In particular, can we get a polynomial dependence between box seminorms and the $L^2(\mu)$ norms of the averages for all 1-bounded functions?
\end{problem}
In the commuting case, the quantification of box seminorm estimates from \cite{DFKS22} was carried out in \cite{DKKST24}, relying on the quantitative concatenation argument from \cite{KKL24a, Kuc23}. However, the use of relative concatenation in the proof of Theorem~\ref{T: nilpotent seminorm estimates} makes the end result fully qualitative. Thus, Problem~\ref{Pr: quantitative estimates} reduces in essence to finding a quantitative version of Theorem~\ref{T: relative concatenation}, the relative concatenation result from \cite{DKKST25}, a problem previously mentioned in the second author's survey \cite[Problem~23]{Kuc26}.

The resolution of Problems \ref{Pr: degree lowering} and \ref{Pr: quantitative estimates} would be of great help for the study of pointwise almost everywhere convergence of nilpotent ergodic averages. The most successful approach to pointwise convergence employs quantitative methods of harmonic analysis and additive combinatorics. Indeed, most recent breakthroughs (e.g. \cite{KMPWW24, KMT20, KMTT24, MWW25, Ter24}) relied in an essential way on degree lowering and quantitative concatenation coupled with a healthy amount of hardcore harmonic machinery. It is therefore hard to imagine advances on the pointwise convergence of nilpotent averages without finding appropriate substitutes for these arguments.

A natural problem that does not require new seminorm estimates would be to remove the total ergodicity assumption from Theorem~\ref{T: nilpotent joint ergodicity}. 
\begin{problem}\label{Pr: Krat}
    Show that for all 2-step nilpotent systems $(X, \CX, \mu, T_1, \ldots, T_\ell)$ and nonconstant distinct-degree polynomials $p_1, \ldots, p_\ell\in\Z[n]$, the rational Kronecker factor\footnote{The \textit{rational Kronecker factor} of $(X, \CX, \mu, T)$ is the $\sigma$-algebra $\Krat(T) = \bigvee_{n=1}^\infty \CI(T^n)$, i.e. the largest periodic factor of the system.} is characteristic for the average
    \begin{align}\label{E: Krat control}
        \E_{n\in\Z}T_1^{p_1(n)}f_1\cdots T_\ell^{p_\ell(n)}f_\ell;
    \end{align}
    that is, \eqref{E: Krat control} is 0 whenever $f_j$ is orthogonal to $\Krat(T_j)$ for some $1\leq j\leq \ell$.
\end{problem}
The same conclusion is expected to hold for all nilpotent systems and all independent polynomials. 

Even in the commuting case, the rational Kronecker factor for averages \eqref{E: Krat control} was not established in the original work of Chu, Frantzikinakis, and Host \cite{CFH11}; the issue was the absence of a good equidistribution-on-nilsystems result. This extension was only proved by Frantzikinakis and the second author \cite{FrKu22a}, and it required the whole degree-lowering machinery. The failure of degree lowering in the nilpotent world makes it less clear how to approach Problem~\ref{Pr: Krat}. 

{A strong structural description like the one in Problem~\ref{Pr: Krat} would likely be needed to obtain the following special case of Conjecture \ref{Conj: popular common differences}.
\begin{problem}
    Show that the conclusions of Theorem~\ref{T: Khintchine} and Corollary \ref{Conj: popular common differences} hold for all nonconstant distinct-degree polynomials $p_1, \ldots, p_\ell\in\Z[n]$.
\end{problem}
}

We conclude this section with a topological analog of Conjecture \ref{C: main conjecture}.
\begin{problem}\label{Pr: joint transitivity}
    Let $(X, T_1, \ldots, T_\ell)$ be a topological system. Suppose that the group $\langle T_1, \ldots, T_\ell\rangle$ is nilpotent, acts minimally on $X$, and that each $T_1, \ldots, T_\ell$ is totally transitive.\footnote{The \textit{minimality} of $\langle T_1, \ldots, T_\ell\rangle$ means that the orbit of every point under this action is dense in $X$ whereas the \textit{total transitivity} of $T_j$ means that for every $r\in\N$ we can find $x_{j,r}\in X$ such that the orbit $\rem{T_j^{rn}x_{j,r}\colon\; n\in\Z}$ is dense in $X$.} If $p_1, \ldots, p_\ell\in\Z[n]$ are independent polynomials, show that {$(T_1^{p_1(n)}, \ldots, T_\ell^{p_\ell(n)})_{n\in\Z}$ is \emph{jointly transitive}, meaning that} the orbit
    \begin{align}\label{E: top orbit}
        (T_1^{p_1(n)}x, \ldots, T_\ell^{p_\ell(n)}x)_{n\in\Z}
    \end{align}
    is dense in $X^\ell$ for a $G_\delta$-set of $x\in X$.
\end{problem}

The problem remains open even under the stronger assumption that each $T_1, \ldots, T_\ell$ is totally minimal. When $T_1 = \cdots = T_\ell$ is totally minimal, the density of \eqref{E: top orbit} in $X^\ell$ for a $G_\delta$-set of points was established by Qiu \cite{Qiu23}. In the nilpotent setting, the same conclusion was reached  by Huang, Shao, and Ye \cite{HSY19} for arbitrary pairwise independent polynomials under the stronger assumption that each nontrivial element of $\langle T_1, \ldots, T_\ell\rangle$ is (topologically) weakly mixing and minimal.

\subsection{Hardy sequences}\label{SS: Hardy}

A result on polynomial ergodic averages can often be extended to other suitable sequences of polynomial growths, such as fractional powers $\sfloor{n^c}$ (for $c\in\R_+\backslash\Z$) and other suitable Hardy sequences. These extensions are never trivial; in fact, they often require a more subtle (and quantitative) PET induction scheme as well as more robust equidistribution results on nilsystems. This leads to the following problem.
\begin{problem}\label{Pr: Hardy}
    Let $0< c_1 < \cdots < c_\ell$ be non-integers. Show that for any 2-step nilpotent system $(X, \CX, \mu, T_1, \ldots, T_\ell)$ and 1-bounded functions $f_1, \ldots, f_\ell\in L^\infty(\mu)$, the following extensions of Theorems \ref{T: nilpotent joint ergodicity} and \ref{T: nilpotent seminorm estimates} hold:
    \begin{enumerate}
        \item (Norm convergence) the average
        \begin{align}\label{E: Hardy}
            \E_{n\in [N]}T_1^{\sfloor{n^{c_1}}}f_1\cdots T_\ell^{\sfloor{n^{c_1}}} f_\ell
        \end{align}
        converges in $L^2(\mu)$ as $N\to\infty$;
        \item (Seminorm estimates) there exists $s = O_{c_1, \ldots, c_\ell}(1)$ such that \eqref{E: Hardy} goes to 0 in $L^2(\mu)$ whenever $\nnorm{f_j}_{s,T_j}= 0$ for some $1\leq j\leq \ell$;
        \item (Joint ergodicity) \eqref{E: Hardy} converges to $\prod_{j=1}^\ell \int f_j\; d\mu$ in $L^2(\mu)$ whenever $T_1, \ldots, T_\ell$ are ergodic;
        \item (Invariant factor control) \eqref{E: Hardy} converges to $\prod_{j=1}^\ell \E(f_j|\CI(T_j))$ in $L^2(\mu)$ (without any ergodicity assumptions).
    \end{enumerate}
\end{problem}
All the questions above can also be asked for the average
\begin{align*}
    \E_{n\in[N]}T_1^n f_1\cdot T_2^{\sfloor{n^{3/2}}}f_2,
\end{align*}
 which does not satisfy the assumptions of Problem~\ref{Pr: Hardy} (because $n$ is not a non-integer power) but is an easier and a more natural place to start.

Of course, the sequences $n^{c_1}, \ldots, n^{c_\ell}$ can be replaced by other logarithmico-exponential Hardy sequences that have sufficiently distinct {polynomial} growth or satisfy more general independence conditions.  In the commuting case, subcases of Problem~\ref{Pr: Hardy} were addressed by Frantzikinakis \cite{Fr12}, while the full problem was resolved by Donoso, Tsinas, and the authors \cite{DKKST24}. 

Compared to the polynomial setting, we expect two major complications in Problem~\ref{Pr: Hardy}. Both the seminorm estimates and the equidistribution result needed for part (iii) are likely be far more challenging given that one will often need to work with finite averages, coming from splitting the averaging range into short intervals. The second of these challenges can be isolated as a separate problem (which we state in a slightly higher generality, analogous to the statement of Theorem~\ref{T: equidistribution}).
\begin{problem}\label{Pr: Hardy equidistribution}
    Let $(Y, \CB_Y, m_Y, S_1, \ldots, S_\ell)$ be a nilsystem with $S_1, \ldots, S_\ell$ ergodic. Show that for any non-integers $0< c_1 < \cdots <c_d$ and  independent fractional polynomials $a_j(n) = \sum_{i=1}^d b_{j,i}n^{c_i}$, the orbit
    \begin{align*}
        (S_1^{\sfloor{a_1(n)}}y, \ldots, S_\ell^{\sfloor{a_\ell(n)}}y)_{n\in\N}
    \end{align*}
    is equidistributed on $Y^\ell$ for $m_Y$-a.e. $y\in Y$.
\end{problem}
More generally, we can replace fractional polynomials by other Hardy sequences whose all nontrivial linear combinations are equidistributed. A useful inspiration for Problem~\ref{Pr: Hardy equidistribution} can be found in works on the single-transformation version of Problem~\ref{Pr: Hardy equidistribution}, e.g. \cite{Fr09, Fr10, R22, Ts24}.

\subsection{Quantitative nilpotent Szemer\'edi theorems}

The next question inquires about quantitative Szemer\'edi theorems in nilpotent groups.
\begin{problem}\label{Pr: quantitative nilpotent Szemeredi}
    Let $(\Z^3, *)$ be the discrete Heisenberg group, i.e. the group $\Z^3$ with the 2-step nilpotent operation
    \begin{align*}
        (a,b,c)*(x,y,z) := (a+x, b+y, c + z + ay).
    \end{align*}
    Consider the F{\o}lner sequence $\Phi=(\Phi_N)_{N\in\N}$ given by $\Phi_N:= [N]\times[N]\times[N^2]$.
    Fix $g,h\in \Z^3$. How big can $A\subseteq \Phi_N$ be if it does not contain a nontrivial pattern
    \begin{enumerate}
        \item $u,\; g^nu,\; h^{n}u$ for any $u\in \Z^3$ and $n\in\N$;
        \item $u,\; g^nu,\; h^{n^2}u$ for any $u\in \Z^3$ and $n\in\N$?
    \end{enumerate}
\end{problem}
Problem~\ref{Pr: quantitative nilpotent Szemeredi} can be stated without reference to the nilpotent structure. For instance, if $g = (1,0,0)$ and $h=(0,1,0)$, the first pattern corresponds to 
\begin{align}\label{E: nilpotent config 1}
(x,y,z),\; (x + n, y, z+ y\binom{n}{2}),\; (x, y + n, z)    
\end{align}
while the second one takes the form
\begin{align}\label{E: nilpotent config 2}
(x,y,z),\; (x + n, y, z +y\binom{n}{2}),\; (x, y + n^2, z).  
\end{align}
Such patterns differ from the configurations appearing in the polynomial Szemer\'edi theorem in that the coefficients of the polynomials depend on the point $(x,y,z)$ - albeit in a rather special way.

Although quantitative Szemer\'edi theorems are known to hold in some nonabelian groups, the existing results generally assume high level of \textit{quasirandomness} of the group (see \cite{Go08} for the definition and properties of quasirandom groups and \cite{A16, Pel18b, Tao13} for samples of such results). This is a very different setting from ours; while nilpotent groups are just barely nonabelian, quasirandom groups are quite far from being abelian. A useful analogy for dynamicists would be that nilpotent groups behave like nilsystems whereas quasirandom groups enjoy mixing-style properties. 

{In certain cases, one can also derive results for nonabelian groups from corresponding results for abelian groups - see e.g. \cite[Corollaries 1.4-1.6]{JLLOS25}. If $h$ is a power of $g$, then both parts of Problem \ref{Pr: quantitative nilpotent Szemeredi} are essentially equivalent to analogous configurations over $\Z$ by considering left translates of $\{g^n\colon n\in[N]\}$.\footnote{To see this, partition $\Phi_N = \bigcup\limits_{u\in U_N}(\{ug^n\colon n\in[N]\}\cap\Phi_N)$. If $h=g^2$, then any pattern $u, ug^n, ug^{2n}$ lies in one cell, and so we can translate the problem to $[N]$. Indeed, if $B\subseteq [N]$ is a set without three term arithmetic progressions, then $A = \bigcup\limits_{u\in U_N}\{ug^n\colon n\in B\}$ contains no nontrivial pattern $u, ug^n, ug^{2n}$. Hence upper bounds in Roth's theorem give upper bounds for the pattern $u, ug^n, ug^{2n}$, and taking $B$ to be the Behrend set gives corresponding lower bounds. A similar reasoning can be applied to $u, ug^n, ug^{n^2}$. We thank James Leng for pointing this out.} Likewise, if $g,h$ commute, then one can likely reduce the question to the corresponding problem over $\Z^2$. We do not know of a similar reduction whenever $g,h$ fail to commute.}

\subsection{Pointwise convergence} Many problems stated in Sections \ref{SS: main open problems} and \ref{SS: Hardy} admit natural pointwise analogs. The examination of pointwise convergence gained considerable momentum in recent years, driven by the breakthroughs of Krause-Mirek-Tao \cite{KMT20} and Kosz-Mirek-Peluse-Wan-Wright \cite{KMPWW24}. Pointwise problems differ considerably from the study of $L^2(\mu)$ limits in that they heavily rely on the quantitative methods of modern harmonic analysis and additive combinatorics. Among many possible open problems of this flavor, the following seems the most approachable.
\begin{problem}\label{Pr: n, n^2 pointwise}
    Let $(X, \CX, \mu, T, S)$ be a 2-step nilpotent system. Show that for all $f_1, f_2\in L^\infty(\mu)$ and $\mu$-a.e. $x\in X$, the average
    \begin{align*}
        \E_{n\in[N]}f_1(T^n x)f_2(S^{n^2}x)
    \end{align*}
    converges as $N\to\infty$. 
\end{problem}
Problem \ref{Pr: n, n^2 pointwise} 
admits the following generalization, which can be thought of as a 2-step nilpotent generalization of the main result of \cite{KMPWW24}.
\begin{problem}\label{Pr: distinct-degree pointwise}
     Let $(X, \CX, \mu, T_1, \ldots, T_\ell)$ be a nilpotent system and  $p_1, \ldots, p_\ell\in\Z[n]$ be nonconstant distinct-degree polynomials. Show that for all functions $f_1, \ldots, f_\ell\in L^\infty(\mu)$ and $\mu$-a.e. $x\in X$, the average  
     \[\E_{n\in[N]}f_1(T_1^{p_1(n)}x)\cdots f_\ell(T_\ell^{p_\ell(n)}x)\]
     converges as $N\to\infty$. 
\end{problem}
The resolution of Problems \ref{Pr: n, n^2 pointwise} and \ref{Pr: distinct-degree pointwise} would require several far-reaching innovations. To emulate the strategy of \cite{KMPWW24, KMT20}, one first needs a quantitative, finitary version of the seminorm estimates from Theorem \ref{T: nilpotent seminorm estimates}; see Problem \ref{Pr: quantitative estimates} and the subsequent discussion. The second step would be an optimal inverse theorem for the corresponding finitary averages. The failure of degree lowering in the nilpotent setting is a major drawback here, which motivates Problem \ref{Pr: degree lowering} above. The third ingredient would be a nilpotent variant of the multilinear circle method from \cite{KMPWW24, KMT20}.

Problems \ref{Pr: n, n^2 pointwise} and \ref{Pr: distinct-degree pointwise} are special cases of the far more general Furstenberg-Bergelson-Leibman conjecture whose version is presented below.
\begin{problem}\label{Pr: FBL}
Let $(X, \mathcal{X}, \mu, T_1,\ldots, T_\ell)$ be a nilpotent system and $p_{j,i} \in \Z[n]$ for $1\leq j\leq\ell$ and $1\leq i\leq k$. 
Show that for all $f_1, \ldots, f_\ell \in L^\infty(\mu)$ and $\mu$-a.e. $x\in X$, the average
\[ \E_{n\in [N]} f_1\Bigbrac{\prod_{i=1}^k T_i^{p_{1,i}(n)}x}\cdots f_\ell\Bigbrac{\prod_{i=1}^k T_i^{p_{\ell,i}(n)}x}\]
converges as $N\to\infty$.
\end{problem}
Problem \ref{Pr: FBL} inquires about the pointwise version of Walsh's convergence theorem (Theorem~\ref{T: Walsh}). It is currently the most important open problem in pointwise ergodic theory and a major challenge in modern harmonic analysis. It is an immensely difficult problem, well beyond reach of the current machinery.
In the nilpotent setting, the only known case is the result of Ionescu, Magyar, Mirek, and Szarek  \cite{IMMS23} for single averages (i.e. $\ell=1$) and 2-step nilpotent systems. 
\subsection{Nilpotent joint ergodicity criteria}

The failure of the joint ergodicity criteria from \cite{BFM22, Fr21, FrKu22a} in the nilpotent universe prompts the search for an alternative. The following question naturally comes to mind.
\begin{problem}\label{Pr: nilpotent joint ergodicity criteria}
    Let $(X, \CX, \mu, T_1, \ldots, T_\ell)$ be an $s$-step nilpotent system, $H=\langle T_1, \ldots, T_\ell\rangle$, and {$a_1, \ldots, a_\ell:\N\to\Z$ be sequences.}
    Prove or disprove the following. 
    The sequences are jointly ergodic for the system if and only if the following holds:
\begin{enumerate}
    \item (Host-Kra seminorm estimates) There exists a natural number $s'\in\N$ such that for all 1-bounded functions $f_1, \ldots, f_\ell\in L^\infty(\mu)$, we have 
    \begin{align*}
        \E_{n\in\Z}T_1^{a_1(n)}f_1\cdots T_\ell^{a_\ell(n)}f_\ell = 0
    \end{align*}
    whenever $\nnorm{f_j}_{s', T_j} = 0$ for some $1\leq j\leq \ell$.
    \item (Joint ergodicity on a simpler factor) $a_1, \ldots, a_\ell$ are jointly ergodic for the factor $(X, \CZ_{s}(H), \mu, T_1, \ldots, T_\ell)$.
\end{enumerate}
\end{problem}

We emphasize that Problem~\ref{Pr: nilpotent joint ergodicity criteria} provides just one possible formulation of nilpotent joint ergodicity criteria. For instance, we could replace the factor $\CZ_s(H)$ by a factor $\CZ_{s''}(H)$ for some other value $s''=O_s(1)$, {or by the affine factor $\CA_s(H)$.\footnote{We define $\CA_s(H)$ to be the inverse limit of all unipotent affine transformations that arise as factors of $(X, \CX, \mu, T_1, \ldots, T_\ell)$.}} While the proposed statement is a natural one to make, we have little evidence in its favor, other than the fact that it holds for $s=1$.

Problem~\ref{Pr: nilpotent joint ergodicity criteria} remains open even for nilsystems. In this case, its polynomial version is (almost\footnote{If, say, $p_1(n) = n$, then the ergodicity of $(S_1^n)_{n\in\Z}$ is tantamount to the ergodicity of $S_1$, not {total ergodicity}. So Problem~\ref{Pr: nilpotent joint ergodicity criteria on nilsystems} imposes natural but slightly stronger assumptions than Problem~\ref{Pr: nilpotent joint ergodicity criteria}.}) equivalent to the following.
\begin{problem}\label{Pr: nilpotent joint ergodicity criteria on nilsystems}
    Take  totally ergodic nilrotations $S_1, \ldots, S_\ell$ on a $k$-step nilsystem $(Y, \CB_Y, m_Y)$ generating an $s$-step nilpotent group. Let $Y=G/\Gamma$, $(G_i)_{i\in\N_0}$ be the lower central series filtration, and $Y_i :=G/(G_{i+1}\Gamma)$ for every $n\in\N_0$. Show that the polynomials $p_1,\ldots, p_\ell\in\Z[n]$ are jointly ergodic on $Y$ if and only if they are jointly ergodic on $Y_{s}$.
\end{problem}
For polynomial orbits on nilsystems, joint ergodicity amounts to the equidistribution of almost every orbit. Hence Problem~\ref{Pr: nilpotent joint ergodicity criteria on nilsystems} really asks whether the equidistribution of     
\begin{align*}
        (S_1^{p_1(n)}y, \ldots, S_\ell^{p_\ell(n)}y)_{n\in\Z}
    \end{align*}
    on $Y^\ell$ for $m_Y$-a.e. $y\in Y$ can be deduced from an analogous property for the factor nilsystem $Y_s$. By invoking Leibman's equidistribution theorem \cite{L05b}, it suffices to consider unipotent affine transformations on a torus, in which case the problem reduces to a linear algebraic one. Interestingly, although Problem~\ref{Pr: nilpotent joint ergodicity criteria on nilsystems} admits an affirmative answer for $s=1$ \cite{FrKu22a} thanks to the degree lowering argument, we do not even know a direct proof in the simplest nontrivial case of $k=2$.

\subsection{Structure theory}

One of the surprising discoveries in this project is that many elementary properties of box seminorms and factors in the commutative universe seem unlikely to extend to the nilpotent world. We therefore conclude this survey with a bunch of questions on structure theory that have a rather different flavor from the hitherto discussed problems. The first one asks about the higher-step extension of Proposition~\ref{P: comparing HK factors in nilpotent systems}.
\begin{problem}\label{Pr: comparing factors}
    Let $(X, \CX, \mu, T_1, \ldots, T_\ell)$ be a $k$-step nilpotent system with $T_1, \ldots, T_\ell$ ergodic, and let $H := \langle T_1, \ldots, T_\ell\rangle$. Given $s\in\N_0$, does there exist a natural number $s' = O_{k,\ell,s}(1)$ such that $\CZ_s(T_j)\subseteq \CZ_{s'}(H)$? What is the smallest $s'$ we can take? Can we take $s'=s$?
\end{problem}

The next problem inquires about the analog of permutation invariance for nilpotent systems.
\begin{problem}
    Let $(X, \CX, \mu, T, S)$ be a nilpotent system. Can we compare in any way the seminorms $\nnorm{\cdot}_{T,S}$ and $\nnorm{\cdot}_{S,T}$ or the corresponding factors? 
\end{problem}

In the proof of Theorem~\ref{T: nilpotent joint ergodicity}, we relied on Theorem~\ref{T: Candela-Szegedy}, the Candela-Szegedy structure theorem. Given that most of the theory discussed in this paper is known to fail for solvable groups of exponential growth, it is natural to inquire whether this failure extends to structure theory.
\begin{problem}\label{Pr: solvable}
Show the failure of Theorem~\ref{T: Candela-Szegedy} for solvable groups of exponential growth. That is:
\begin{enumerate}
    \item Find a solvable group of exponential growth $H=\langle T,S\rangle$ acting ergodically on $(X, \CX, \mu)$  such that the Host-Kra factors $\CZ_s(H)$ are not $H$-pronilsystems.
    \item Show that for \textit{any} finitely generated solvable group of exponential growth such a system can be found.
\end{enumerate}
\end{problem}
The strategy for Problem~\ref{Pr: solvable} is pretty clear: one should try to adapt the counterexamples of Bergelson and Leibman \cite{BL02, BL04}.

{The last question concerns the missing direction in \eqref{E: seminorms vs. factors nonabelian}.
\begin{problem}\label{Pr: nonobvious direction}
    Does the converse of \eqref{E: seminorms vs. factors nonabelian} hold for arbitrary subgroups of an arbitrary countable amenable group? How about subgroups of a finitely generated nilpotent group.
\end{problem}
Problem~\ref{Pr: nonobvious direction} admits an affirmative answer whenever the ambient group is abelian \cite[Theorem~2.4]{TZ16} or when $H_1 = \cdots = H_s$ is a finitely generated nilpotent group \cite[Theorem~3.21]{CS23}. But the answer remains unclear even for the seminorm $\nnorm{\cdot}_{T,S}$ when $T,S$ generate a 2-step nilpotent group.
}
\appendix

\section{Invariant and Kronecker factors}\label{A: Kronecker}
Let $H$ be a group and $(X, \CX, \mu, U)$ be an $H$-system. 
\begin{definition}[Invariant factor and algebra]\label{D: invariant}
    The \emph{invariant factor} of $U$ is $$\CI(U) := \{A\in \CX:\; U_h\inv A = A\; \textrm{for all}\; h\in H\}.$$
    We also define the corresponding \textit{invariant algebra} to be
    \begin{align*}
        I(U) := \{f\in L^2(\mu)\colon\; U_hf = f\; \textrm{for all}\; h\in H\}.
    \end{align*}
    Whenever we want to emphasize the dependence on the measure, we use the labels $\CI_\mu(U)$ and $I_\mu(U)$.\footnote{A priori, the factor $\CI(U)$ has nothing to do with the measure on the underlying space. However, since we think of all sets and functions up to measure 0, we consider $\CI_\mu(U)$ to be collection of sets that agree with $U$-invariant sets up to $\mu$-null sets.} If 
    we consider sets invariant under the restricted action $U'=(U_h)_{h\in H'}$ for some subgroup $H'\subseteq H$, we use the labels $\CI(H')=\CI(U')$ and $I(H')=I(U')$. 
\end{definition}

It follows easily from the definitions that $I(U) = Z_0(U)$ and $\CI(U) = \CZ_0(U)$.

Likewise, we define the Kronecker factor of $U$ as follows.
\begin{definition}[Eigenfunctions, eigenvalues, and Kronecker factor]
      An \textit{eigenfunction} of $U$ is a function $f\in L^2(\mu)$ satisfying $U_h f = \lambda(h)f$ for some group homomorphism\footnote{Since we do not endow $H$ with any topology, the question of continuity of $\lambda$ does not arise here.} $\lambda\colon H\to\S^1$ (which we call \textit{eigenvalue}). Let $K(U)\subseteq L^2(\mu)$ be the algebra generated by eigenfunctions of $U$ and $\CK(U)$ be the corresponding $\sigma$-algebra. Whenever needed, we record the dependence on $\mu$ with the labels $K_\mu(U)$ and $\CK_\mu(U)$.
\end{definition}

If $H=\Z$, i.e. $U = (T^h)_{h\in\Z}$ for some measure-preserving transformation $T$, then we also write $I(T), \CI(T), K(T), \CK(T)$, etc.

Note that for any eigenfunction $f$ of $U$, we have
\begin{align}\label{E: commutativity on Kronecker}
    U_{hh'}f = \lambda(hh')\cdot f = \lambda(h)\lambda(h')\cdot f = \lambda(h'h)f = U_{h'h}f,
\end{align}
i.e. $U$ acts commutatively on $K(U)$. 

{Let $U,V$ be $H$-actions on $X$. We use $U\times V$ to denote the (diagonal) $H$-action on $X^{2}$ given by $U\times V:=(U_{h}\times V_{h})_{h\in H}$.}
We start with a lemma that relates the Kronecker factor of two actions to the product of their Kronecker factors. 
\begin{lemma}\label{L: product of Kroneckers}
    Let $H$ be a finitely generated group, and let $U,V$ be two $H$-actions on $(X, \CX, \mu)$. Then $K_{\mu_{1}\times\mu_{2}}(U\times V) = K_{\mu_{1}}(U)\times K_{\mu_{2}}(V)$.
\end{lemma}
In fact, it suffices to assume that the abelianization $H/[H,H]$ is finitely generated. However, our applications involve only finitely generated groups, so we restrict to this simpler setting.
\begin{proof}
    It is clear that $K_{\mu_{1}}(U)\times K_{\mu_{2}}(V)\subseteq K_{\mu_{1}\times\mu_{2}}(U\times V)$, so the whole point is to prove the converse direction. 
    We restrict our attention to the space $$K_{\mu_{1}\times\mu_{2}}(U\times V) = L^2(X^2,\CK(U\times V), \mu_{1}\times \mu_{2}, U\times V).$$ 
    By \eqref{E: commutativity on Kronecker}, the $[H,H]$-action $(U_h\times V_h)_{h\in[H,H]}$ on $K_{\mu_{1}\times\mu_{2}}(U\times V)$ is the identity. On replacing $H$ with $H/[H,H]$, we can assume that $H$ is a finitely generated abelian group. By the structure theorem of finitely generated abelian groups, $H$ is a product of some $\Z^\ell$ and finitely many cyclic groups. The action of each cyclic group can be expanded to the action of $\Z$ by periodicity, so without loss of generality we can assume that $H = \Z^\ell$.

    We complete the proof just like in \cite[Chapter 4.4]{Fu81}. Let $(X,\CY_1, \mu_{1}, U)$, $(X,\CY_2, \mu_{2}, V)$ be the marginals of $(X^2,\CK_{\mu_{1}\times\mu_{2}}(U\times V), \mu_{1}\times\mu_{2}, U\times V)$. We split
    \begin{align*}
        L^2(\CY_1, \mu_{1}) = K_{\mu_{1}}(U) \oplus W_1\quad \textrm{and}\quad L^2(\CY_2, \mu_{2}) = K_{\mu_{2}}(V) \oplus W_2,
    \end{align*}
    and then we split $K_{\mu_{1}\times\mu_{2}}(U\times V)$ into four parts accordingly. The goal is to show that of these four subspaces, $U\times V$ only acts nontrivially on $K_{\mu_{1}}(U)\times K_{\mu_{2}}(V)$. Let $f\in L^2(\CX, \mu_{1})$ and $g\in W_2$.
    Then
    \begin{align*}
        \E_{h\in\Z^\ell}\abs{\langle f\otimes g, (U\times V)_h(f\otimes g)\rangle_{\mu_{1}\times\mu_{2}}}^2\leq \norm{f}_{L^2(\mu)}^4\cdot\E_{h\in\Z^\ell}\abs{\langle g, V_h g\rangle_{\mu_{2}}}^2 = 0
    \end{align*}
    since $g\in W_2$.\footnote{We use here the well-known characterization: $g\perp K_{\mu_{2}}(V)$ if and only if $g$ is weakly mixing with respect to $V$, which means that $\E_{h\in\Z^\ell}\abs{\langle g, V_h g\rangle_{\mu_{2}}}^2 = 0.$ This can be proved using either the spectral theorem, as in \cite[Lemma~4.15]{Fu81}, or Hilbert-Schmidt operators, following e.g. \cite[Section 2.12]{TaoPoincare}.} Analogously, for any $f\in W_1$ and $g\in L^2(\CX, {\mu_{2}})$, we have
    \begin{align*}
        \E_{h\in\Z^\ell}\abs{\langle f\otimes g, (U\times V)_h(f\otimes g)\rangle_{\mu_{1}\times\mu_{2}}}^2 = 0.
    \end{align*}
    Hence $U\times V$ indeed vanishes outside $K_{\mu_{1}}(U)\times K_{\mu_{2}}(V)$, implying the claim. 
\end{proof}

The purpose of this appendix is to prove two lemmas that relate the invariant algebra $I_{\mu\times_{\CI(U)}\mu}(U\times U)$ to the invariant algebra of transformations $V_1\times V_2$, where $V_1, V_2$ are conjugates of $U$. These results are used to prove Propositions \ref{P: face transformations invariance} and \ref{P: Z_s(T_1) is T_2-invariant}. They both rely on a similar underlying idea that we isolate in the following proposition.
\begin{proposition}\label{P: equal invariant factors}
    Let $H$ be a finitely generated group and $(X, \CX, \mu)$ be a standard probability space endowed with the following measure-preserving actions:
    \begin{enumerate}
        \item $\Z$-actions $T_j$ for $1\leq j\leq 4$;
        \item $H$-actions $U = (U_h)_{h\in H}$ and $V_j = (T_j\inv U_h T_j)_{h\in H}$ for $1\leq j\leq 4$.
    \end{enumerate}
    Suppose that $I(U) = I(V_1) = \cdots =  I(V_4)$, and let $\mu = \int \mu_x\; d\mu(x)$ be the corresponding disintegration. Lastly, suppose that for $\mu$-a.e. $x\in X$:
    \begin{enumerate}
        \item $K_{\mu_x}(V_1) = K_{\mu_x}(V_3)$ and $K_{\mu_x}(V_2) = K_{\mu_x}(V_4)$;
        \item for every $f\in K_{\mu_x}(V_1)$ and $g\in K_{\mu_x}(V_2)$, we have
        \begin{align}\label{E: same eigenvalues}
         V_1 f\cdot \bar f = V_2 g\cdot \bar g\quad \iff\quad V_3 f\cdot \bar f = V_4 g\cdot \bar g .  
        \end{align}
    \end{enumerate}
    Then
    \begin{align*}
        I_{\nu}(V_1\times V_2) = I_\nu(V_3\times V_4),
    \end{align*}
    where $\nu:=\mu_{1,U} = \mu\times_{\CI(U)}\mu$.
\end{proposition}
The assumptions in Proposition~\ref{P: equal invariant factors} look a bit artificial, but they are precisely what we use in two corollaries of Proposition~\ref{P: equal invariant factors} that we shall need.
\begin{proof}
    Let $F\in I_{\nu}(V_1\times V_2)$. By the definition of relatively independent joinings, we have $\nu = \int \mu_x \times \mu_x\; d\mu(x)$.
    Then $F\in I_{\mu_x \times \mu_x}(V_1\times V_2)$ for $\mu$-a.e. $x\in X$ since
    \begin{align*}
        0 = \norm{F-(V_1\times V_2)F}_{L^2(\nu)}^2 = \int \norm{F-(V_1\times V_2)F}_{L^2(\mu_x\times \mu_x)}^2 \,d\mu(x).
    \end{align*}
    Observe that
    \begin{align*}
        I_{\mu_x \times \mu_x}(V_1\times V_2) \subseteq K_{\mu_x\times \mu_x}(V_1\times V_2)= K_{\mu_x}(V_1)\times K_{\mu_x}(V_2),
    \end{align*}
    where the first inclusion is trivial and the second follows from Lemma~\ref{L: product of Kroneckers}.

    Therefore, for $\mu$-a.e. $x\in X$, we can decompose
    \begin{align*}
        F = \sum_i c_{i,x} f_{i,x}\otimes \bar g_{i,x}
    \end{align*}
    in {$L^2(\mu_x\times \mu_{x})$}, where $\sum_i|c_{i,x}|^2<\infty$ and $f_{i,x}\in K_{\mu_x}(V_1)$, $g_{i,x}\in K_{\mu_x}(V_2)$ are modulus-1 eigenfunctions of corresponding transformations. 
    Decomposing $f_{i,x}$ and $\bar g_{i,x}$ further if necessary, we can assume that $(f_{i,x}\otimes \bar g_{i,x})_i$ is an orthonormal subset of $L^2(\mu_x\times \mu_{x})$. Then $F\in I_{\mu_x \times \mu_x}(V_1\times V_2)$ implies $f_{i,x}\otimes \bar g_{i,x}\in I_{\mu_x \times \mu_x}(V_1\times V_2)$ for every $i$, and so the $V_1$-eigenvalue of $f_{i,x}$ and the $V_2$-eigenvalue of $g_{i,x}$ are the same. We call it $\lambda_{i,x}$. By \eqref{E: same eigenvalues} {and condition (i) before \eqref{E: same eigenvalues}}, we also get that $f_{i,x}, g_{i,x}$ are eigenfunctions of $V_3$ and $V_4$ respectively with the same eigenvalue. It follows that $F\in I_{\mu_x\times \mu_x}(V_3\times V_4)$ for $\mu$-a.e. $x\in X$, and hence also $F\in I_{\nu}(V_3\times V_4)$. The other inclusion follows by symmetry.
\end{proof}

{We now present two corollaries of Proposition~\ref{P: equal invariant factors}.}
The first corollary is needed in the proof of Proposition~\ref{P: face transformations invariance}.
\begin{corollary}\label{C: equal invariant factors 2}
        Let $H$ be a finitely generated group and $(X,\CX, \mu, U)$ be an $H$-system. Let $U'_h:=U_{aha\inv}$ and $U''_h:=U_{bhb\inv}$ for some $a,b\in H$ and $\nu:=\mu_{1,U} = \mu\times_{\CI(U)}\mu$. Then $$I_\nu(U'\times U'') = I_\nu(U\times U).$$
\end{corollary}
\begin{proof}
    Since $U'$ and $U''$ are conjugates of $U$ by $U_a, U_b$ respectively, clearly 
    \begin{align*}
        I(U) = I(U')=I(U'')\quad \textrm{and}\quad K_{\mu_x}(U) = K_{\mu_x}(U') = K_{\mu_x}(U'')
    \end{align*}
    by (\ref{E: commutativity on Kronecker}),
    and moreover
    \begin{align}\label{eerr}
    U_h f = \lambda(h)f  \quad \iff \quad U'_h f = \lambda(h)f \quad \iff \quad U''_h f = \lambda(h)f.
    \end{align}
    The assumptions of Proposition~\ref{P: equal invariant factors} are therefore satisfied, and the result follows.
\end{proof}

The second corollary shows up in the proof of Propositions \ref{P: face transformations invariance} and \ref{P: Z_s(T_1) is T_2-invariant}.
\begin{corollary}\label{C: equal invariant factors 3}
    Let $H$ be a finitely generated group, $U:=(U_h)_{h\in H}$ be an $H$-action, $T$ be a $\Z$-action, and $V := (T\inv U_h T)_{h\in H}$ be an $H$-action on a standard probability space $(X, \CX, \mu)$. Suppose that:
    \begin{enumerate}
        \item the joint action of $U,T$ is 2-step nilpotent;
        \item $I_\mu(U)=I_\mu(V)$.
    \end{enumerate}
     Let $\nu := \mu_{1,U}=\mu\times_{\CI_{\mu}(U)}\mu.$ 
    Then $I_\nu(U\times U) = I_\nu(V\times V)$. 
\end{corollary}
\begin{proof}
Again, the proof amounts to verifying that the assumptions of Proposition~\ref{P: equal invariant factors} are satisfied. Let $\mu = \int \mu_x\; d\mu(x)$ be the disintegration of $\mu$ over $\CI_\mu(U) = \CI_\mu(V)$. Since $\mu_x$ is ergodic with respect to $U, V$ for $\mu$-a.e. $x\in X$, each eigenspace has dimension 1, and so \eqref{E: same eigenvalues} follows trivially from the fact that $f,g$ satisfying either side of \eqref{E: same eigenvalues} have to be constant multiples of each other. So all that remains to be shown is that $K_{\mu_x}(V) = K_{\mu_x}(U)$ for $\mu$-a.e. $x\in X$.

Suppose that $f\in K_{\mu_x}(U)$ has $U$-eigenvalue $\lambda$. 
    Then for any $h,h'\in H$,
    \begin{align*}
        U_h(V_{h'} f) &= V_{h'}U_hf= U_{h'}[U_{h'},T]U_h f = U_h U_{h'}[U_{h'},T] f\\
        &= \lambda(h)\cdot U_{h'}[U_{h'},T]f = \lambda(h)\cdot V_{h'}f
    \end{align*}
    (all the identities in this paragraph are in $L^2(\mu_x)$).
In the second equality, we use the identity $V_{h'}= U_{h'}[U_{h'},T]$. In the third equality, we commute $U_h$ and $[U_{h'},T]$ because of the assumption that $U,T$ generate a 2-step nilpotent group, and we commute $U_h$ with $U_{h'}$ because of  \eqref{E: commutativity on Kronecker}. It follows that 
\begin{align*}
    U_h(V_{h'}f\cdot \bar f) = \lambda(h)\cdot V_{h'}f\cdot \bar\lambda(h)\cdot \bar f = V_{h'}f\cdot \bar f,
\end{align*}
and so $V_{h'}f\cdot f$ is $U$-invariant for every $h'\in H$. By the ergodicity of $\mu_x$ with respect to $U$, we get that $V_{h'}f = \eta(h')f$ for some $\eta(h')\in\S^1$. The fact that $V$ is a group action forces $\eta:H\to\S^1$ to be a group homomorphism, and so $f\in K_{\mu_x}(V)$. The other direction follows by symmetry.
\end{proof}

\section{$H$-nilsystems and their basic properties}\label{A: nilsystems}
This appendix contains the definition and main properties of $H$-nilsystems and $H$-pronilsystems that show up in the Candela-Szegedy structure theorem (Theorem~\ref{T: Candela-Szegedy}).
\begin{definition}[$H$-nilsystems]\label{D: nilsystems}
    Let $H$ be a finitely generated group. Then an \textit{$s$-step topological $H$-nilsystem} is a topological $H$-system $(Y, U)$, where $Y$ is an $s$-step nilmanifold $Y$ and
    $H$ acts on $Y$ via nilrotations. That is, $Y=G/\Gamma$ for an $s$-step nilpotent Lie group $G$ and a cocompact lattice $\Gamma\subseteq G$, and for every $y\in Y$, we have $U_h y = b_h y$ for some elements $b_h\in G$ satisfying $b_hb_{h'} = b_{hh'}$ for every $h,h'\in H$.
    
    If $\CB_Y$ is the Borel $\sigma$-algebra and $m_Y$ is the Haar measure on $Y$, then $(Y, \CB_Y, m_Y, U)$ is an \textit{$s$-step $H$-nilsystem}. An inverse limit of $s$-step $H$-nilsystems is called an \textit{$s$-step $H$-pronilsystem.}
\end{definition}
As is standard, we assume without loss of generality that 
the group $G$ in the definition of $Y$ is generated by its connected component $G^o$ and the elements $(b_h)_{h\in H}$.

Note that 
we can always identify $H$ with a subgroup of $G$, and we do so whenever convenient. In particular, if $H$ acts on an $s$-step nilmanifold via nilrotations, then we can take it to be nilpotent of step at most $s$.\footnote{More precisely, if $H_\bullet = (H_i)_{i\in\N_0}$ is the lower central series, then the action of $H$ is $H_{s+1}$-invariant, and hence we can replace $H$ by $H/H_{s+1}$ and embed the latter into $G$.} 

The result below shows that ergodicity of $H$-actions on the nilmanifold is equivalent to their ergodicity on the maximal factor torus. 

\begin{proposition}[{\cite[Theorem~2.17]{L05b}}]\label{P: ergodicity of nilsystems}
    Let $Y = G/\Gamma$ be a nilmanifold endowed with the lower central series filtration $G_\bullet=(G_i)_{i\in\N_0}$. Let $Y':=G/(G_{2}\Gamma)$. 
    Then an $H$-nilsystem $(Y, \CB_Y, m_Y, U)$ is ergodic if and only if $(Y', \CB_{Y'}, m_{Y'}, U')$ is, where $U'$ is the projection of $U$ onto $Y$.
\end{proposition}

Furthermore, the classic equivalence of ergodicity, unique ergodicity, minimality and transitivity of $\Z$-nilsystems extends to $H$-nilsystems.
\begin{proposition}\label{P: equivalence of ergodicity on nil}
    Let $(Y, U)$ be a topological $H$-nilsystem. Then the following properties are equivalent: (i) transitivity, (ii) minimality, (iii) unique ergodicity, (iv) ergodicity of the Haar measure $m_Y$. Moreover, these properties hold if and only if they hold for $(Y', U')$ defined in Proposition~\ref{P: ergodicity of nilsystems}.
\end{proposition}
\begin{proof}
    The property (iv) is equivalent to (iii) by \cite[Theorem~2.19]{L05b} and to (ii) by \cite[Corollary 2.20]{L05b}. The properties (i) and (ii) are then equivalent by distality of $H$-nilsystems. Lastly the equivalence of properties for $Y$ and $Y'$ follows from Proposition~\ref{P: ergodicity of nilsystems}.
\end{proof}

\section{Cubic structures}\label{A: cubes}
This section summarizes the material on cubic structures in various settings: algebraic, topological, and measure-preserving, with particular focus on the interplay of these a priori different notions for nilsystems. Most of the material is a straightforward adaptation of arguments from \cite{HK18}; the main difference is that we are interested in actions by more general groups than $\Z$. The results in this section could likely be derived from more general results in the nilspace theory (see e.g. \cite{Can17a, Can17b,CGSS20}). However, given that we work within the relatively simpler setting of nilsystems, we instead opted to present a self-contained summary of all the definitions and results that we actually use.

\subsection{Algebraic cubes}
We start by presenting cubic structures on groups. 
\begin{definition}[Faces of a cube]
    A \textit{face} of $\cube{k}:=\{0,1\}^k$ of \textit{codimension} $1\leq r\leq k$ is a set
    \begin{align*}
        \alpha := \{\eps\in\cube{k}\colon \eps_{i_1} = a_{i_1}, \ldots \eps_{i_r}=a_{i_r}\}
    \end{align*}
    for some $r\in\N_0$, $1\leq i_1 < \cdots < i_r\leq k$ and $a_{i_1}, \ldots, a_{i_r}\in\{0,1\}$. Note that $\alpha=\cube{k}$ is the unique face of codimension 0. {We call a face \textit{upper} if it contains $(1,\ldots, 1)$, or equivalently if the elements $a_{i_1}, \ldots, a_{i_r}$ defining it are all 1.}
\end{definition}

\begin{definition}[Cubes of groups]
    Let $G$ be a group. For $k\in\N_0$, denote $G^{\cube{k}} := G^{\{0,1\}^k}$. Given a face $\alpha$ of $\cube{k}$, let $g^{(\alpha)} \in G^{\cube{k}}$ be defined by
    \begin{align*}
        (g^{(\alpha)})_\eps = \begin{cases}
            g,\; &\eps\in \alpha,\\
            e_G,\; &\eps\notin\alpha.
        \end{cases}
    \end{align*}
\end{definition}

\begin{definition}[Cubic groups]
    Let $k\in\N_0$, and let $G$ be a group. We then define
    \begin{align*}
        Q^\cube{k}(G) := &\langle g^{(\alpha)}\in G^\cube{k}\colon\; g\in G,\; \codim \alpha = 1\rangle
    \end{align*}
    to be the \textit{$k$-dimensional cube of $G$}.
\end{definition}
The next two lemmas provide an alternative description of the cubic group.
\begin{lemma}[{\cite[Chapter 6, Lemma~14]{HK18}}]\label{L: alternative cube description}
Let $k\in\N_0$, and let $G$ be a group endowed with the lower central series filtration $G_\bullet = (G_i)_{i\in\N_0}$. Then
    \begin{align*}
        Q^\cube{k}(G)=\langle g^{(\alpha)}\in G^\cube{k}\colon\; g\in G_{i},\; \codim \alpha = i,\; 0\leq i\leq k\rangle.
    \end{align*}
\end{lemma}

In fact, the cubic group can be generated using only upper faces. 
\begin{lemma}[Upper face decomposition {\cite[Chapter 6, Proposition~17]{HK18}}]\label{L: upper face}
    Let $k\in\N_0$, and let $G$ be a group endowed with the lower central series filtration $G_\bullet = (G_i)_{i\in\N_0}$. Then there exists an ordering $\alpha_1, \ldots,\alpha_\ell$ of the upper faces of $\cube{k}$ satisfying
    \begin{align*}
        {\codim \alpha_1 \leq \cdots \leq \codim \alpha_\ell}
    \end{align*}
    such that every element of $Q^{\cube{k}}(G)$ can be uniquely written as $g_1^{(\alpha_1)}\cdots g_\ell^{(\alpha_\ell)}$
    with $g_i\in G_{\codim \alpha_i}$ for all $1\leq i\leq \ell$.
\end{lemma}

We shall use the following corollary.
\begin{corollary}\label{C: face}
    Let $k\in\N_0$, and let $G$ be a group endowed with the lower central series filtration $G_\bullet = (G_i)_{i\in\N_0}$. If $g^{(\alpha)}\in Q^{\cube{k}}(G)$ for some codimension-$i$ face $\alpha$, then $g\in G_i.$
\end{corollary}
\begin{proof}
The proof amounts to showing that in the upper case decomposition of $g^{(\alpha)}$, all the coefficients corresponding to faces of codimension greater than $i$ are trivial while all the coefficients of codimension up to $i$ are $g, g\inv$, or $e_G$. The claim then follows from the uniqueness part of Lemma~\ref{L: upper face}. For instance, if $k=2$ and $\alpha = \{(0,0)\}$, then
  Lemma~\ref{L: upper face} allows us to write 
\begin{align}\label{E: lower-upper face2}
    (g,e_{G},e_{G},e_{G})=(g_{0},g_{0},g_{0},g_{0})(e_{G},g_{1},e_{G},g_{1})(e_{G},e_{G},g_{2},g_{2})(e_{G},e_{G},e_{G},g_{3})
\end{align}
for some $g_{0},g_{1},g_{2}\in G$ and $g_{3}\in G_{2}$.
Then 
$$g=g_{0} \quad \textrm{and}\quad g_{0}g_{1}=g_{0}g_{2}=g_{0}g_{1}g_{2}g_{3}=e_{G},$$
which forces $g_{0}=g$, $g_{1}=g_{2}=g^{-1}$ and $g_{3}=g$. In particular, the latter gives $g_3\in G_2$.

We move on to the general case. Given $J\subseteq[k]$, define the upper face
\begin{align*}
    \beta_J:= \rem{\eps\in\cube{k}\colon\; \eps_{j}=1\; \textrm{for}\; j\in J}.
\end{align*}
Note that $\codim \beta_J = |J|$.
Let
\begin{align*}
g^{(\alpha)}=\prod_{J\subseteq [k]}g_{J}^{(\beta_{J})}    
\end{align*}
 be the upper face decomposition provided by Lemma~\ref{L: upper face}; then $g_{J}\in G_{|J|}$ for every $J\subseteq[k]$. In particular, observe that
\begin{align}\label{E: ufd}
    (g^{(\alpha)})_\eps=\prod_{\substack{J\subseteq[k]\colon\;\eps\in \beta_J}}(g_{J}^{(\beta_{J})})_\eps = \prod_{\substack{J\subseteq\supp(\eps)}}(g_{J}^{(\beta_{J})})_\eps
\end{align}
for all $\e\in\cube{k}$,
where $\supp(\eps):=\{j\in[k]\colon\; \eps_j = 1\}$.

Let $I,I'\subseteq [k]$ be the sets of coordinates such that 
    \begin{align*}
    \alpha = \rem{\eps\in\cube{k}\colon\; \eps_{j} =1\; \textrm{for}\; j\in I\; \textrm{and}\; \eps_{j'}=0\; \textrm{for}\; j'\in I'}.
\end{align*}
We claim that
\begin{align*}
    g_J =\begin{cases}
        g,\; &I\subseteq J\subseteq I\cup I',\; |J\backslash I|\; \textrm{even},\\
        g\inv,\; &I\subseteq J\subseteq I\cup I',\; |J\backslash I|\; \textrm{odd},\\
        e_G,\; &\textrm{otherwise}.
    \end{cases}
\end{align*}
In particular $g_{I\cup I'}$ is equal to $g$ or $g^{-1}$. Since $g_{I\cup I'}\in G_{\vert I\cup I'\vert}=G_{i}$, this forces $g\in G_{i}$ and we are done.

We shall prove the claim on the structure of $g_J$'s by inducting on $|J|$ and comparing the $\eps$-coordinates of both sides of \eqref{E: ufd} for $\supp(\eps) = J$. One should think of this coordinate as the ``lower corner'' of the face $\beta_J$, and its key property is that it sees nontrivial contributions only from $J$ and upper faces of codimension strictly less than $|J|$. 

We start with $|J| = 0$, i.e. $J =\emptyset$. By analyzing the coordinates of \eqref{E: ufd} corresponding to $\eps = (0, \ldots, 0)$, we deduce that $(g^{(\alpha)})_\eps = g_\emptyset$, and this is $g$ if $\alpha = \cube{k}$ (i.e. $I=I'=\emptyset$) and $e_G$ otherwise. Hence $g_\emptyset$ takes the claimed form.

Suppose now that the claim holds for all $|J|< l$ for some $l\in[k]$, and consider the $\eps$-coordinate of \eqref{E: ufd} for $J = \supp(\eps)$. If $I\not\subseteq J$, then $\eps\notin \alpha$, and so $(g^{(\alpha)})_\eps = e_G$. In this case, we also have $I\not\subseteq J'$ for all $J'\subseteq J$. By induction, we get that $g_{J'} = e_G$ for all $J'\subsetneq J$, and hence \eqref{E: ufd} forces $g_J = e_G$ as well. So, the only nontrivial $g_J$'s are those for which $I\subseteq J$.

Assume therefore that $I\subseteq J$. If $\eps\in \alpha$, then  $(g^{(\alpha)})_\eps = g$. At the same time, this assumption also implies that $I'\cap J = \emptyset$. Now, two things may happen. Either $I=J$, and then it follows from \eqref{E: ufd} that $g_J = g_I={g}$  {since $g_J^{(\beta_J)}$ is the only term in \eqref{E: ufd} with a possibly nontrivial $\eps$-coordinate.} This satisfies the claim as $|J\backslash I| = 0$.
Or $I\subsetneq J$; then consider a subset $I\subseteq J'\subsetneq J$. {Since $I'\cap J' = \emptyset$,} by the induction hypothesis, for such $J'$ we have $g_{J'} = g$ if $J' = I$ and $g_{J'} = e_G$ otherwise. Hence it follows from \eqref{E: ufd} that $g = (g^{(\alpha)})_\eps = g_I g_J = g g_J$, implying that $g_J = e_G$, which is consistent with the claim as then $J\subsetneq I\cup I'$.

The remaining case to analyse is when $\eps\notin\alpha$; then $(g^{(\alpha)})_\eps = e_G$ and $I'\cap J \neq \emptyset$. Let $J^*:=J\cap(I\cup I')$, and set $m:=|J^*\backslash I|$. If $J^*\subsetneq J$, i.e. $J\subsetneq I\cup I'$, then \eqref{E: ufd}, the inductive hypothesis, and the binomial theorem give
\begin{align*}
    e_G = \brac{\prod_{I\subseteq J'\subseteq J^*} g_{J'}}g_J = g^{\sum\limits_{j=0}^m (-1)^j \binom{m}{j}}g_{J} = g_J;
\end{align*}
above, $\binom{m}{j}$ is the number of subsets $I\subseteq J'\subseteq J^*$ with $|J'\backslash I| = j$. If $J = J^*$, i.e. $J\subseteq I\cup I'$, then \eqref{E: ufd} and the inductive hypothesis give
\begin{align*}
    e_G = \prod_{I\subseteq J'\subseteq J} g_{J'} = g^{\sum\limits_{j=0}^{m-1} (-1)^j \binom{m}{j}}g_{J} = g^{-(-1)^m} g_J,
\end{align*}
and so $g_J = g^{(-1)^{m}}$. In either case, we get the claim.
\end{proof}

The cubic groups can be described inductively as fiber products. We recall that a \textit{fiber product} of a group $G$ over its subgroup $H$ is the group
\begin{align*}
    G\times_H G :=\langle(g,g), (h, e_G)\colon\; g\in G,\; h\in H\rangle.
\end{align*}
\begin{lemma}[Fiber product structure]\label{L: fiber products}
    Let $k\in\N_0$, and let $G$ be a group endowed with the lower central series filtration $G_\bullet = (G_i)_{i\in\N_0}$. Then
    \begin{align*}
        Q^\cube{k+1}(G) = Q^\cube{k}(G)\times_{L^\cube{k}(G)}Q^\cube{k}(G)
    \end{align*}
    for
    \begin{align*}
        L^\cube{k}(G) &:=\langle g^{(\alpha)}\in G^\cube{k}\colon\; g\in G_{i+1},\; \codim \alpha = i,\; 0\leq i\leq k\rangle.
    \end{align*}
\end{lemma}
\begin{proof}
    We show that both groups have the same generators. Observe first that every face in $\cube{k+1}$ is of the form\footnote{If $\alpha= \{\eps\in\cube{k}\colon \eps_{i_1} = a_{i_1}, \ldots \eps_{i_r}=a_{i_r}\}$, then $\alpha1:= \{\eps\in\cube{k+1}\colon \eps_{i_1} = a_{i_1}, \ldots \eps_{i_r}=a_{i_r}, \eps_{k+1} = 1\}$, likewise for $\alpha 0$.} $\alpha 1$ or $\alpha 0$ for some face $\alpha\in\cube{k}$. Suppose that $\codim \alpha = i$. Then $g^{(\alpha 1)} = (g^{(\alpha)}, g^{(\alpha)})$, and it lies in $Q^\cube{k+1}(G)$ iff $g\in G_{i}$. Indeed, {the forward direction follows from Corollary \ref{C: face} while the converse is a consequence of Lemma~\ref{L: alternative cube description}}. From this and the definition of fiber products it is easy to conclude that $g^{(\alpha 1)}\in Q^\cube{k+1}(G)$ iff it lies in $Q^\cube{k}(G)\times_{L^\cube{k}(G)}Q^\cube{k}(G)$. 
    Likewise, $g^{(\alpha 0)} = (g^{(\alpha)},e_G)$ is in $Q^\cube{k+1}(G)$ iff $g\in G_{i+1}$ (again by Lemma~\ref{L: alternative cube description} and Corollary \ref{C: face} since $\codim \alpha 0=i+1$) iff it lies in $Q^\cube{k}(G)\times_{L^\cube{k}(G)}Q^\cube{k}(G)$.
\end{proof}
 
\subsection{Dynamical cubes} We move on to define cubes in two dynamical settings: topological and measure-preserving.
The definition below holds for group actions in any category; in particular, it applies equally well to topological and measure-preserving actions.
 \begin{definition}[Face transformations]
     Let $H$ be a group acting on a set $X$ via the action $U=(U_h)_{h\in H}$. For every face $\alpha\in\cube{k}$, we define the \textit{face transformation} $U^{(\alpha)}$ on $X^\cube{k}:=X^{\{0,1\}^k}$ via
    \begin{align*}
        (U^{(\alpha)})_\eps x := \begin{cases} Ux,\; &\eps\in\alpha\\
        x,\; &\eps\notin\alpha.
        \end{cases}
    \end{align*}
    Then 
    $$Q^\cube{k}(U):=\langle U^{(\alpha)}\colon\; \alpha\in\cube{k}\rangle$$
    defines a $Q^\cube{k}(H)$-action on $X^\cube{k}$.
 \end{definition}

In the topological setting, we define the cubic structure as follows. 
\begin{definition}[Topological cubes]
    Let $H$ be a group and $(X, U)$ be a minimal topological $H$-system. 
    Then its \textit{$k$-dimensional topological cube} is the topological $Q^\cube{k}(H)$-system $Q^\cube{k}(X, U)$ given by the action of $Q^\cube{k}(U)$ on the closure of the orbit $$\rem{\widetilde Ux^\cube{k}\colon\; \widetilde U\in Q^\cube{k}(U)},$$ where $x^\cube{k}:=(x)^{\{0,1\}^k}$ for any point $x\in X$.
\end{definition}
Because of minimality, the definition of $Q^\cube{k}(X, U)$ does not depend on the choice of the point $x\in X$.

\begin{proposition}[Invariance under the group of face transformations]\label{P: face transformations invariance}
    Let $H$ be a  finitely generated amenable group and $(X, \CX, \mu, U)$ be an $H$-system. Then the measure $\mu_{k,H}$ is $Q^\cube{k}(U)$-invariant for any $k\in\N_0$.
\end{proposition}

\begin{proof}
    For $k=0$, the statement reduces to the $U$-invariance of $\mu$, and we assume the statement to hold up to some fixed $k\in\N$. Then $Q^\cube{k+1}(U)$ is generated by $U^{(\alpha)}$ for codimension-1 faces $\alpha\in\cube{k+1}$. Suppose first that $\alpha = \beta0$ for the unique codimension-0 face $\beta\in\cube{k}$. Then $U^{(\alpha)} = U^\cube{k}\times \Id_{X^\cube{k}}$, and the definition of $\mu_{k+1,H}$ as a relatively independent product implies that 
    \begin{align*}
        \int U_h^{(\alpha)}(f_1 \otimes f_2)\; d\mu_{k+1,H} &= \int \E(U_h^\cube{k}f_1|\CI(U^\cube{k}))\E(f_2|\CI(U^\cube{k}))\;d\mu_{k,H}\\
        &= \int \E(f_1|\CI(U^\cube{k}))\E(f_2|\CI(U^\cube{k}))\;d\mu_{k,H} = \int f_1\otimes f_2\; d\mu_{k+1,H}
    \end{align*}
    for any $f_1,f_2\in L^\infty(\mu_{k,H})$ and $h\in H$, giving the $U^{(\alpha)}$-invariance.

    Suppose now that $\alpha = \beta1$ for some codimension-1 face $\beta\in\cube{k}$. Then $U^{(\alpha)} = U^{(\beta)}\times U^{(\beta)}$, and so
    \begin{align*}
        \int U_h^{(\alpha)}(f_1 \otimes f_2)\; d\mu_{k+1,H} &= \int \E(U_h^{(\beta)}f_1|\CI(U^\cube{k}))\E(U_h^{(\beta)}f_2|\CI(U^\cube{k}))\;d\mu_{k,H}.
    \end{align*}
    To show $U^{(\alpha)}$-invariance, it thus suffices to establish that 
    \begin{align}\label{E: invariance identity}
        \E(U_h^{(\beta)}f|\CI(U^\cube{k})) =\E(f|\CI(U^\cube{k}))
    \end{align}
    for any $f\in L^\infty(\mu_{k,H}))$ and $h\in H$. From the mean ergodic theorem for amenable groups, we get that
    \begin{align*}
        \E(U_h^{(\beta)}f|\CI(U^\cube{k})) = \E_{h'\in H}U^\cube{k}_{h'}(U^{(\beta)}_h f) = U_h^{(\beta)}\brac{\E_{h'\in H} U_h^{(\beta)}U_{h'}^\cube{k}U_{h\inv}^{(\beta)}f}.
    \end{align*}
    Since $\beta$ has codimension 1, there exists a unique codimension-1 face $\gamma\in \cube{k}$ such that $\cube{k}$ is a disjoint union of $\beta$ and $\gamma$. By an $L^2(\mu_{k,H})$-approximation argument, we can assume that $f = f_\beta \otimes f_\gamma$ with $(f_\beta)_\eps = (f_\gamma)_{\eps'} = 1$ if $\eps\notin\beta$ and $\eps'\notin\gamma$. Then
    \begin{align*}
        \E_{h'\in H} U_h^{(\beta)}U_{h'}^\cube{k}U_{h\inv}^{(\beta)}f = \E_{h'\in H}(U_{hh'h\inv}^{(\beta)}f_\beta \otimes U_{h'}^{(\gamma)}f_\gamma) = \E(f_\beta\otimes f_\gamma|\CI(U'^{(\beta)} \times U^{(\gamma)})),
    \end{align*}
    where $U'_h := U_{hh'h\inv}$, {and the invariant factor $\CI(U'^{(\beta)} \times U^{(\gamma)})$ is taken with respect to $\mu_{1,H} = \mu_{1,U} =\mu_{1,U'}$.}
    By Corollary \ref{C: equal invariant factors 2}, this is the same as 
    \begin{align*}
        \E(f_\beta\otimes f_\gamma|\CI(U^{(\beta)} \times U^{(\gamma)})) = \E_{h'\in H} U_{h'}^\cube{k}f =  \E(f|\CI(U^\cube{k})).
    \end{align*}
    This establishes \eqref{E: invariance identity}, from which the $U^{(\alpha)}$-invariance of $\mu_{k+1,H}$ follows.
\end{proof}

Proposition~\ref{P: face transformations invariance} allows us to define the measure-preserving system of $k$-dimensional cubes. 
\begin{definition}[Measurable cubes]
    Let $H$ be a  finitely generated amenable group and $(X, \CX, \mu, U)$ be an $H$-system. Then its \textit{$k$-dimensional (measurable) cube} is the $Q^\cube{k}(H)$-system $(X^\cube{k}, \CX^\cube{k}, \mu_{k,H}, Q^\cube{k}(U))$.
\end{definition}

\subsection{Cubic structure on nilmanifolds}
We conclude Appendix \ref{A: cubes} by examining the interplay of various cubic structures for $H$-nilsystems, defined in Appendix \ref{A: nilsystems}. Recall that the definition of $H$-nilsystems presumes that $H$ is finitely generated. We begin with defining cubes on nilmanifolds (without yet any dynamics).
\begin{definition}[Cubes of nilmanifolds]\label{D: cubes of nilmanifolds}
    Let $Y=G/\Gamma$ be a nilmanifold. 
    We then let
     \begin{align*}
    Q^{\cube{k}}(Y):=Q^{\cube{k}}(G)/Q^{\cube{k}}(\Gamma)    
    \end{align*}
     be the \textit{$k$-dimensional cube of $Y$}. We denote the Borel $\sigma$-algebra on $Q^\cube{k}(Y)$ via $Q^\cube{k}(\CB_Y)$
     and the Haar measure via $m_Y^{\cube{k}}$.  
\end{definition}

Following a remark in \cite{HK18}, we note that $Q^\cube{k}(Y)$ depends not only on $Y$ but also on the representation $Y=G/\Gamma$.

By \cite[Chapter 6, Proposition~19]{HK18}, $Q^{\cube{k}}(\Gamma) = Q^{\cube{k}}(G)\cap \Gamma^\cube{k}$, and  by \cite[Chapter 12, Theorem~2]{HK18}, the quotient $Q^{\cube{k}}(Y)$ is a subnilmanifold of $Y^\cube{k}$.
Moreover, it is invariant under the $Q^\cube{k}(H)$-action of $Q^{\cube{k}}(U)$, turning $(Q^\cube{k}(Y),Q^\cube{k}(U))$ into a topological $Q^\cube{k}(H)$-nilsystem. The following result shows that $(Q^\cube{k}(Y), Q^\cube{k}(\CB_Y), m_Y^\cube{k}, Q^\cube{k}(U))$ inherits ergodicity (and hence also transitivity, minimality, and unique ergodicity via Proposition~\ref{P: equivalence of ergodicity on nil}) from $(Y, \CB_Y, m_Y, U)$.

\begin{proposition}\label{P: ergodicity on cubes}
    Let $(Y, \CB_Y, m_Y, U)$ be an ergodic $H$-nilsystem. Then the $Q^\cube{k}(H)$-nilsystem  
    $(Q^\cube{k}(Y), Q^\cube{k}(\CB_Y), m_Y^\cube{k}, Q^\cube{k}(U))$    is also ergodic.
\end{proposition}
\begin{proof}
    The proof is essentially the same as that of \cite[Chapter 12, Theorem~3]{HK18}. Let $Y=G/\Gamma$ be the representation with respect to which we define $Q^\cube{k}(Y)$. By Proposition~\ref{P: ergodicity of nilsystems}, it suffices to verify the ergodicity of the induced action of $Q^\cube{k}(U)$ on
    \begin{align*}
        \frac{Q^\cube{k}(G)}{(Q^\cube{k}(G))_2\cdot Q^\cube{k}(\Gamma)}.
    \end{align*}
    As shown in the proof of \cite[Chapter 12, Theorem~3]{HK18}, this nilmanifold can in fact be identified with the compact abelian group $Q^\cube{k}(Y')$, where $Y':=G/(G_2\Gamma)$.  The group $Q^\cube{k}(Y')$ can then be parametrized by $(Y')^{k+1}$ via the continuous group isomorphism $\iota: (Y')^{k+1}\to Q^\cube{k}(Y')$ given by $$\iota(y_0, \ldots, y_k)=(y_0 + \eps_1y_1 + \cdots + \eps_k y_k)_{\eps\in\cube{k}}.$$

    Let $H':=H/H_2$ and $U'$ be the induced action of $H'$ on $Y'$. Then $Q^\cube{k}(U')$ is the induced  $Q^\cube{k}(H')$-action on $Y'$.
    Note that $Q^\cube{k}(H')\cong {H'}^{k+1}$, as there are $k+1$ faces of codimension $\leq 1$. We then define the ${H'}^{k+1}$-action $V\colon(Y')^{k+1}\to (Y')^{k+1}$ via $$V_{h_0, \ldots, h_k}(y_0, \ldots, y_k) := (U'_{h_0}y_0 , \ldots, U'_{h_k}y_k).$$ Since $Y'$ is abelian, $U'$ acts on $Y'$ via ergodic rotation, and so does $V$ on $(Y')^{k+1}$. The conclusion that $Q^\cube{k}(U')$ is ergodic on $Q^\cube{k}(Y')$ follows since $\iota$ intertwines $V$ with $Q^\cube{k}(U')$.
\end{proof}

As a consequence of Propositions \ref{P: equivalence of ergodicity on nil} and \ref{P: ergodicity on cubes}, we get the following immediate corollary, which is a direct extension of \cite[Chapter 12, Theorem~3]{HK18}.
\begin{corollary}\label{C: unique ergodicity of cubes}
    Let $(Y, \CB_Y, m_Y, U)$ be an ergodic $H$-nilsystem. Then the topological $Q^\cube{k}(H)$-nilsystem  
    $(Q^\cube{k}(Y), Q^\cube{k}(U))$ is uniquely ergodic, with $m_Y^\cube{k}$ being the unique invariant measure.
\end{corollary}

Emulating \cite[Chapter 12, Proposition~7]{HK18}, we show that for ergodic systems, the Haar measure $m^\cube{k}_Y$ that depends only on the inherent structure of the nilmanifold agrees with the dynamical measure $\mu_{k,H}$ coming from the action of $U$ on $Y$.
\begin{proposition}\label{P: identification of measures}
    Let $(Y, \CB_Y, m_Y, U)$ be an ergodic $H$-nilsystem. Then for every $k\in\N_0$, the measure $\mu_{k,H}$ on $Y^{\cube{k}}$ constructed from $\mu = m_Y$ in Definition \ref{D: cubic measures} agrees with the measure $m_Y^\cube{k}$ from Definition \ref{D: cubes of nilmanifolds}.
\end{proposition}
The ergodicity of $U$ is needed since it gives $m_Y^\cube{1} = m_Y\times m_Y$ via Corollary \ref{C: unique ergodicity of cubes}. On the other hand, we have $\mu_{1,H} = m_Y \times_{\CI(H)}m_Y$, and this is different from $m_Y^\cube{1}$ for a nonergodic action of $H$.
\begin{proof}
We follow closely the proof of \cite[Chapter 12, Proposition~7]{HK18}.
Let $Y=G/\Gamma$, and assume that $H\subseteq G$ without loss of generality.

    For $k=0$, both measures equal $\mu_Y$, and so suppose that $m_Y^\cube{k} = \mu_{k,H}$ for some $k\in\N_0$.
    Let $L^\cube{k}(G)$ be defined as in Lemma~\ref{L: fiber products}, and define
    \begin{align*}
        W_k := \frac{Q^\cube{k}(G)}{L^\cube{k}(G)\cdot Q^\cube{k}(\Gamma)}
    \end{align*}
        Let $\pi_k: Q^\cube{k}(Y)\to W_k$ be the natural quotient map. Arguing as in the proof of \cite[Chapter 12, Proposition~7]{HK18} in a way that explores the fiber product structure of $Q^\cube{k+1}(G)$ given by Lemma~\ref{L: fiber products}, we conclude that 
        \begin{align}\label{E: nil fiber product}
                    Q^\cube{k+1}(Y) &= Q^\cube{k}(Y)\times_{W_k}Q^\cube{k}(Y)\\
                    \nonumber
            &:=\{(y,y')\in Q^\cube{k}(Y)\times Q^\cube{k}(Y)\colon\; \pi_k(y)=\pi_k(y')\}
        \end{align}
        and
     \begin{align}\label{E: Haar as relatively independent joining}
        m^\cube{k+1}_Y = m^\cube{k}_Y \times_{W_k}m^\cube{k}_Y= \mu_{k,H} \times_{W_k}\mu_{k,H}.
    \end{align}

    Since $\langle h^{\cube{k}}\colon h\in H\rangle\subseteq L^\cube{k}(G)$, we have $\pi_k\circ U_h^{\cube{k}}=\pi_k$ for all $h\in H$. By the properties of relatively independent joining,
any $U^\cube{k}$-invariant $F\in L^\infty(Y^\cube{k})$ satisfies $F(y)=F(y')$ for $\mu_{k+1,H}$-a.e. $(y,y')\in Y^\cube{k+1}$. Hence any $F'\in C(W_k)$ satisfies $F'\circ \pi_k(y)=F'\circ\pi_k(y')$, implying that $\pi_k(y) = \pi_k(y')$. This and the fiber product structure \eqref{E: nil fiber product}, \eqref{E: Haar as relatively independent joining} imply that the measure
\begin{align*}
    \mu_{k+1,H} = \mu_{k,H}\times_{\CI(U^\cube{k})}\mu_{k,H}
\end{align*}
is supported on $Q^\cube{k+1}(Y)$. By Proposition~\ref{P: face transformations invariance}, the measure $\mu_{k+1,H}$ is invariant under the action $Q^\cube{k+1}(U)$. The unique ergodicity of $(Q^\cube{k+1}(Y),Q^\cube{k+1}(U))$ established in Corollary \ref{C: unique ergodicity of cubes} delivers the claim $\mu_{k+1,H} =  m_Y^\cube{k+1}$.
\end{proof}

The previous result implies that Host-Kra seminorms of step $k+1$ define norms on $k$-step $H$-nilsystems.
\begin{proposition}\label{P: nilsystems are systems of order k}
Let $(Y, \CB_Y, m_Y, U)$ be an $H$-nilsystem of step $k$. Then $\nnorm{\cdot}_{k+1, H}$ defines a norm on $L^\infty(m_Y)$. Equivalently, $\CB_Y = \CZ_{k,H}$.
\end{proposition}

\begin{proof}
We follow the proof of \cite[Chapter 12, Proposition~16]{HK18}. To show that $\nnorm{\cdot}_{k+1, H}$ is a norm, it suffices to show that $f=0$ whenever $\nnorm{f}_{k+1, H} = 0$. Assume the latter. By the Gowers-Cauchy-Schwarz inequality, 
\begin{align*}
\int \bigotimes_{\e\in\{0,1\}^{k+1}} f_{\epsilon}\,d\mu_{k+1, H}=0     
\end{align*}
 for all $f_{\e}\in L^\infty(\mu)$, where $f_{\emptyset}:=f$. By density and linearity, we have
 \begin{align*}
 \int f(y) F(y_{\ast})\,d\mu_{k+1, H}(y,y_{\ast})=0    
 \end{align*}
  for all $F$ in $2^{k+1}-1$ variables. Proposition~\ref{P: identification of measures} implies that the dynamical measure $\mu_{k+1,H}$ agrees with the Haar measure $m_Y^\cube{k+1}$ supported on the nilsystem $Q^\cube{k+1}(Y)$. By the ``only if'' part of \cite[Proposition~14, Chapter 12]{HK18}, there exists 
    a map $\Phi\colon Y^{2^{k+1}-1}\to Y$ such that 
    $y=\Phi(y_{\ast})$ for $m_Y^\cube{k+1}$-a.e. $(y,y_{\ast})\in Y^{2^{k+1}}$. 
    Hence
    \begin{align*}
            0&=\int f(y) \overline{f}\circ \Phi(y_{\ast})\,d\mu_{k+1,H}(y,y_{\ast}) = \int f(y) \overline{f}\circ \Phi(y_{\ast})\,dm^\cube{k+1}_Y(y,y_{\ast})\\
            &=\int f(y) \overline{f}(y)\,dm^\cube{k+1}_Y(y,y_{\ast})=\int |f(y)|^2\,dm_{Y}(y),        
    \end{align*}
which implies that $f=0$.
\end{proof}

\bibliography{library}
\bibliographystyle{plain}
\end{document}